\documentclass[11pt]{article}
\usepackage[utf8]{inputenc}
\usepackage{float} 
\usepackage{amsmath}
\usepackage{amssymb}
\usepackage{titletoc}
\usepackage{array}
\usepackage{bbm}
\usepackage{makeidx}
\usepackage[linktocpage=true]{hyperref}
\usepackage[capitalise]{cleveref}
\usepackage{amsthm}
\usepackage{caption}
\usepackage{subcaption}
\usepackage{color}
\usepackage{mathtools}
\usepackage{graphicx} 
\usepackage[LGR,T1]{fontenc}
\usepackage{siunitx}
\usepackage{mathrsfs}
\usepackage{scalerel}
\usepackage{forest}

\usepackage{algorithm,algpseudocode}
\usepackage{lipsum}
\usepackage[inline]{}
\usepackage{xcolor}
\usepackage[title]{appendix}
\hypersetup{
    colorlinks,
    linkcolor={blue!80!black},
    citecolor={green!50!black},
}
\usepackage{accents}
\usepackage{apptools}
\usepackage{tikz}
\usetikzlibrary{arrows, positioning, automata}
\usepackage{comment}
\usepackage[shortlabels]{enumitem}

\usepackage[mathscr]{euscript}
\DeclareSymbolFont{rsfs}{U}{rsfs}{m}{n}
\DeclareSymbolFontAlphabet{\mathscrsfs}{rsfs}

\AtAppendix{\counterwithin{lemma}{section}}

\usepackage{mathtools}

\newtheoremstyle{myexample} 
    {\topsep}                    
    {\topsep}                    
    {\rm }                   
    {}                           
    {\bf }                   
    {.}                          
    {.5em}                       
    {}  

\newtheoremstyle{myremark} 
    {\topsep}                    
    {\topsep}                    
    {\rm}                        
    {}                           
    {\bf}                        
    {.}                          
    {.5em}                       
    {}  

\theoremstyle{plain}
\renewenvironment{proof}{\noindent{\bfseries Proof.}}{\begin{flushright} \qedsymbol\end{flushright}}

\newtheorem{theorem}{Theorem}[section]
\newtheorem{proposition}{Proposition}[section]

\newtheorem{lemma}{Lemma}[section]
\newtheorem{corollary}{Corollary}[section]
\newtheorem{assumption}{Assumption}[section]

\theoremstyle{myremark} 
\newtheorem{remark}{Remark}[section]

\theoremstyle{myexample} 
\newtheorem{example}{Example}[section]

\allowdisplaybreaks
\usepackage[top=1in,bottom=1in,left=1in,right=1in]{geometry}
\usepackage[utf8]{inputenc}

\DeclareMathOperator*{\plim}{p-lim}

\def\<{\langle}
\def\>{\rangle}
\def\dd{\mathrm{d}}

\newcommand{\indep}{\perp \!\!\! \perp}
\def\Pj{\mathrm{\mathbf{P}}}

\def\top{\sT}

\def\barbLambda{\bar{\bLambda}}
\def\barbTheta{\bar{\bTheta}}

\def\barTheta{\bar{\Theta}}
\def\barblambda{\bar{\vlambda}}
\def\barbtheta{\bar{\vtheta}}

\def\bartheta{\bar{\theta}}

\def\cF{\mathcal{F}}
\def\cE{\mathcal{E}}

\def\tensorl{ \mu_{\Lambda}^{\otimes n}(\dd\vlambda)}
\def\tensort{ \mu_{\Theta}^{\otimes d}(\dd\vtheta)}
\def\tensorlb{ {\mu}_{\bar\Lambda}^{\otimes n}(\dd \bar\vlambda)}
\def\tensortb{ {\mu}_{\bar\Theta}^{\otimes d}(\dd \bar{\vtheta})}
\def\vlambda{\boldsymbol{\lambda}}
\def\vtheta{\boldsymbol{\theta}}

\def\GOE{{\rm GOE}}
\def\MMSE{{\rm MMSE}}
\def\MMSEas{{\rm MMSE}^{\mbox{\tiny\rm asym}}}
\def\MMSEsy{{\rm MMSE}^{\mbox{\tiny\rm symm}}}
\def\Infoas{{\rm I}^{\mbox{\tiny\rm asym}}}
\def\Infosy{{\rm I}^{\mbox{\tiny\rm symm}}}

\def\de{{\rm d}}

\def\pertl{\bar H_{n,\lambda}^{\scriptscriptstyle  {(pert)}}}
\def\pertt{\bar H_{n,\theta}^{\scriptscriptstyle  {(pert)}}}

\def\perttot{\bar H_{n}^{\scriptscriptstyle  {(tot)}}}

\def\bartheta{\bar{\theta}}

\def\cF{\mathcal{F}}

\def\Overlap{\mbox{\rm Overlap}}
\def\Pairoverlap{\mbox{PairOverlap}}
\def\cE{\mathcal{E}}
\def\ep{\varepsilon}

\def\Holder{H\"{o}lder}

\def\NN{\mathbb{N}}

\def\RR{\mathbb{R}}

\def\bA{\mathbf{A}}
\def\bB{\mathbf{B}}
\def\bC{\mathbf{C}}
\def\bD{\mathbf{D}}
\def\bE{\mathbf{E}}

\def\bG{\mathbf{G}}

\def\bI{\mathbf{I}}

\def\bL{\mathbf{L}}
\def\bM{\mathbf{M}}

\def\bO{\mathbf{O}}
\def\bP{\mathbf{P}}
\def\bQ{\mathbf{Q}}
\def\bR{\mathbf{R}}
\def\bS{\mathbf{S}}

\def\bU{\mathbf{U}}
\def\bV{\mathbf{V}}
\def\bW{\mathbf{W}}
\def\bX{\mathbf{X}}
\def\bY{\mathbf{Y}}
\def\bZ{\mathbf{Z}}

\def\ba{\boldsymbol{a}}

\def\be{\boldsymbol{e}}

\def\bg{\boldsymbol{g}}

\def\bv{\boldsymbol{v}}

\def\bx{\boldsymbol{x}}
\def\by{\boldsymbol{y}}
\def\bz{\boldsymbol{z}}

\def\bA{\boldsymbol{A}}
\def\bB{\boldsymbol{B}}
\def\bC{\boldsymbol{C}}
\def\bD{\boldsymbol{D}}
\def\bE{\boldsymbol{E}}

\def\bG{\boldsymbol{G}}

\def\bI{\boldsymbol{I}}

\def\bL{\boldsymbol{L}}
\def\bM{\boldsymbol{M}}

\def\bO{\boldsymbol{O}}
\def\bP{\boldsymbol{P}}
\def\bQ{\boldsymbol{Q}}
\def\bR{\boldsymbol{R}}
\def\bS{\boldsymbol{S}}

\def\bU{\boldsymbol{U}}
\def\bV{\boldsymbol{V}}
\def\bW{\boldsymbol{W}}
\def\bX{\boldsymbol{X}}
\def\bY{\boldsymbol{Y}}
\def\bZ{\boldsymbol{Z}}

\def\normal{{\mathsf{N}}}

\def\btheta{\boldsymbol{\Theta}}
\def\bttheta{\tilde{\boldsymbol{\Theta}}}
\def\bhtheta{\hat{\boldsymbol{\Theta}}}

\def\bSigma{\boldsymbol{\Sigma}}
\def\bGamma{\boldsymbol{\Gamma}}

\def\blambda{\boldsymbol{\Lambda}}

\def\bhLambda{\hat{\boldsymbol{\Lambda}}}

\def\bOmega{\boldsymbol{\Omega}}

\def\bLambda{\boldsymbol{\Lambda}}
\def\bLambdas{\boldsymbol{\Lambda}_{0}}

\def\bTheta{\boldsymbol{\Theta}}
\def\bThetas{\boldsymbol{\Theta}_{0}}

\def\bfone{{\boldsymbol 1}}

\def\hbLambda{\hat{\boldsymbol \Lambda}}
\def\hbTheta{\hat{\boldsymbol \Theta}}
\def\id{{\boldsymbol I}}

\def\op{\mbox{\tiny\rm op}}

\def\cD{{\mathcal D}}

\def\sT{{\sf T}}

\renewcommand{\P}{\mathbb{P}}
\newcommand{\E}{\mathbb{E}}
\newcommand{\R}{\mathbb{R}}

\newcommand{\eps}{\varepsilon}
\newcommand{\Var}{\operatorname{Var}}
\newcommand{\argmax}{\operatorname{argmax}}
\newcommand{\argmin}{\operatorname{argmin}}

\newcommand{\toP}{\overset{P}{\to}}
\newcommand{\Cov}{\operatorname{Cov}}

\newcommand{\diag}{\operatorname{diag}}

\newcommand{\Unif}{\operatorname{Unif}}

\newcommand{\RN}[1]{%
  \textup{\uppercase\expandafter{\romannumeral#1}}%
}

\newcommand\iidsim{\stackrel{\mathclap{iid}}{\sim}}

\newcommand{\RNum}[1]{\uppercase\expandafter{\romannumeral #1\relax}}



\makeatletter
\newcommand*{\rom}[1]{\expandafter\@slowromancap\romannumeral #1@}
\makeatother

\title{Fundamental Limits of Low-Rank Matrix Estimation \\
with Diverging Aspect Ratios}
\author{
	Andrea Montanari\footnotemark[2]
	\thanks{Department of Electrical Engineering, Stanford University; School of Mathematics, 
	Institute for Advanced Studies, Princeton} 
	\and 
	Yuchen Wu\thanks{Department of Statistics, Stanford University}
}
\date{}
\begin{document}
\maketitle

\begin{abstract}
We consider the problem of estimating the factors of a low-rank $n\times d$ matrix, 
when this is corrupted by additive Gaussian noise. A special example of our setting corresponds to 
clustering mixtures of Gaussians with equal (known) covariances. 
Simple spectral methods do not take into account the distribution of the entries of these
factors and are therefore often suboptimal. 
Here, we characterize the asymptotics of the minimum estimation error under
the assumption that the distribution of the entries  is known to the statistician.

Our results apply to the high-dimensional regime $n,d\to\infty$ and  $d/n\to\infty$
(or $d/n\to 0$) and  generalize earlier work that focused on the proportional asymptotics 
$n,d\to\infty$,  $d/n\to\delta\in (0,\infty)$. 
We outline an interesting signal strength regime in which $d/n\to \infty$ and 
partial recovery is possible for the left singular vectors while impossible 
for the right singular vectors. 
  
We illustrate the general theory by deriving consequences for Gaussian mixture clustering 
and carrying out a numerical study on genomics data.
\end{abstract}
\tableofcontents

\section{Introduction}\label{sec:Intro}

The problem of low-rank matrix estimation has received enormous attention 
within high-dimensional statistics, probability theory, random matrix theory,
and computer science. The \emph{spiked model} introduced by Johnstone 
\cite{johnstone2001distribution} (also known as `signal-plus-noise' model
or `deformed ensemble')  plays a central role in theoretical analysis, and
has inspired a number of statistical insights
 \cite{hoyle2004principal,baik2005phase,baik2006eigenvalues,johnstone2006high,amini2008high,johnstone2009sparse,ma2013sparse,deshpande2014sparse,deshpande2014information,benaych2011eigenvalues, benaych2012singular,johnstone2020testing,macris2020all}. 

In the spiked model, we observe a matrix $\bA\in\RR^{n\times d}$ which is 
given by the sum of a low-rank signal and random noise
\begin{align}\label{model:general}
	\bA = s_n \bLambda \bTheta^{\sT} + \bZ \, .
\end{align}
Here, $\bLambda\in \RR^{n\times r}$, $\bTheta\in \RR^{d\times r}$ are factors 
which we would like to estimate and $\bZ$ is a noise matrix with i.i.d. 
entries $Z_{ij}\sim\normal(0,1)$.
Finally,  $s_n \in \RR_{>0}$ is a signal-to-noise ratio which 
we also assume to be known. We will consider high-dimensional asymptotics
whereby $d,n\to\infty$ with $r$ fixed.
In what follows, we will denote by $\ba_1,\dots,\ba_n$ the rows of $\bA$.

Low-rank matrix estimation is ubiquitous in high-dimensional statistics, with special cases
including sparse PCA \cite{johnstone2009sparse,journee2010generalized,berthet2013optimal,deshpande2014sparse,deshpande2014information},
community detection \cite{abbe2017community,deshpande2017asymptotic}, submatrix localization 
\cite{kolar2011minimax,hajek2017submatrix}, and Gaussian mixture clustering 
\cite{moitra2010settling,regev2017learning,ndaoud2022sharp,cai2019chime}.
We illustrate the broad applicability of model \eqref{model:general} 
using two examples:
\begin{example}[Sparse PCA]\label{example:SparsePCA}
In a simple model for sparse PCA \cite{johnstone2009sparse}, we observe
vectors 
$$\ba_1, \dots,\ba_n  \sim_{iid}  \normal(0,\bSigma),$$ where 
$\bSigma= s_n^2 \bTheta\bTheta^{\sT}+\id_d$ with $\bTheta\in\RR^d$ a sparse vector that we would like to
estimate. 

This is the special case of model \eqref{model:general}, if we let $r=1$ and 
$\bLambda \sim \normal(\mathbf{0},\id_n)$.        
\end{example}  

\begin{example}[Mixture of Gaussians with known covariance]
In a mixture of Gaussian model, we observe vectors 
$\bar\ba_1,\dots,\bar\ba_n\sim_{iid} p\, \normal(\btheta_1,\bSigma_1)+(1-p)\, \normal(\btheta_2,\bSigma_2)$.
If the covariances coincide and are known: $\bSigma_1=\bSigma_2=\bSigma$,
and the population mean $\bar\btheta :=p\btheta_1+(1-p)\btheta_2$ can be estimated accurately, then 
we can define $\ba_i = \bSigma^{-1/2}(\bar\ba_i-\bar\btheta)$.
Hence the model is equivalent to observing
$\ba_1,\dots,\ba_n\sim_{iid} p\, \normal((1-p)\bTheta,\id_d)+(1-p)\, \normal(-p\bTheta,\id_d)$,
with $\bTheta :=\bSigma^{-1/2}(\btheta_1-\btheta_2)$.

This is another special case of model \eqref{model:general}, with $r=1$ and
$(\Lambda_i)_{i\le n}\sim_{iid} p\delta_{(1-p)}+(1-p)\delta_{-p}$. Estimating $\bLambda$
amounts to estimating the cluster labels.
\end{example}

As demonstrated in these examples, it is often the case that the latent
factors $\bLambda$, $\bTheta$ have additional structure. In the first example 
$\bTheta\in\RR^d$ is a sparse vector, while in the second one $\bLambda$ is a vector with i.i.d.
 entries distributed according to a two-point mixture.
 In this paper, we assume 
a stylized model whereby the rows $\bLambda$, $\bTheta$  are mutually independent (and independent of $\bZ$)
with  $(\bLambda_i)_{i \leq n} \iidsim \mu_{\Lambda}$ and $(\bTheta_j)_{j \leq d} \iidsim \mu_{\Theta}$.
Here, $\mu_{\Lambda}$ and $\mu_{\Theta}$ are fixed probability distributions on $\RR^r$.
We will be concerned with the problem of determining the Bayes optimal estimation error
under the idealized  setting in which the distributions
$\mu_{\Lambda}, \mu_{\Theta}$ and the signal-to-noise ratio $s_n$ are known to the statistician.

This setting was considered several times in recent past, see e.g., \cite{miolane2017fundamental,lelarge2019fundamental,
 montanari2021estimation,banks2018information} and Section \ref{sec:Related} for further
 references. Closely related 
 to our work  are the results of \cite{miolane2017fundamental,lelarge2019fundamental}, 
 who determined the precise asymptotics of mutual information and (certain)
 estimation error metrics when $n,d\to\infty$,
 in the proportional regime $d/n\to\delta\in (0,\infty)$.
 
 Our goal is to move beyond the proportional asymptotics 
 and consider the cases $d/n\to \infty$ and $d/n\to 0$. We believe an analysis
 of these regimes is essential for at least two reasons.
 \begin{itemize}
 \item \emph{First,} in modern applications, it is often the case that the dimension $d$ is much 
 larger than the sample size $n$. For instance, in genomics studies, $n$ is
 the number of sequences (subjects) and is often between one hundred and a few thousands, while 
 $d$ is the number of gene variations under study and can be of the order of hundreds of thousands.
 At the other end of the spectrum, $n$ can be  significantly 
 larger than $d$ yet classical low-dimensional theory is not accurate. 
\item \emph{Second,} in practice, the statistician does not observe a sequence of data matrices $\{\bA_n\}$
 with increasing $\{(n,d)\}$, but a single pair $(n,d)$. It is a natural reflex to apply proportional 
 asymptotics results with the nominal aspect ratio $\hat{\delta} := d/n$.
 However, this choice does not have a rigorous justification. This is not just a mathematicians’ quibble. 
 Imagine, for a moment, what would happen if  the asymptotics for 
 $d/n^{\alpha}\to \delta_{\alpha}\in (0,\infty)$ (with $\alpha>1$, say $\alpha =3/2$) 
 were different from the 
 one for $d/n\to\delta$ 
 followed by $\delta\to\infty$.
 Given data with $n=100$ and $d=1000$, we would obtain different theoretical prediction
 depending on whether we regarded it as an element of the sequence $d(n)=10n$
 (i.e., proportional asymptotics with $\delta=10$), or $d(n)=n^{3/2}$
 (i.e., $\alpha=3/2$, $\delta_{\alpha}=1$).
   We will see that this is not
 the case.
 \end{itemize}
Before proceeding, we point out that recent work addresses the assumption
that  $s_n, \mu_{\Lambda}, \mu_{\Theta}$ are known. Namely, \cite{zhong2022empirical}
uses empirical Bayes  techniques to show that the Bayes error can be achieved even if
$s_n, \mu_{\Lambda}, \mu_{\Theta}$ are unknown, under the proportional asymptotics
(and in certain settings). We expect that our results should translate into similar 
guarantees for empirical Bayes methods in the case of diverging aspect ratios.

We next summarize our results, focusing to be definite on the case $d/n\to\infty$:
analogous statements for $d/n\to 0$ follow by interchanging $n$ and $d$,
as well as $\blambda$ and $\bTheta$.

We show that depending on the scaling of the signal-to-noise ratio $s_n$,
there are two interesting regimes that control the behavior of the estimation problem. 
\begin{description}
\item[Strong signal regime.] This is obtained for $s_n \asymp n^{-1/2}$, and is relatively easy
to characterize analytically. Under this scaling, $\bLambda$ can be estimated 
 consistently (possibly up to a rotation), while the minimum normalized
 estimation error of $\bTheta$ remains bounded away from $0$. We characterize the limiting 
 error of estimating $\bTheta$.
\item[Weak signal regime: Estimation of $\bTheta$.] This regime corresponds to $s_n \asymp (nd)^{-1/4}$,
and most of our technical work is devoted to its analysis. 
We prove that, in this regime, non-trivial estimation of $\bTheta$ is impossible:
any estimator has asymptotically the same risk as the the null
estimator $\hbTheta_0 = \E[\bTheta] $. 
\item[Weak signal regime: Mutual information.] On the other hand, still in 
taking $s_n \asymp (nd)^{-1/4}$,
estimation 
of $\bLambda$ is non-trivial. As a first result in this direction, we characterize the asymptotic 
mutual information 
$$\lim_{n,d\to\infty}\frac{1}{n}I(\bA;\bLambda),$$ 
and show that this is
non-vanishing. 
Further, this mutual information is asymptotically the same as for a
symmetric observation model in which instead of $\bA\in\RR^{n\times d}$, we observe
 $\bY \in \RR^{n \times n}$ given by
\begin{align}\label{model:spike0}
	\bY = \frac{q_{\Theta}}{n}\bLambda\bLambda^{\sT} + \bW,\;\;\;\; \bW\sim\GOE(n)\, .
\end{align}
(Here, $q_{\Theta}: =r^{-1}\int \|\btheta\|^2\, \mu_{\Theta}(\dd\btheta)$, we take without loss of 
generality $s_n= (nd)^{-1/4}$, and
$\GOE(n)$ denotes the distribution of a symmetric matrix with independent entries 
on or above the diagonal $(W_{ij})_{i\le j}$ such that $W_{ii}\sim\normal(0,2/n)$
and $W_{ij}\sim\normal(0,1/n)$ for $i<j$.)
\item[Weak signal regime: Estimation error.] We then proceed to study the 
asymptotics of the Bayes optimal matrix mean square error:
\begin{align}
\MMSE(\mu_{\Lambda},\mu_{\Theta}):= \lim_{n,d\to\infty}\frac{1}{n^2}\E\big\{\big\|
\bLambda\bLambda^{\sT}-\E[\bLambda\bLambda^{\sT}|\bA]
\big\|_F^2\big\}\, .
\end{align}
We characterize this limit in two regimes: $n\ll d\ll n^{6/5}$ or $n^3\ll d$,
and in certain cases for all $n\ll d$ (here $\ll$ hides logarithmic factors.)
In these cases we prove equivalence with model \eqref{model:spike0}.
\end{description}
We believe that the conditions $n\ll d\ll n^{6/5}$ or $n^3\ll d$ are artifacts
of the proof. Indeed, the conclusion holds for all $n\ll d$ under a natural
(unproven) continuity assumption.  We leave it to future work to cover the intermediate range
$n^{6/5}\lesssim d\lesssim n^{3}$.

Both the limiting mutual information and the asymptotic estimation error
$\MMSE(\mu_{\Lambda},\mu_{\Theta})$ are given by explicit expressions known as `replica 
symmetric formulas,' because they are correctly predicted by the replica method in spin glass theory 
\cite{mezard1987spin,montanari2022short}. However, the asymptotic equivalence with the
symmetric model \eqref{model:spike0} is insightful in itself (i.e., independently of the fact
that we can give explicit formulas for the asymptotic error and mutual information):
\begin{enumerate}
\item The asymptotic equivalence between model \eqref{model:general} and model \eqref{model:spike0} 
 implies that the optimal estimation depends on $\mu_{\Theta}$
only though its second moment. In other words, no substantial improvement is achieved in the
regime covered by this equivalence exploiting the knowledge of the distribution of $\bTheta$.
\item The symmetric matrix $\bY$ is closely related to the Gram matrix $\bY'=(\bA\bA^{\sT}-d \id_n)/\sqrt{nd}$,
an observation that is confirmed by inspecting the proof. This implies that there is no 
substantial loss of accuracy in estimating $\bLambda$ uniquely on the basis of $\bY'$.
This yields a substantial reduction in complexity for $n\ll d$.
\end{enumerate}
We warn the reader that these conclusions do not apply in settings that are not 
captured here. For instance, in sparse PCA, cf. Example \ref{example:SparsePCA},
one might be interested in cases in which the number of non-zeros of the principal
component $\bTheta$ is sub-linear in the dimension $d$. This case cannot be modeled
as above, and requires instead to consider $\mu_{\Theta}$ dependent on $n,d$.

The rest of the paper is organized as follows. We briefly review related work in Section
\ref{sec:Related}.
We then present our results for the strong signal regime in Section \ref{sec:strong-signal} 
and the weak signal regime in Section \ref{sec:weak-signal}.
We finally apply the general theory to the case of Gaussian mixture models in \cref{sec:GMM} 
and compare it with  analysis
on real data in Section \ref{sec:real-data}.

\subsection{Notations and conventions}\label{sec:notation}


For $k \in \NN$, we define the set $[k] := \{1,2,\cdots, k\}$. 
We typically use lower case non-bold letters for scalars ($m$, $n$, $j$), 
and bold for vectors  and matrices ($\bx$, $\by$, $\bz$, $\bA$, $\bB$, $\bC$).
We use  $\|\bv\|$ to denote the Euclidean norm of a vector $\bv$, and $\|\bM\|_F$ to denote the Frobenius norm of a matrix $\bM$.
 For $\{c_n\}_{n \in \NN_+}, \{d_n\}_{n \in \NN_+} \subseteq \RR_+$, we say $c_n \gg d_n$ 
 if and only if $c_n / d_n \rightarrow \infty$, and for $\{e_n\}_{n \in \NN_+} \subseteq \RR$, we say $e_n = o_n(1)$ if and only if $e_n \to 0$ as $n \to \infty$. We denote by $\plim$ convergence in probability. 
%

For $k \in \NN_+$, we denote by $S_k^+$ the set of positive semi-definite matrices in
 $\RR^{k \times k}$, and denote by $\mathcal{O}(k)$ the set of orthogonal matrices in 
 $\RR^{k \times k}$. For $\bM \in S_k^+$, we let $\bM^{1/2} \in S_k^+$ be any positive
  semi-definite matrix such that $\bM = \bM^{1/2}\bM^{1/2}$. 
  

 We denote the $i$-th row of the factors $\bLambda$, $\bTheta$ by
 $\bLambda_i$ and $\bTheta_i$, respectively. We use $\bLambdas$ and 
 $\bThetas$ to 
 represent length-$r$ random vectors drawn from probability distributions $\mu_{\Lambda}$ and
  $\mu_{\Theta}$. We sometimes need to write the posterior distribution of 
  $(\bLambda, \bTheta)$ given particular observations. In this case, we use the lower case 
  letters $\vtheta, \vlambda, \vlambda_i, \vtheta_i$ to represent variables corresponding to 
  $(\bTheta, \bLambda, \bLambda_i, \bTheta_i)$ in the posterior distribution. 
 
Throughout the paper, we use capital letter $C$ 
 to represent various numerical constants.

\section{Further related work}
\label{sec:Related}

As mentioned in the introduction, most earlier work deriving sharp asymptotics results
focuses on the proportional regime $n\asymp d$, $s_n= n^{-1/2}$.
In particular, 
 \cite{lesieur2015mmse} first obtained the limiting expression for Bayesian mean square error
 using  non-rigorous tools from statistical mechanics. The conjectured expression was 
 rigorously justified for special distributions $\mu_{\Lambda}$, $\mu_{\Theta}$
  in \cite{deshpande2014information,deshpande2017asymptotic}. 
 However, the proof technique of \cite{deshpande2014information,deshpande2017asymptotic}
  relies on the fact that approximate message passing (AMP) algorithm 
 achieves Bayes optimality and does not apply to the general case.
 
 Several groups developed rigorous approaches to prove the asymptotic formulas in increasing degrees of
 generality: spatial coupling \cite{dia2016mutual}; the cavity method 
   \cite{lelarge2019fundamental, miolane2017fundamental,el2018estimation}; adaptive interpolation
   \cite{barbier2019adaptive}; partial differential equation
   techniques \cite{dominguez2022mutual}.

A different line of research uses the  second moment method to derive upper and 
lower bounds on the information-theoretic thresholds
\cite{banks2018information,perry2016optimality,perry2020statistical}
for partial or exact recovery. 
This approach typically yields non-asymptotic bounds, under a broader class of settings
but the results only determine such thresholds up to undetermined multiplicative constants. 
In contrast, here we attempt to obtain a characterization that is
accurate up to $(1+o_n(1))$
factors.

From  a computational viewpoint, AMP-based algorithms 
can be shown to achieve the Bayesian error for a large region of parameters
\cite{bayati2011dynamics,montanari2021estimation}. 
       One appealing fact about the AMP is that its high-dimensional behavior can be sharply 
       characterized by  \emph{state evolution}. 

Minimax guarantees were obtained by a number of groups for
special cases of the low-rank model \eqref{model:general}. 
Sparse PCA and Gaussian mixtures are arguably the most studied models
in the literature, see e.g.,
 \cite{royer2017adaptive,giraud2019partial,ndaoud2022sharp,berthet2013computational,cai2017computational} .
These works often yield characterizations that hold up to 
usually a constant or logarithmic multiplicative gap.

Gaussian mixture
  models (GMM) provide a useful context for evaluating and comparing various clustering algorithms. 
  We will use it here to illustrate the applicability of our general results. 
   The goal can be either estimating the centers
  \cite{dasgupta1999learning, dasgupta2007probabilistic, klusowski2016statistical,
   mixon2017clustering,regev2017learning}, or recovering the underlying cluster
    assignments \cite{vempala2004spectral, achlioptas2005spectral, brubaker2008isotropic,
    kumar2010clustering, awasthi2012improved, fei2018hidden}. 
    As we will see, in the high-dimensional weak signal regime, the cluster
    centers cannot be estimated, but the cluster assignments can be estimated with non-trivial accuracy. 

Several algorithms were studied in detail for clustering under GMM, including
    semi-definite programming
     (SDP)  \cite{peng2007approximating, awasthi2015relax,fei2018hidden,iguchi2015tightness,li2020birds},
      iterative algorithms with spectral initialization 
      \cite{achlioptas2005spectral,vempala2004spectral,kumar2010clustering,awasthi2012improved,lu2016statistical}, 
       the  method of moments 
      \cite{pearson1894contributions,feldman2006pac,kalai2010efficiently,moitra2010settling,hsu2013learning,belkin2015polynomial,hardt2015tight},
      and EM-based algorithms  \cite{dasgupta2007probabilistic,balakrishnan2017statistical,jin2016local,cai2019chime}.

Finally, in concurrent work, Donoho and Feldman recently
characterized the accuracy of eigenvalue
shrinkage methods in the spiked  model with diverging aspect ratio 
\cite{feldman2021spiked,donoho2022optimal}.

\section{Strong signal regime}\label{sec:strong-signal}
We first consider the strong signal regime in which we set $s_n = 1 / \sqrt{n}$,
and therefore we have
\begin{align}\label{model:strong-signal}
	\bA = \frac{1}{\sqrt{n}}\bLambda \bTheta^\sT + \bZ \in \mathbb{R}^{n \times d}. 
\end{align}
%
%
We define $\bQ_{\Lambda} := \E_{\bLambdas \sim \mu_{\Lambda}}[\bLambdas\bLambdas^{\sT}]
\in S_r^+$ 
and $\bQ_{\Theta} := \E_{\bThetas \sim \mu_{\Theta}}[\bThetas\bThetas^{\sT}] \in S_r^+$. Before we proceed,  we establish the following conventions for the 
distributions $\mu_{\Lambda}, \mu_{\Theta}$.
\begin{remark}\label{assumption:second-moment}

Without loss of generality we can and will assume that both $\bQ_{\Lambda}$ and $\bQ_{\Theta}$ 
are invertible.  Furthermore, we can assume that $\bQ_{\Theta} = q_{\Theta} \id_r$ for some
$q_{\Theta} \in \RR_{>0}$.
\end{remark} 
 More precisely, we next show that 
 ---given arbitrary probability distributions $(\mu_{\Theta}, \mu_{\Lambda})$---
the conditions of Remark \ref{assumption:second-moment} 
can always be satisfied by a reparameterization. 

 For $\bLambdas \sim \mu_{\Lambda}$ and $\bThetas \sim \mu_{\Theta}$, if either
 $\bLambdas \overset{a.s.}{=} 0 $ or $\bThetas \overset{a.s.}{=} 0$ then 
  estimation  becomes trivial. We can therefore assume that this is not the case.
  By eigendecomposition of $\bQ_{\Lambda}$ and $\bQ_{\Theta}$, there exist  $ 0 < k_1, k_2 \leq r$ 
  and non-random matrices $\bM_1 \in \RR^{r\times k_1}$, $\bM_2 \in \RR^{r\times k_2}$ with 
  full column ranks such that
  $\bLambdas = \bM_1\bLambdas'$, $\bThetas = \bM_2 \bThetas'$,
  and $\E[\bLambdas'(\bLambdas')^{\top}] = \id_{k_1}$, $\E[\bThetas'(\bThetas')^{\top}]=\id_{k_2}$.

 Assume $\bM_1^{\sT} \bM_2$ has rank $k_3\le \min(k_1,k_2)$, and let 
 $\bM_1^{\sT} \bM_2 = \bU\bS\bV^{\sT}$ be its singular value decomposition (SVD) with $\bU\in\RR^{k_1\times k_3}$,
 $\bV\in\RR^{k_2\times k_3}$ having orthonormal columns. We then
 set $\bar\bLambda\in\RR^{n\times k_3}$ a matrix with i.i.d. rows that are copies
 of $\bS\bU^{\sT}\bLambdas'$ and $\bar\bTheta\in\RR^{d\times k_3}$ a matrix with i.i.d. 
 rows that are copies
 of $\bV^{\sT}\bThetas'$. We can then write $\bLambda\bTheta^{\sT} =  \bar\bLambda\bar\bTheta^{\sT}$
 and the latter satisfies the conditions of Remark \ref{assumption:second-moment}.

Note that this argument shows that we could assume $q_{\Theta}=1$ as well,
but it is convenient to keep this as a free parameter.

\subsection{Estimation of $\bLambda$}
We first consider  estimation of $\bLambda$. We will show that
 a simple spectral estimator provides a consistent estimate up to a rotation in the 
$r$-dimensional Euclidean space.
Consistency in terms of vector mean
 square error is not guaranteed due to potential non-identifiability issues. 
 We propose sufficient conditions on $(\mu_{\Lambda}, \mu_{\Theta})$,
 that imply consistency in terms of vector mean square error as well.

 Denote by  $\hat{\bLambda}_s \in \RR^{n \times r}$ the matrix whose columns are the 
 top $r$ eigenvectors of $\bA \bA^{\sT}$, normalized so that
 $\hat{\bLambda}_s^{\sT}\hat{\bLambda}_s / n = \id_r$.
Denote by $\Pj, \hat\Pj\in \mathcal{O}(n)$ the projection matrices onto the column spaces of 
$\bLambda$ and $\hat{\bLambda}_s$, respectively. We use the following distance between the two subspaces as estimation loss
\begin{align}\label{eq:loss-function}
	L^{\sin}(\hat{\blambda}_s, \blambda) := \|\Pj(\id - \hat\Pj)\|_{\op} = \|\hat\Pj(\id - \Pj)\|_{\op} 
	= \sin \alpha(\hat{\bLambda}_s, \bLambda),
\end{align}  
where $\alpha(\hat{\bLambda}_s, \bLambda)$ is the principal angle between the two column spaces. 
\begin{theorem}\label{thm:strong-signal-lambda-consistent}
Assume $\mu_{\Lambda}$, $\mu_{\Theta}$ have finite non-singular second moments 
$\bQ_{\Lambda},\bQ_{\Theta}$ with $\bQ_{\Theta}=q_{\Theta}\id_r$ (with no loss of generality
per Remark~\ref{assumption:second-moment}).
If
$n, d \rightarrow \infty$  with $d / n \rightarrow \infty$, then under the model 
of Eq.~\eqref{model:strong-signal}:
\begin{enumerate}
	\item $L^{\sin}(\hat{\bLambda}_s, \bLambda) \toP 0$. 
	\item If we further assume that for some $\eps > 0$ we have $\E_{\bLambdas \sim \mu_{\Lambda}}[\|\bLambdas\|^{4 + \eps}] < \infty$, then there exists 
	an estimator $\hat{\bL}: \bA\mapsto \hat{\bL}(\bA) \in \RR^{n \times n}$, such that
	 $\E\big[ \|\hat\bL(\bA) - \bLambda \bLambda^{\sT}\|_F^2 \big] / n^2 \rightarrow 0$ as 
	 $n,d \rightarrow \infty$. 
	\item Let $\bLambdas \sim \mu_{\Lambda}$. If we further assume that there does not exist 
	$\bOmega \in \mathcal{O}(r)$, such that $\bOmega \neq \id_r$ and $\bOmega \bLambdas \overset{d}{=} \bLambdas$,
	 then there exists $\hat{\bLambda} :\bA\mapsto \hat{\bLambda}(\bA) \in \RR^{n \times r}$, 
	 such that $\E\big[\|\hat{\bLambda}(\bA) - \bLambda\|_F^2 \big] / n \to 0$ as $n,d \to \infty$. 
\end{enumerate}
\end{theorem}
We delay the proof of \cref{thm:strong-signal-lambda-consistent} to Appendix \ref{proof-of-thm:strong-signal-lambda-consistent}.
\subsection{Estimation of $\bTheta$}
Next, we turn to the estimation of $\bTheta$.   
According to Theorem \ref{thm:strong-signal-lambda-consistent}, $\bLambda$ can be estimated consistently 
under identifiability conditions. Therefore,
a reasonable first step is to study the case in which $\bLambda$ is given.
This yields a lower bound on the Bayesian error of the original problem. 
We will see that this lower bound
can be achieved asymptotically even if $\bLambda$ must be estimated.

We can explicitly write the conditional distribution of $\bTheta$ given $(\bLambda, \bA)$. 
Using the Gaussian density formula, we see that for all $j \in [d]$, the posterior distribution of 
$\bTheta_j$ is
\begin{align}
	p(\dd\vtheta_j | \bLambda, \bA) \propto \exp\left( -\frac{1}{2n} \sum\limits_{i = 1}^n \langle \bLambda_i, \vtheta_j \rangle^2 + \frac{1}{\sqrt{n}}\sum\limits_{i = 1}^n  A_{ij} \langle \bLambda_i, \vtheta_j\rangle \right) \mu_{\Theta}(\dd\vtheta_j).\label{eq:5}
\end{align}
\cref{eq:5} leads to the following asymptotic lower bound: 
\begin{theorem}\label{thm:strong-signal-theta-lower-bound}
Consider the strong signal model of Eq.~\eqref{model:strong-signal},
assuming, without loss of generality, the setting of Remark \ref{assumption:second-moment}.
We let $n, d \rightarrow \infty$ simultaneously with $d / n \rightarrow \infty$, then for any 
estimator $ \hat{\btheta}:\bA\mapsto \hat{\btheta}(\bA) \in \RR^{d \times r}$, we have
	\begin{align}\label{eq:6}
		\liminf\limits_{n,d \rightarrow \infty}\frac{1}{d}\E\left[ \| \hat{\btheta}(\bA) - \btheta \|_F^2\right] \geq rq_{\Theta} - \E\left[ \big\|\E[\bThetas | \bQ_{\Lambda}^{1/2} \bThetas + \bG] \big\|^2 \right],
	\end{align}
	where $\bG \sim \normal(0,\id_r)$, $\bThetas \sim \mu_{\Theta}$ are mutually independent. Notice that the right hand side of \cref{eq:6} is independent of $(n,d)$.
	 
	If we further assume $\E[\|\bThetas\|^4] < \infty$, then for any 
	$\hat{\bM}:\bA\mapsto\hat{\bM}(\bA) \in \RR^{d \times d}$, we have
	\begin{align}\label{eq:7}
		\liminf\limits_{n,d \rightarrow \infty} \frac{1}{d^2} \E\left[ \|\hat{\bM}(\bA) - \bTheta \bTheta^{\top} \|_F^2 \right] \geq rq_{\Theta}^2 - \Big\|\E\left[ \E[\bThetas | \bQ_{\Lambda}^{1/2} \bThetas + \bG] \E[\bThetas | \bQ_{\Lambda}^{1/2} \bThetas + \bG]^\sT\right]\Big\|_F^2.
	\end{align}
\end{theorem}
We postpone the proof of \cref{thm:strong-signal-theta-lower-bound} to Appendix 
\ref{proof-of-thm:strong-signal-theta-lower-bound}. Next, we show that the lower bound proposed in  
\cref{thm:strong-signal-theta-lower-bound} can be achieved under identifiability
 conditions. 
\begin{theorem}\label{thm:strong-signal-theta-achieve-lower-bound}
Under the conditions of Theorem \ref{thm:strong-signal-lambda-consistent}, claim 3, 
there exist estimators
$ \hat{\btheta}:\bA\mapsto \hat{\btheta}(\bA)$ and 
$\hat{\bM}:\bA\mapsto\hat{\bM}(\bA)$, such that
	\begin{align*}
		& \lim\limits_{n,d \rightarrow \infty}\frac{1}{d} \E\Big[ \|\hat{\bTheta}(\bA) - \bTheta\|_F^2  \Big] = rq_{\Theta} - \E\left[ \big\|\E[\bThetas | \bQ_{\Lambda}^{1/2} \bThetas + \bG] \big\|^2 \right], \\
		& \lim\limits_{n,d \rightarrow \infty}  \frac{1}{d^2} \E\Big[\|\hat{\bM}(\bA) - \bTheta \bTheta^{\top}\|_F^2 \Big] =  rq_{\Theta}^2 - \Big\|\E\left[ \E[\bThetas | \bQ_{\Lambda}^{1/2} \bThetas + \bG] \E[\bThetas | \bQ_{\Lambda}^{1/2} \bThetas + \bG]^\sT\right]\Big\|_F^2.
	\end{align*}
\end{theorem}
We defer the proof of \cref{thm:strong-signal-theta-achieve-lower-bound} to Appendix \ref{proof:thm:strong-signal-theta-achieve-lower-bound}. \cref{thm:strong-signal-theta-lower-bound,thm:strong-signal-theta-achieve-lower-bound} together complete the analysis for the estimation of $\bTheta$ in the strong signal regime. 

\section{Weak signal regime}\label{sec:weak-signal}
In this section, we consider the weak signal regime where $s_n = 1 / \sqrt[4]{nd}$. Thus the model of 
interest is
\begin{align}\label{model:weak-signal}
	\bA = \frac{1}{\sqrt[4]{nd}}\blambda \btheta^\sT + \bZ \in \mathbb{R}^{n \times d}\, .
\end{align}
For convenience, we define $r_n := \sqrt[4]{d / n}$. By assumption,
we see that $r_n \to \infty$ as $n,d \to \infty$.

\subsection{Background: the symmetric spiked model}

As mentioned in the introduction,  our main technical result is that,
in the weak signal regime, estimation under model \eqref{model:weak-signal} is equivalent to estimation 
under a symmetric spiked model. 
Under this model we observe $\bY \in \RR^{n \times n}$ given by
\begin{align}\label{model:symmetric}
	\bY = \frac{q_{\Theta}}{n}\bLambda\bLambda^{\top} + \bW,
\end{align}
where
$\bW \overset{d}{=}  \GOE(n)$, and $\bLambda_i \iidsim \mu_{\Lambda}$, independent of each other. 
 We view $q_{\Theta}>0$ as a signal-to-noise ratio parameter.

We denote the Bayesian MMSE of model \eqref{model:symmetric} by
\begin{align}
	\MMSEsy_n(\mu_{\Lambda}; q_{\Theta}) := \min\limits_{\hat{\bM}(\,\cdot\,)} \frac{1}{n^2} 
	\E\left[ \left\| \hat{\bM}(\bY)- \bLambda\bLambda^\sT  \right\|_{F}^2 \right]. \label{eq:MMSEY}
\end{align}
Note that the Bayesian MMSE is achieved by the posterior expectation
$\hat{\bM}(\bY)=\E[ \bLambda\bLambda^\sT|\bY]$. We also define the normalized mutual information
\begin{align}
\Infosy_n(\mu_{\Lambda};q_{\Theta}) &:= \frac{1}{n}\E\log \frac{\de \P_{\bLambda,\bY}}{\de (\P_{\bLambda}
\times\P_{\bY})}
  (\bLambda,\bY)\, .\label{eq:Info-Symmetric-Def}
  \end{align}
A significant amount of rigorous information is available
about this model. For $s > 0$ and $\bQ \in S_r^+$, we define the 
free energy functional $\cF(s, \bQ)$ and its maximizer  $\bQ^{\ast}(s)\in  S_r^+$
via
\begin{align}
	& \cF(s, \bQ) := -\frac{s}{4}\|\bQ\|_F^2 + \E\left\{\log \left( \int \exp(\sqrt{s}\bz^{\top} \bQ^{1/2}\vlambda + s\vlambda^{\top} \bQ\bLambdas - \frac{s}{2}\vlambda^{\top} \bQ\vlambda)\mu_{\Lambda}(\dd\vlambda) \right)\right\}.\label{eq:FsQ}\\
	& \bQ^{\ast}(s) \in \argmax_{\bQ \in S_r^+} \cF(s, \bQ). \label{eq:Qstar}
\end{align}
In the above expression,  expectation is taken over $\bLambdas \sim \mu_{\Lambda}$ and  
$\bz \sim \normal(0, \id_r)$
 independent of each other. These functionals are directly related to the 
 mutual information and the Bayes MMSE of model 
 \eqref{model:symmetric}, as stated below.
\begin{theorem}[\cite{lelarge2019fundamental}, Corollary 42, Proposition 43]\label{prop:symmetric-spike}
There exists a deterministic countable set $\cD\subseteq\RR_{\ge 0}$ such that
\begin{align*}
	q_{\Theta}\in \RR_{\ge 0}
	\;\;\Rightarrow\;\; &
\lim_{n\to\infty}	\Infosy_n(\mu_{\Lambda};q_{\Theta}) =  \frac{1}{4}q_{\Theta}^2\|\E_{\bLambdas \sim \mu_{\Lambda}}[\bLambdas \bLambdas^{\top}]\|_F^2
 -  \sup_{\bQ\in S_r^+} \cF(q_{\Theta}^2, \bQ)\, ,\\
	q_{\Theta}\in \RR_{\ge 0}\setminus \cD
	\;\;\Rightarrow\;\;&
	\lim_{n\rightarrow \infty}\MMSEsy_n(\mu_{\Lambda}; q_{\Theta}) = \|\E_{\bLambdas \sim \mu_{\Lambda}}[\bLambdas \bLambdas^{\top}]\|_F^2 - \|\bQ^{\ast}(s)\|_F^2.
\end{align*}
\end{theorem}
 
\subsection{Estimation of  $\bTheta$}\label{sec:weak-regime-theta}

We first consider  estimation of $\bTheta$. We claim that in this case no estimator outperforms a naive one. 
\begin{theorem}\label{thm:weak-signal-theta}
Consider the weak signal model of \cref{model:weak-signal}, assuming, without loss of generality, the 
setting of Remark \ref{assumption:second-moment}. Let $n, d \rightarrow \infty$ simultaneously with $d / n \rightarrow \infty$.
 Then for any estimator $\bhtheta:\bA\mapsto\bhtheta(\bA) \in \mathbb{R}^{d \times r}$, we have 
	\begin{align*}
		 \liminf\limits_{n,d \rightarrow \infty}\frac{1}{d} \E[ \|\bhtheta(\bA) - \btheta\|_F^2] \geq rq_{\Theta} - \|\E_{\bThetas \sim \mu_{\Theta}}[\bThetas]\|^2. 
	\end{align*}
	If we further assume $\mu_{\Theta}$ has bounded fourth moment, then for any 
	$\hat\bM :\bA\mapsto\hat{\bM}(\bA)\in \RR^{d \times d}$, we have
	\begin{align*}
		 \liminf\limits_{n,d \rightarrow \infty}\frac{1}{d^2} \E[ \|\hat\bM(\bA)-\bTheta \bTheta^{\top} \|_F^2 ] \geq rq_{\Theta}^2 - \|\E_{\bThetas \sim \mu_{\Theta}}[\bThetas]\|^4\, .
	\end{align*}
	Notice that the above lower bounds are achieved by the null estimators $\bhtheta(\bA) = 
	\E[\bTheta] \in \RR^{d \times r}$ and $\hat\bM(\bA) =
	 \E[\bTheta \bTheta^{\top}] \in \RR^{d \times d}$. 
\end{theorem}
The proof of this statement is similar to the one of \cref{thm:strong-signal-theta-lower-bound}.
Namely, we will prove that the mean square error achieved by simply taking the
 prior mean asymptotically agrees with the Bayesian MMSE for an estimator that
 has access to $\bLambda$ as additional information. 
 The argument is summarized in Appendix \ref{proof:thm:weak-signal-theta}.

\subsection{Estimation of  $\bLambda$}\label{sec:weak-regime-estimate-lambda}
We finally consider the technically most interesting case, namely the
 estimation of $\blambda$ in the weak signal regime.  For simplicity, we will restrict ourselves to
studying the matrix mean square error:
\begin{align}\label{eq:MMSEA}
 \MMSEas_n(\mu_{\Lambda}, \mu_{\Theta}) := \inf\limits_{\hat{M}(\,\cdot\, )} \frac{1}{n^2}
  \E\left[ \left\| \bLambda\bLambda^\sT - \hat{\bM}(\bA) \right\|_{F}^2 \right],
\end{align} 
where the infimum is taken over all estimators (measurable functions) 
$\hat{\bM}:\bA\mapsto \hat{\bM}(\bA) \in \RR^{n \times n}$.
Of course $\MMSEas_n$ depends on the distributions $\mu_{\Lambda}, \mu_{\Theta}$.
 
 In the rank-one case $r = 1$, if $\E_{\bThetas \sim \mu_{\Theta}}[\bThetas] \neq 0$, then 
 the naive estimator $r_n^{-1}\bA\by$ with 
 $\by = \E_{\mu_{\Theta}}[\bThetas] ^{-1}\mathbf{1}_d / \sqrt{d}$ is consistent:
\begin{align*}
	\frac{1}{n}\left\| r_n^{-1} \bA \by - \bLambda \right\|^2 \toP 0. 
\end{align*} 
In this case, a consistent estimate of $\bLambda \bLambda^{\sT}$ naturally follows.
Therefore, if $\bThetas$ has non-vanishing expectation, the  estimation problem is significantly easier. 
 
 When $r \geq 2$,  if $\mu_{\Theta}$ has non-zero mean, the same construction 
 leads to consistent estimation of the projection 
 of $\bLambda$ onto the direction determined by $\E[\bThetas]$
 (for $\bThetas\sim\mu_{\Theta}$). 
 Once this component is subtracted, the problem is effectively reduced to one 
 in which $\mu_{\Theta}$ has zero mean.
 
 For the remainder of this section, we focus on the more challenging case
 $\E[\bThetas] = \mathbf{0}_r$.
 In addition, for technical reasons  we will require $\mu_{\Theta}$ to have vanishing third  moment. 
\begin{assumption}\label{assumption:info}
	We assume $\E[\bThetas] =  \mathbf{0}_r$, 
	$\E[\bThetas  \otimes \bThetas \otimes \bThetas] = \bold{0}_{r \times r \times r}$, where $\otimes$ denotes the tensor product. 
	Furthermore, we assume that $\mu_{\Theta}, \mu_{\Lambda}$ are sub-Gaussian.
\end{assumption}
Our main results establish that, according to several criteria, 
estimation in the asymmetric model \eqref{model:weak-signal} with $n,d\to\infty$,
$d/n\to\infty$ is equivalent to estimation in the  symmetric spiked model 
\eqref{model:symmetric}.

Our first result on the relation between these models is in terms of mutual information.

\begin{theorem}\label{thm:free-energy}
Define the mutual information per coordinate in asymmetric model
of Eqs.~\eqref{model:weak-signal}, via
\begin{align} 
\Infoas_n(\mu_{\Lambda},\mu_{\Theta}) &:=\frac{1}{n}
 \E\log \frac{\de \P_{\bLambda,\bA}}{\de (\P_{\bLambda}\times\P_{\bA})}
  (\bLambda,\bA)\, . \label{eq:Info-Asymmetric-Def}
  \end{align}
 Further recall the definition of mutual information in the symmetric model 
 \eqref{model:symmetric}  given by Eq.~\eqref{eq:Info-Symmetric-Def}.
Within the setting of  Assumption \ref{assumption:info}, 
	we let $n,d \rightarrow \infty$ simultaneously with $d / n \rightarrow \infty$. In addition, 
	we require without loss of generality $\bQ_{\Theta} = q_{\Theta} \id_r$, cf.	Remark \ref{assumption:second-moment}. 
Then 
	the following limits exist and are equal
	\begin{align*}
		  \lim_{n,d \to \infty}  \Infoas_n(\mu_{\Lambda},\mu_{\Theta})  = 
		   \lim_{n\to \infty} \Infosy_n(\mu_{\Lambda};q_{\Theta})  \, .
	\end{align*}
\end{theorem} 

The proof of \cref{thm:free-energy}  is presented in 
 Appendix \ref{sec:proof-of-thm-lower-bound-lower-bound} for $\mu_{\Lambda}$ with bounded support.
  The generalization to
   $\mu_{\Lambda}$ with unbounded support is discussed in Appendix \ref{section:reduction-to-bdd-support}.

As mentioned above, earlier work determined 
the asymptotics of the  mutual information for the symmetric model
 $\Infosy_n(\mu_{\Lambda}; q_{\Theta})$. In particular, the next corollary follows 
directly from \cref{thm:free-energy} and Theorem \ref{prop:symmetric-spike}.
\begin{corollary}\label{corollary:free-energy-limit}
Recall that $S_r^+$ denotes the set of $r\times r$ positive semidefinite matrices, and  
 $\cF: \R_{\ge 0} \times S_r^+\to \R$ is defined in \cref{eq:FsQ}.
Under the conditions of \cref{thm:free-energy}, we have 
\begin{align*}
\lim_{n,d \to \infty}   \Infoas_n(\mu_{\Lambda},\mu_{\Theta}) = & \,
 \frac{1}{4}q_{\Theta}^2\|\E_{\bLambdas \sim \mu_{\Lambda}}[\bLambdas \bLambdas^{\top}]\|_F^2 -\cF_*(q_{\Theta}) \\
  := & \,
 \frac{1}{4}q_{\Theta}^2\|\E_{\bLambdas \sim \mu_{\Lambda}}[\bLambdas \bLambdas^{\top}]\|_F^2  - \sup_{\bQ\in S_r^+} \cF(q_{\Theta}^2, \bQ) \, .
\end{align*}
\end{corollary}

Recall the de Bruijn identity relating mutual information and minimum mean square error, 
see \cite{stam1959some,guo2005mutual,deshpande2017asymptotic}:
\begin{align}
\frac{1}{4}\MMSEsy_n(\mu_{\Lambda}; \sqrt s)= 	\frac{\de \;}{\de s}\Infosy_n(\mu_{\Lambda};\sqrt s) \, .
\label{eq:Identity-I-MMSE}
\end{align}
Since $\MMSEsy_n(\mu_{\Lambda};\sqrt s)$  is non-increasing in $s$, the asymptotics of
$\Infosy_n(\mu_{\Lambda};\sqrt s)$ essentially determines the asymptotics of 
$\MMSEsy_n(\mu_{\Lambda};\sqrt s)$. Namely, we have $\lim_{n\to\infty}\MMSEsy_n(\mu_{\Lambda};\sqrt s)=
\|\E_{\bLambdas \sim \mu_{\Lambda}}[\bLambdas \bLambdas^{\top}]\|_F^2 -4\frac{\partial }{\partial s}\cF_*(\sqrt s)$ for almost all values of $s$.

It would be tempting to conclude that Theorem \ref{thm:free-energy} and Corollary 
\ref{corollary:free-energy-limit} lead directly to analogous theorems relating
$\MMSEas_n(\mu_{\Lambda},\mu_{\Theta})$ and $\MMSEsy_n(\mu_{\Lambda};q_{\Theta})$.
Establishing such a consequence is more challenging than one would naively expect because 
we do not have an identity 
analogous
\footnote{One could differentiate the mutual information $\Infoas_n(\mu_{\Lambda},\mu_{\Theta}) $ with respect to the signal-to-noise ratio parameter ${q_{\Theta}}$, but the result is related to error in estimating $\bLambda\bTheta^{\sT}$ instead of $\bLambda\bLambda^{\sT}$.} 
to \cref{eq:Identity-I-MMSE} for the asymmetric model.
We can nevertheless establish the following, via a perturbation argument.
\begin{theorem}\label{thm:lower-bound}
	Under the conditions of \cref{thm:free-energy}, for all but countably many values of $q_{\Theta} > 0$, 
	 we have
	\begin{align}\label{eq:thm-lower-bound}
		\liminf\limits_{n,d \rightarrow \infty}\MMSEas_n(\mu_{\Lambda},\mu_{\Theta}) 
		\geq \lim\limits_{n \rightarrow \infty} \MMSEsy_n(\mu_{\Lambda};q_{\Theta}). 
	\end{align}
Further, consider a modified model in which the statistician observes $(\bA,\bY'(\ep))$,
where $\bA$ is given by Eq.~\eqref{model:weak-signal}, and 
\begin{align*}
	\bY'(\ep) := \frac{\sqrt{\eps}}{n} \bLambda \bLambda^{\top} + \bW', \;\;\; \bW'\sim\GOE(n)\, .
\end{align*}
Here, we assume $\bW'$ is independent of everything else. Denote by $\MMSEas_n(\mu_{\Lambda},\mu_{\Theta};\ep)$ the corresponding matrix mean square error.
Then, for all but countably many values of $q_{\Theta} > 0$, 
	 we have
\begin{align}\label{eq:thm-upper-ep-bound}
		\lim_{\eps\to 0+}\limsup\limits_{n,d \rightarrow \infty}\MMSEas_n(\mu_{\Lambda},\mu_{\Theta};\eps) 
		\le \lim\limits_{n \rightarrow \infty} \MMSEsy_n(\mu_{\Lambda};q_{\Theta}). 
	\end{align}
\end{theorem}
The proof of \cref{thm:lower-bound} is outlined in
 Appendix  \ref{sec:proof-of-thm:lower-bound} 
 (for distributions with bounded support)
  and \ref{section:reduction-to-bdd-support} (for the general case).
  
  \begin{remark}
In particular, Theorem \ref{thm:lower-bound} establishes that 
the estimation errors under the symmetric and asymmetric models coincide asymptotically,
provided that the error in the perturbed model $\MMSEas_n(\mu_{\Lambda},\mu_{\Theta};\eps)$
is uniformly continuous (in $n$) as $\eps\downarrow 0$. We expect this to be generically the case,
but proving this remains an open problem.
\end{remark}
The next theorem establishes a sequence of sufficient conditions under 
which we can prove asymptotic equivalence of estimation errors in
the asymmetric and symmetric models.
\begin{theorem}\label{thm:upper-bound}
	Under the conditions of Theorem \ref{thm:free-energy}, we further assume at 
	least one of the following conditions holds:
	\begin{enumerate}
		\item[$(a)$] $d n^{-3}(\log n)^{-6} \rightarrow \infty$.
		\item[$(b)$] $d(\log d)^{8/5}/n^{6/5}\to 0$ and $\mu_{\Lambda}$ has bounded support.
		\item[$(c)$] For the case $r=1$, define $Y = \sqrt{\gamma} \Lambda_0 + G$
		with $G\sim\normal(0,1)$ independent of $\Lambda_0\sim\mu_{\Lambda}$, 
		and define $\mathsf{I}(\gamma) = \E \log \frac{\dd p_{Y \mid \Lambda_0}}{\dd p_Y}(Y, \Lambda_0)$.
		Let 
	\begin{align*}
		\Psi(\gamma, s) = \frac{s^2}{4} + \frac{\gamma^2}{4{s}} - \frac{\gamma}{2} + \mathsf{I}(\gamma).
	\end{align*}
	Assume that the global maximum of 
	$\gamma \mapsto \Psi(\gamma, q_{\Theta})$ over $(0, \infty)$ is also the first stationary 
	point of the same function. 
	\end{enumerate}
Then, we have
\begin{align}
		\lim\limits_{n,d \rightarrow \infty}\MMSEas_n(\mu_{\Lambda},\mu_{\Theta}) 
		= \lim\limits_{n \rightarrow \infty} \MMSEsy_n(\mu_{\Lambda};q_{\Theta}). 
		\label{eq:EqualityErrors}
	\end{align}
(For condition $(b)$, the conclusion is guaranteed to hold 
for all but countably many values of  $q_{\Theta} > 0$.)
\end{theorem}
We defer the proof of \cref{thm:upper-bound} to Appendix \ref{sec:proof-thm:upper-bound}.

As anticipated in the introduction,  the results presented
in \cref{sec:weak-regime-theta} and \cref{sec:weak-regime-estimate-lambda} 
support two key statistical insights, which we next summarize:
\begin{enumerate}
\item In the weak signal regime, it is possible to partially recover $\bLambda$ while
 impossible to recover $\bTheta$ in any non-trivial sense. For instance, in the high-dimensional
 Gaussian mixture model, we might be able to estimate the labels, even if it is 
 impossible to estimate the cluster centers.
 
 In the next section, we will further explore the application of these results
 to Gaussian mixture models,  while in \cref{sec:real-data} we will investigate
 such asymmetry in real world datasets.
\item In this regime, ideal estimation accuracy is asymptotically independent of the distribution 
of the high-dimensional factor $\bTheta$. As demonstrated, for instance, by 
Eq.~\eqref{eq:EqualityErrors}, the only dependence on $\mu_{\Theta}$ is through its second moment.
\end{enumerate}

 The three sufficient conditions given in Theorem \ref{thm:upper-bound} correspond to
 three different arguments. 
 
 The most straightforward case is the one of condition $(a)$. We use the fact that 
 \begin{align}
 	\frac{1}{\sqrt{nd}}
 	\big(\bA\bA^{\sT}-d\id_n\big) = \frac{1}{n}\blambda \hat\bQ_{\Theta}
 	\blambda^{\sT} + \frac{1}{\sqrt{nd}}
 	\big(\bZ\bZ^{\sT}-d\id_n\big)+\mbox{cross terms}\, ,
 	\end{align}
 	where $\hat\bQ_{\Theta}:=\bTheta^{\sT}\bTheta/d \approx q_{\Theta}\id_r$.
 	For $d\gg n^3$, \cite{bubeck2016testing} proved that the total variation distance between
 	the distribution of  the Wishart matrix
 	$\big(\bZ\bZ^{\sT}-d\id_n\big)/\sqrt{nd}$ and the one 
 	of $\bW\sim \GOE(n)$ converges to $0$. While we still have to deal with the cross terms, 
 	under this condition the two models are close to each other.
 	
  	For $d\ll n^{3}$
  	the Wishart and GOE distributions are asymptotically mutually singular  \cite{bubeck2016testing},
  	and therefore proving asymptotic equality of the mean square error has
  	to rely on a more carefully analysis. 
 	 In fact, the proof of part $(b)$ follows a different path and relies heavily on
 	 Theorem \ref{thm:lower-bound}.
 	 
 	 Finally, part $(c)$ combines the bound of Theorem \ref{thm:lower-bound}
 	 with a matching bound that is based on the analysis of a Bayesian approximate 
 	 message passing (AMP) algorithm \cite{montanari2021estimation}. Indeed, the sufficient condition of part $(c)$ 
 	 coincides with the condition that Bayes AMP achieves Bayes optimal estimation error.

\section{Clustering under the Gaussian mixture model}\label{sec:GMM}

As an application of our theory, we consider clustering under Gaussian mixture model (GMM).
Throughout, we will assume that all Gaussian components have equal covariance 
$\bSigma$, and that $\bSigma$ is known. Without loss of generality, we can therefore assume
 that data are preprocessed  so that $\bSigma = \id_d$.  We will focus on the weak signal regime, 
 because  it is mathematically the most interesting regime. 

The Gaussian mixture model fits our general framework, with the $\bLambda_i$'s encoding 
the data point labels: $\bLambda_i$ takes $k$ possible values, with $k$ being the number of clusters. 
We will measure estimation accuracy using the overlap
\begin{align*}
\Overlap_n := \max\limits_{\pi\in \mathfrak{S}_k}\frac{1}{n}\sum\limits_{i = 1}^n 
\mathbbm{1}\left\{\hat{\bLambda}_i  = \bLambda^{\pi}_i\right\}\, .
\end{align*}
Here, $\mathfrak{S}_k$ denotes the group of permutations over $k$ elements, and $\bLambda_i^{\pi}$
denotes the action of this group on the cluster label encodings of the $i$-th sample.
(We will work with slightly different encodings for the cases $k=2$ and $k\ge 3$ below.)

%

\subsection{Two clusters with symmetric centers}\label{sec:symmetric-GMM}

As a warm-up example, we consider the case of $k=2$ clusters with equal weights. For $i\in\{1,\dots\, , n\}$, we observe an independent  sample
\begin{align}\label{model:simple-GMM}
\ba_i\sim \frac{1}{2}\,\normal(\bTheta / \sqrt[4]{nd}, \id_d)+\frac{1}{2}\,\normal(-\bTheta / \sqrt[4]{nd}, \id_d)\, .
\end{align}
Here, $\pm\bTheta / \sqrt[4]{nd} \subseteq \RR^d$ are the cluster
 centers.
Denoting by  $(\Lambda_i)_{i\le n} \iidsim \Unif(\{-1,+1\})$ the cluster labels, 
and by $\bA\in\RR^{n\times d}$ the matrix whose $i$-th row  corresponds to the $i$-th sample,
we have that $\bA$ follows model \eqref{model:weak-signal} with $r=1$.

 We further assume $\bTheta$ has independent coordinates: $(\Theta_j)_{j\le d} \iidsim \mu_{\Theta}$, where
 $\mu_{\Theta}$ is a centered sub-Gaussian distribution and has zero  third moment.
 \begin{remark}\label{rmk:MinimaxCenters}
 We note that, for $r=1$, there is no real loss of generality in assuming 
 $(\Theta_j)_{j\le d}$ to be i.i.d. sub-Gaussian. Indeed, model \eqref{model:weak-signal}
 is equivariant under rotations $\bTheta\mapsto \bOmega\bTheta$, 
 $\bA\mapsto \bA\bOmega^{\sT}$, where $\bOmega\in\RR^{d\times d}$
 is an orthogonal matrix. Further, any loss function that depends uniquely on $\bLambda$
 is also invariant under the same group. Consider minimax estimation when $\bTheta$
 belongs to the sphere:  $\|\bTheta\|^2_2 = d\, q_{\Theta}$. As a consequence of the 
 Hunt-Stein theorem, the least favorable prior is the uniform distribution over 
 the same sphere. We expect the asymptotic Bayes risk under this prior
 (and therefore the minimax risk) to be the same as the risk under the 
 prior $(\Theta_j)_{j\le d} \iidsim \normal(0,q_{\Theta})$.
 \end{remark} 
 
 \begin{proposition}\label{prop:GMM}
 Consider the Gaussian mixture model as in \cref{model:simple-GMM}.
 Assume $n, d \to \infty$ simultaneously and $d/n\to \infty$, then the following results hold:
 \begin{itemize}
 \item[$(a)$] If $q_{\Theta}\le 1$, then, for any clustering estimator $\hbLambda$, as $n,d \to \infty$ we have
 \begin{align}
 \Overlap_n \toP \frac{1}{2}\, .
 \end{align} 
 \item[$(b)$] If $q_{\Theta}> 1$, let $s_*$ be the largest non-negative solution of
 \begin{align}
 s = q_{\Theta}^2\E\big\{\tanh\big(s+ \sqrt{s}G\big)^2\big\}\, ,
 \end{align} 
 where $G \sim \normal(0,1)$. Then $s_*>0$ and there exists an estimator achieving 
 \begin{align}
 \Overlap_n \toP \Phi(\sqrt{s_*}), \,
 \end{align}
 where $\Phi$ denotes the cumulative distribution function for standard Gaussian distribution.
 \end{itemize}
 \end{proposition}

The proof of this result uses the characterization of optimal estimation in the 
corresponding symmetric model proven in \cite{deshpande2017asymptotic}, and we
 present the proof of point $(a)$ in Appendix \ref{sec:proof-of-prop:GMM}.
The overlap in point $(b)$ can be achieved using orthogonal invariant Bayes AMP with 
spectral initialization  on $(\bA\bA^{\top} - d\id_n) / \sqrt{nd}$.
This algorithm is described and analyzed in \cite{mondelli2021pca}. 

\subsection{Two or more clusters with orthogonal centers}\label{sec:general}

We next consider the case of $k\ge 2$ clusters with approximately orthogonal centers. We denote by 
$\{\bTheta_{\cdot i} / \sqrt[4]{nd}: i \in [k]\} \subseteq \RR^d$ the cluster centers.
 Let $\bTheta \in \RR^{d \times k}$ with the $i$-th column given by $\bTheta_{\cdot i}$.
  For $j \in [d]$, we let  $\bTheta_j \in \RR^k$ be the $j$-th row of $\bTheta$. 
  We assume $\bTheta_j \iidsim \mu_{\Theta}$, where $\mu_{\Theta}$ 
  is sub-Gaussian  with vanishing first and third moments 
  and diagonal covariance: $\Cov(\bTheta_1) = q_{\Theta} \id_k$. 
  
Let $\be_j$ be the $j$-th standard basis vector in $\RR^k$.
 We encode the data point labels by setting $\bLambda_i = \be_j$ if and only if the  
 $i$-th sample belongs to the $j$-th cluster,
 and consider the case of equal proportions, so that 
 $(\bLambda_i)_{i\le n} \iidsim \Unif(\{\be_1, \cdots, \be_k\})$.

As before, we let  $\bA \in \RR^{n \times d}$ be the matrix  whose 
rows are i.i.d. samples $\ba_i$ from the Gaussian mixture model with centers 
$\{\bTheta_{\cdot i} / \sqrt[4]{nd}: i \in [k]\}$. With these definitions,
the matrix $\bA$ is distributed according to model \eqref{model:weak-signal}.
\begin{remark}
While we state our results for random $\bTheta$, we can 
generalize Remark \ref{rmk:MinimaxCenters} to the present setting. 
This argument implies that the results of
this section also characterize the minimax estimation error over the class
of problems with orthogonal centers $\bTheta^{\sT}\bTheta =  d\, q_{\Theta} \id_k$.
\end{remark}

Our next theorem establishes the threshold for weak recovery of the cluster
labels in the high-dimensional regime $d/n\to\infty$. Recall the function $\cF(s, \bQ)$
is defined in Eq.~\eqref{eq:FsQ}, where we take $\mu_{\Lambda} = 
\sum_{i=1}^k\delta_{\be_i}/k$. We let $\bQ_0 := \bfone_k\bfone_k^{\sT}/k^2$ and define the threshold
\begin{align}
q_{\Theta}^{\mathsf{info}}(k):= \inf\big\{q_{\Theta}\ge 0:\; \sup_{\bQ\in S_k^+}\cF(q_{\Theta}^2, \bQ)
>\cF(q_{\Theta}^2, \bQ_0)\big\}\, .\label{eq:ClusterThreshold}
\end{align}

 \begin{theorem}\label{thm:GMM1}
 Consider the Gaussian mixture model with $k$ components of equal weights, 
 in the high-dimensional asymptotics $d,n\to\infty$, $d/n\to\infty$.
 Under the above assumptions on the centers $\bTheta$, the following results hold:
\begin{itemize}
\item[$(a)$] If $q_{\Theta} < q_{\Theta}^{\mathsf{info}}(k)$, then for any estimator $\hat{\bLambda}: \RR^{n \times d} \to \{\be_j: j \in [k]\}^n$ that is a measurable function of the input $\bA$, we have
 	\begin{align*}
 		 \plim_{n,d\to\infty}\Overlap_n  = \frac{1}{k} \, . 
 	\end{align*}
\item[$(b)$] Assume either $d n^{-3}(\log n)^{-6} \rightarrow \infty$ or 
$d n^{-6/5}(\log d)^{8/5} \to 0$.
If $q_{\Theta}>q_{\Theta}^{\mathsf{info}}(k)$, 
then there exists an estimator $\hat{\bLambda}: \RR^{n \times d} \to \{\be_j: j \in [k]\}^n$ that is a measurable function of the input $\bA$, such that 
 	\begin{align*}
 		  \liminf_{n,d\to\infty}\E[\Overlap_n]  > \frac{1}{k} \, . 
 	\end{align*}
\end{itemize}
 %
 \end{theorem}
 We defer the proofs of parts $(a)$ and $(b)$ of
 \cref{thm:GMM1} to Appendices \ref{sec:proof-of-thm:GMM1} and \ref{sec:proof-of-thm:GMM2},
  respectively. Note that \cite[Theorem 2]{banks2018information} 
  implies $q_{\Theta}^{\mathsf{info}}(k) = 2\sqrt{k \log k}\cdot (1+o_k(1))$ as
  $k \to \infty$ (however, \cite{banks2018information} does not establish  a sharp threshold).
 In contrast, Theorem  \ref{thm:GMM1} derives the exact threshold for every $k$.
 
 Table \ref{table:thresholds} collects values for the thresholds 
 $q_{\Theta}^{\mathsf{info}}(k)$ for a few values of $k$,
 as obtained by numerically evaluating Eq.~\eqref{eq:ClusterThreshold}. 
 For $k\le 4$, this is expected to coincide with the spectral threshold,
 namely  $q_{\Theta}^{\mathsf{info}}(k)=k$ \cite{lesieur2016phase}. (Notice that the apparent discrepancy with the 
 threshold for $k=2$ in the previous section is due to the different normalization adopted here.) 
 
  \begin{table}
 \centering 
 	\begin{tabular}{lc}
 	\hline
 		$k$ & $q_{\Theta}^{\mathsf{info}}(k)$ \\ \hline
 		2 & 2 \\
 		3 & 3 \\
 		4 & 4 \\
 		5 & 4.95 \\
 		6 & 5.81 \\
 		7 & 6.61 \\
 		8 & 7.36 \\ \hline
 	\end{tabular}
 	\caption{Information-theoretic thresholds $q_{\Theta}^{\mathsf{info}}(k)$ for $2 \leq k \leq 8$.}
 	\label{table:thresholds}
 \end{table}

\subsection{Numerical experiments}

We present in this section numerical experiments suggesting that
the theory of the last sections is already relevant at moderate values of $d,n$. 
We consider several clustering methods, and compare their performances 
with the threshold $q_{\Theta}^{\mathsf{info}}(k)$.

For our experiment we use built-in functions in {\sf Python3}, for
the following clustering methods:
\begin{enumerate}
	\item[$(1)$] Lloyd's algorithm, as implemented by the function $\mathsf{KMeans()}$ in the
	 $\mathsf{scikit}$-$\mathsf{learn}$ module with option $\mathsf{algorithm = ``lloyd"}$.
	\item[$(2)$] Agglomerative clustering, implemented by 
	the function $\mathsf{AgglomerativeClustering()}$ in the 
	$\mathsf{scikit}$-$\mathsf{learn}$ module with default parameters.
	\item[$(3)$] EM algorithm, implemented by the function 
	 $\mathsf{GaussianMixture()}$ in the  $\mathsf{scikit}$-$\mathsf{learn}$ module with 
	 default parameters. 
	\item[$(4)$] A semidefinite programming (SDP) relaxation  described in \cite{peng2007approximating}. 
	We use the $\mathsf{cvxpy}$ module for the optimization steps. 
\end{enumerate}
In Figure \ref{fig:sim} we present results for these algorithms for 
 $n = 100$, $d = 2000$, and $\mu_{\Theta} = \normal(0,q_{\Theta})$. 
 For each value of the pair $(k, q_{\Theta})$, we run 100 independent trials, and plot 
 the average overlap versus $q_{\Theta}$. For the case $k=2$,
 we consider two slightly different settings: ``Symmetric=True'' 
 corresponds to the case of two centers symmetric around the origin, as
 in \cref{sec:symmetric-GMM}, and 
 ``Symmetric=False'' corresponds to the case of two approximately orthogonal centers as per
 \cref{sec:general}. 
 We also report the threshold $q_{\Theta}^{\mathsf{info}}(k)$, its large $k$ approximation
 $2\sqrt{k\log k}$, and the algorithmic threshold $q^{\mathsf{algo}}_{\Theta}(k) = k$
 (this is the conjectured threshold for efficient recovery, which coincides with the spectral threshold 
 \cite{lesieur2016phase,montanari2021estimation}).
 
Despite the small sample size, we observe that $q_{\Theta}^{\mathsf{info}}(k)$ 
appears to capture the onset of non-trivial clustering accuracy across multiple algorithms.

\begin{figure}[htp]
     \centering
     \begin{subfigure}[b]{0.48\textwidth}
         \centering
         \includegraphics[width=\textwidth]{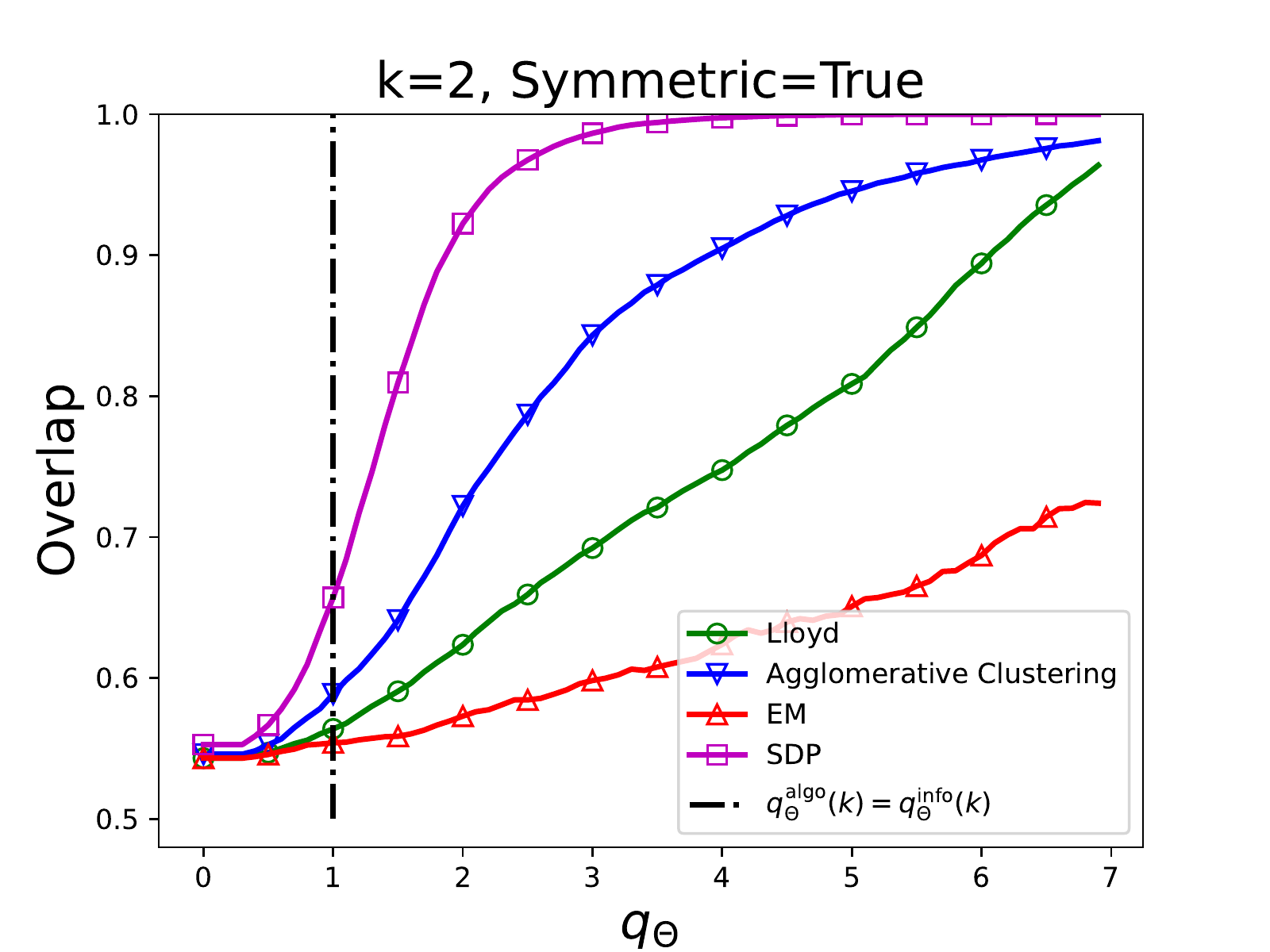}
     \end{subfigure}
     \hfill
     \begin{subfigure}[b]{0.48\textwidth}
         \centering
         \includegraphics[width=\textwidth]{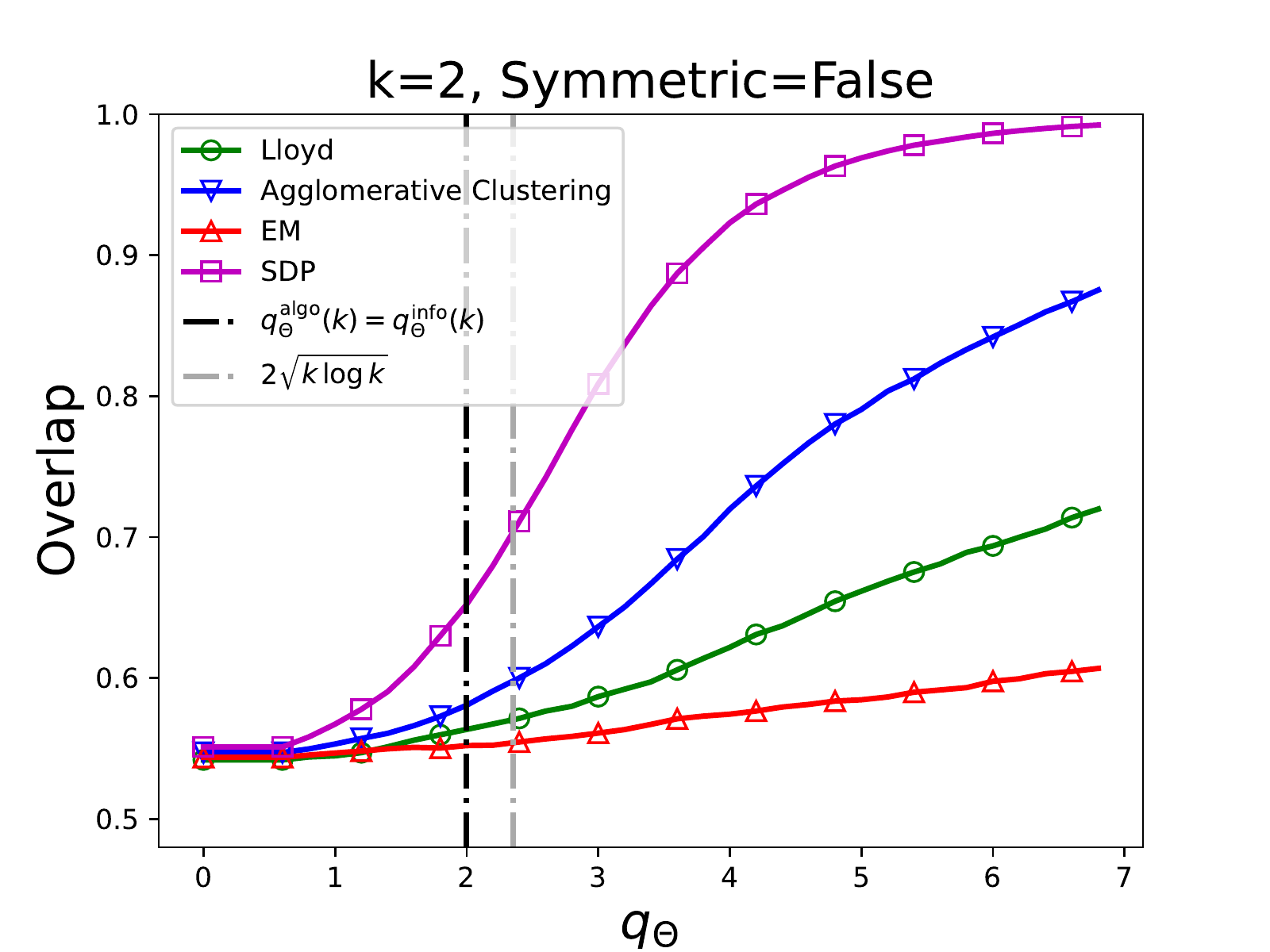}
     \end{subfigure} \\
	 \begin{subfigure}[b]{0.48\textwidth}
         \centering
         \includegraphics[width=\textwidth]{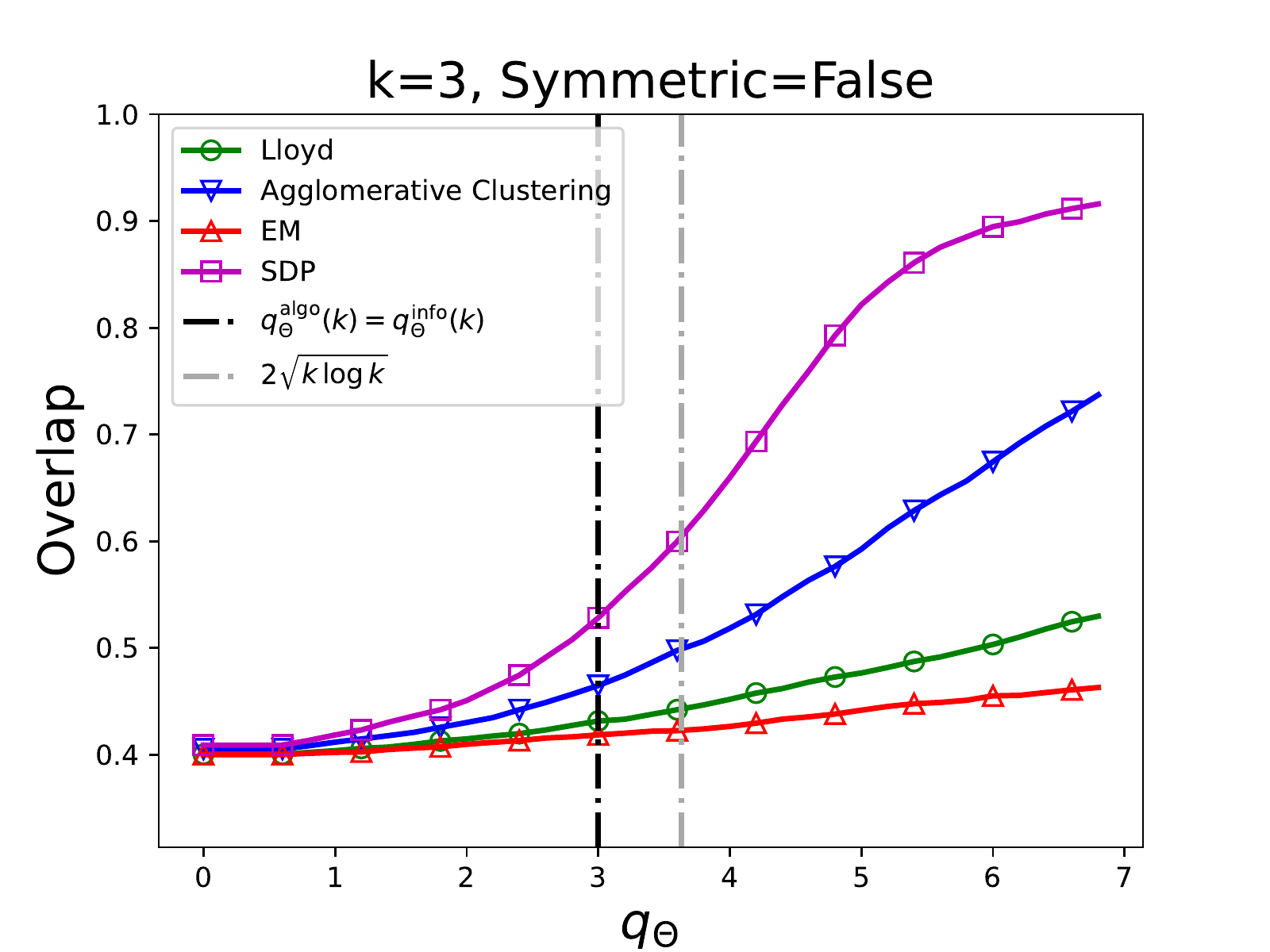}
     \end{subfigure}
     \hfill
     \begin{subfigure}[b]{0.48\textwidth}
         \centering
         \includegraphics[width=\textwidth]{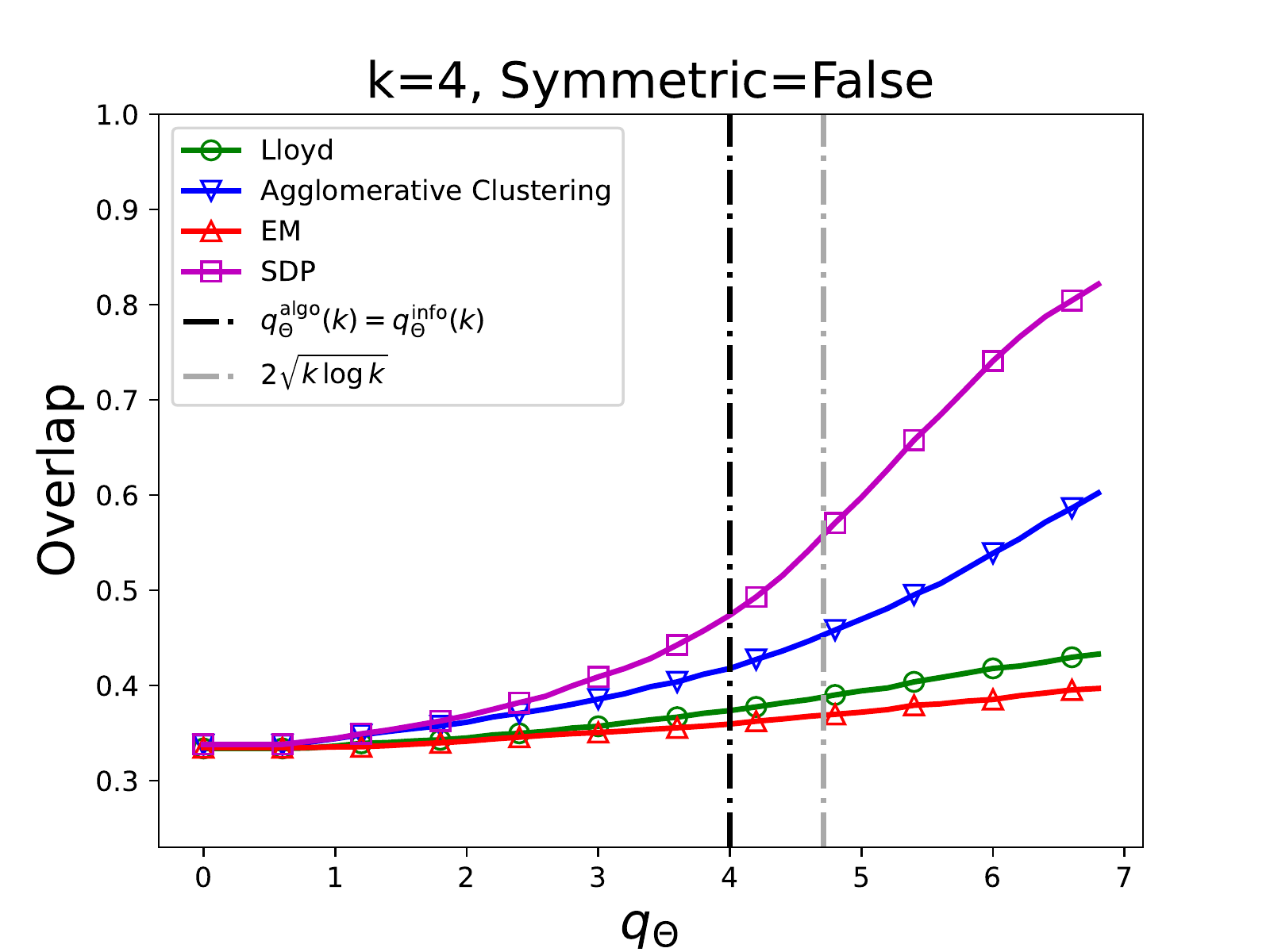}
     \end{subfigure}\\
	 \begin{subfigure}[b]{0.48\textwidth}
         \centering
         \includegraphics[width=\textwidth]{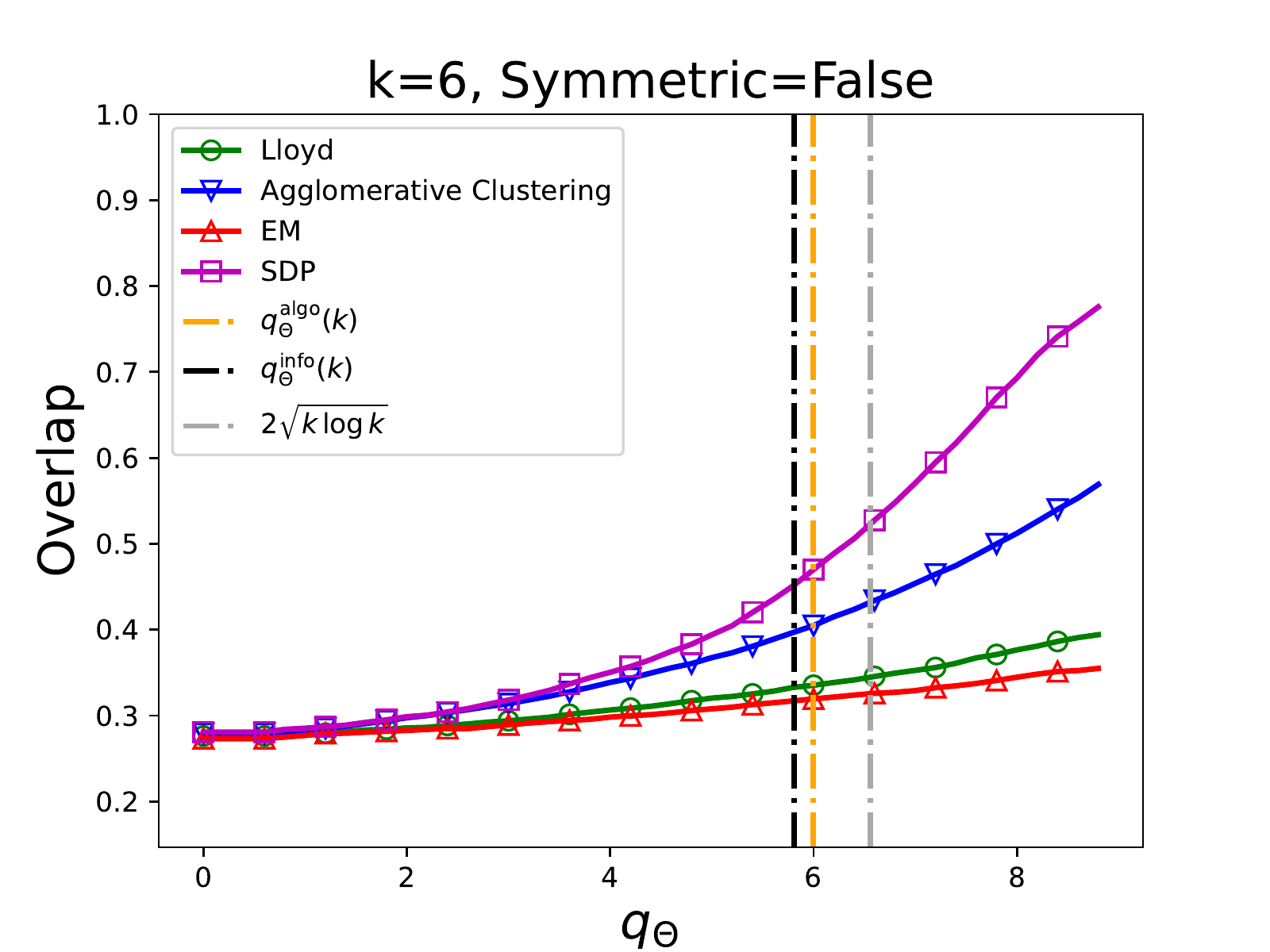}
     \end{subfigure}
     \hfill
     \begin{subfigure}[b]{0.48\textwidth}
         \centering
         \includegraphics[width=\textwidth]{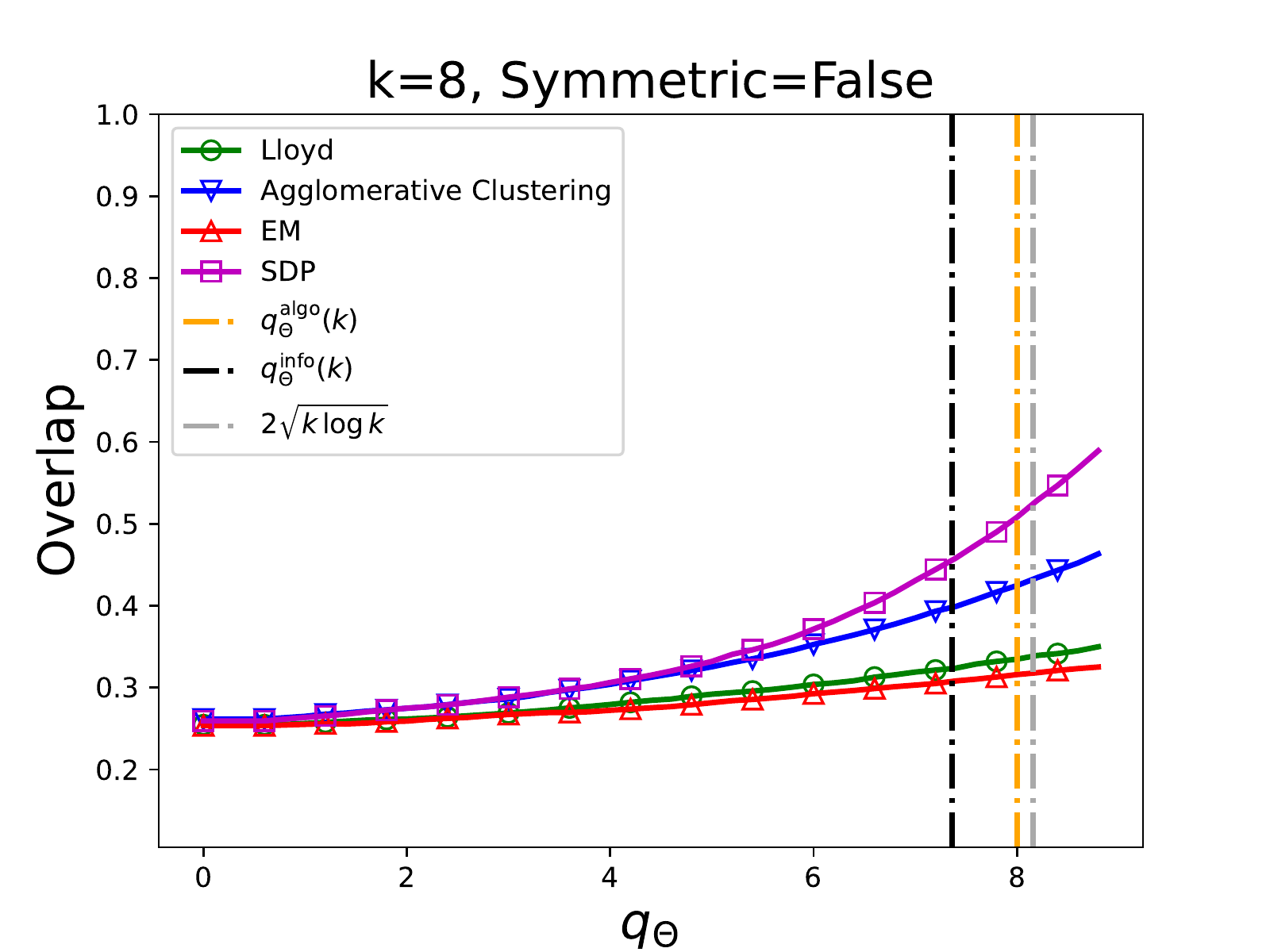}
     \end{subfigure}
        \caption{Average overlap achieved by several clustering algorithms on the Gaussian mixture
        model with $n = 100$ datapoints, $d = 2000$ dimensions, averaged
        over $100$ instances. The black vertical line corresponds to the information-theoretic threshold for
        identifying clusters significantly better than random guessing; the orange vertical line 
        corresponds to the spectral or algorithmic threshold; the grey vertical line corresponds
        to the approximated information-theoretic threshold $2\sqrt{k \log k}$.}
        \label{fig:sim}
\end{figure}

\section{Asymmetry in factors estimation: real world datasets}\label{sec:real-data}

Previous sections imply the existence of gaps in the estimation of 
 $\bLambda$ and $\bTheta$, in the high-dimensional asymptotics $d/n\to\infty$. 
In summary, in the strong signal regime, $\bLambda$ can be estimated consistently up 
to a potential rotation, while $\bTheta$ can only be partially recovered. 
On the other hand, in the weak signal regime, $\bLambda$ can be partially recovered, 
while no estimator achieves better asymptotic performance than a naive one in terms
of the estimation of $\bTheta$. 

In this section, we investigate this asymmetry in real world datasets. 
We focus on a problem that can be modeled as clustering with $k=2$ clusters
(this can be modeled as a GMM model, leading to Eq.~\eqref{model:spike0} with $r=1$ 
as described in the previous section). 

\subsection{1000 Genomes Project}\label{sec:1000G}

Our first experiment involves genotype data from the 1000 Genomes Project \cite{10002015global}. 
This provides genotypes for $n=2,504$ individuals grouped in five
population groups (corresponding to their geographic origins).
For our experiments, we extract $d=100,000$ common single-nucleotide polymorphisms (SNPs).
 Our preprocessing steps follow from \cite{zhong2020empirical}. After preprocessing, we add
 independent Gaussian noise with variance 5 to the data matrix, to make the problem more challenging.

Principal component analysis (PCA) is often used in genome-wide association studies, 
in particular to explore the genetic structure of human populations \cite{novembre2008genes,novembre2008interpreting}.  As a first step of
 our experiment, for each pair of population groups, we randomly extract 30 subjects from each 
 group without replacement. The subsampled observations form a $60 \times 100,000$ genotype matrix, 
 the columns of which are then centered and rescaled. We next run PCA on this subset, and plot the projections 
 onto the top 2 principal components. We display one typical outcome of PCA in
  Figure \ref{fig:1000G_PCA}. From the figures, we see that despite the high-dimensionality, 
  PCA still reflects the underlying population structure. We interpret this as indicating that 
  non-trivial clustering can be achieved on these data.

\begin{figure}[htp]
\centering
\includegraphics[width=17cm]{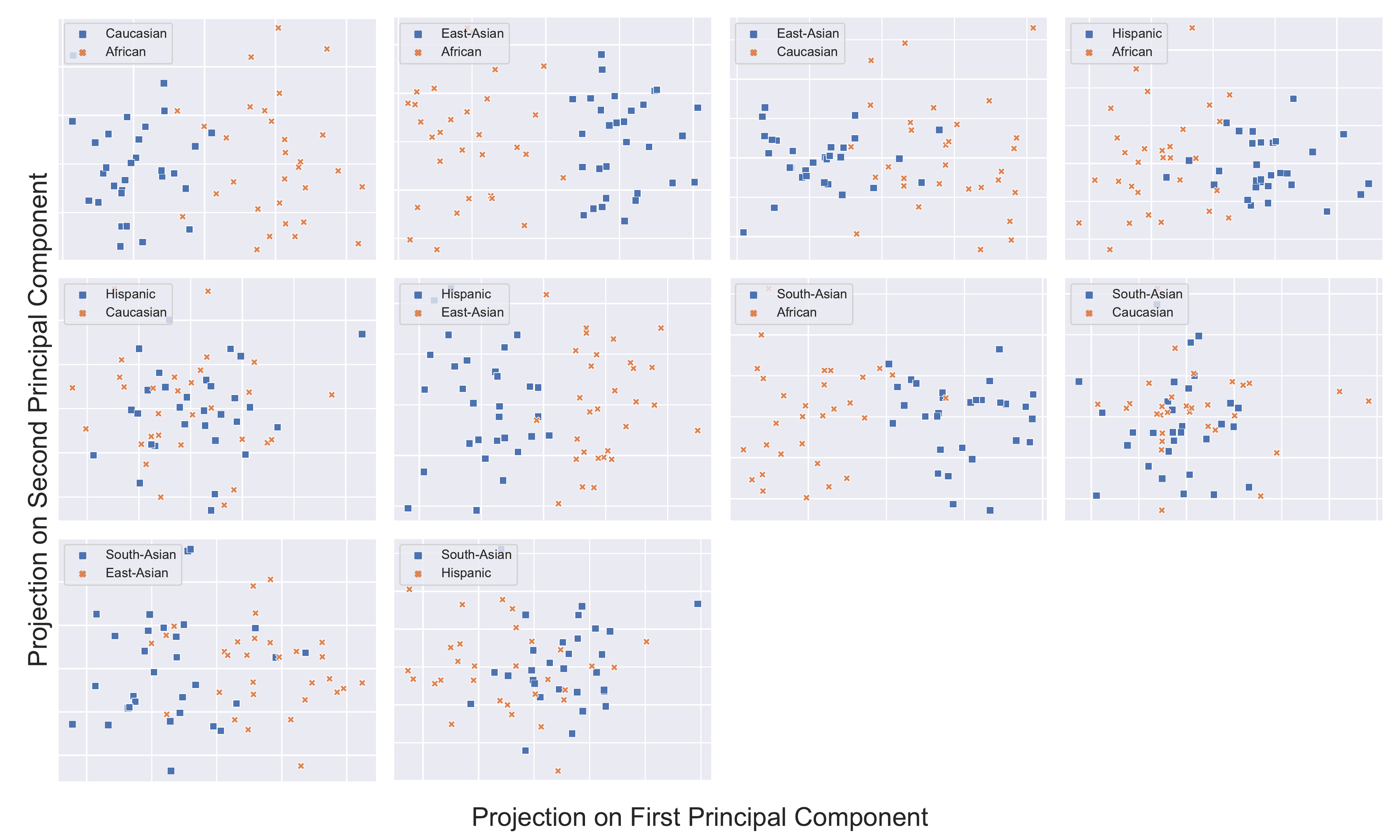}
\caption{Illustration of PCA on a subset of 1000 Genomes Project data. 
In these plots, the $x$ axis represents the projection onto the first principal component, 
and the $y$ axis represents projection onto the second principal component. 
Point colors and shapes correspond to population groups.  
Each experiment involves 60 individuals in total, with 30 individuals from each of the 
two population groups. }
\label{fig:1000G_PCA}
\end{figure} 

\begin{figure}[htp]
\centering
\includegraphics[width=17cm]{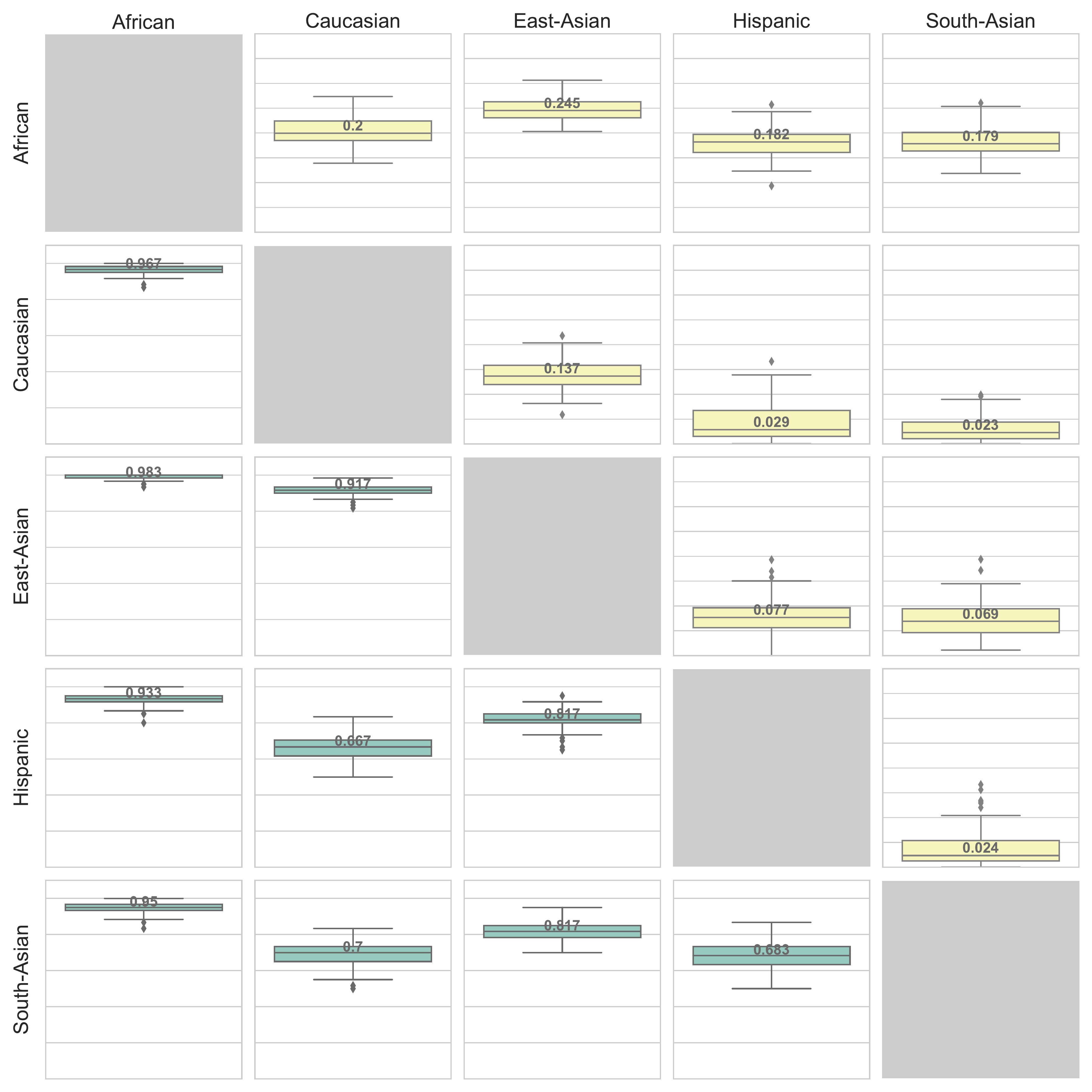}
\caption{Simulation results derived from 1000 independent experiments for the 1000 Genomes 
Project dataset. The boxplots in the upper triangle display the 
quantiles of the  normalized inner products between the estimated cluster centers.
 The boxplots in the lower triangle display the quantiles of overlaps between true labels
 and labels obtained via K-means
  clustering. We annotate the medians in the corresponding figures for readers' 
  convenience.}
\label{fig:1000G_boxplots}
\end{figure}

To further support our conclusion, we run K-means clustering on the subsampled datasets 
(using $\mathsf{sklearn.cluster.KMeans}$ in Python 3 with default parameters).
 We then compute the overlap between the true and estimated labels (in this example labels correspond
 to population groups). 
 We repeat this procedure  independently  1000 times on randomly selected subsets 
  of the data. 
 The outcomes are recorded and displayed in the lower triangle of Figure \ref{fig:1000G_boxplots}. 
 From the figures, we see that K-means clustering estimates the labels significantly 
 better than random guessing (i.e., better than $50\%$ accuracy)
 and achieves near-perfect recovery for certain pairs of population groups.

We next  estimate the cluster centers $\bTheta$, for each pair of population groups.
We take two non-overlapping subsets of the data (each with size 60) and run K-means 
on each subset: this leads to two distinct estimates of the cluster  centers
$(\hbTheta_i^{(1)})_{i\in\{1,2\}}$ for data subset $1$, and
 $(\hbTheta_i^{(2)})_{i\in\{1,2\}}$ for data subset $2$. 
 We then compute the  maximum normalized inner
  product 
  $$\max_{i,j \le 2}|\langle\hbTheta_i^{(1)}, \hbTheta_j^{(2)}\rangle|/(\|\hbTheta_i^{(1)}\|_2\|\hbTheta_j^{(2)}\|_2)$$ 
   between the estimated
  cluster centers obtained via K-means from these two subsets of data.
 
This procedure is again repeated for
    1000 times independently, and the distributions of the maximum normalized inner products are displayed in
     the upper triangle of Figure \ref{fig:1000G_boxplots}. 
We observe that, for several population pairs, the estimates  $\hbTheta^{(1)}$,
$\hbTheta^{(2)}$ are not significantly correlated (using initials, this is the case
for the pairs C-H, C-SA, EA-H, EA-SA, H-SA). Since these estimates are obtained based on independent samples from the same population, we conclude that they are also 
not significantly correlated with the true centers. When this happens, the behavior of this
clustering problem seems to be captured by the weak signal regime analyzed in the
previous sections: clusters can be estimated in a non-trivial way, but cluster centers
cannot be estimated.

For the other population pairs, the cluster centers estimates are correlated,
and clustering accuracy is very high (this is the case for pairs A-C, A-EA, A-H, A-SA, 
with C-EA not as clear a case). This is analogous to what we observe in our model
in the strong signal regime.


\subsection{RNA-Seq gene expression}

\begin{figure}[htp]
\centering
\includegraphics[width=17cm]{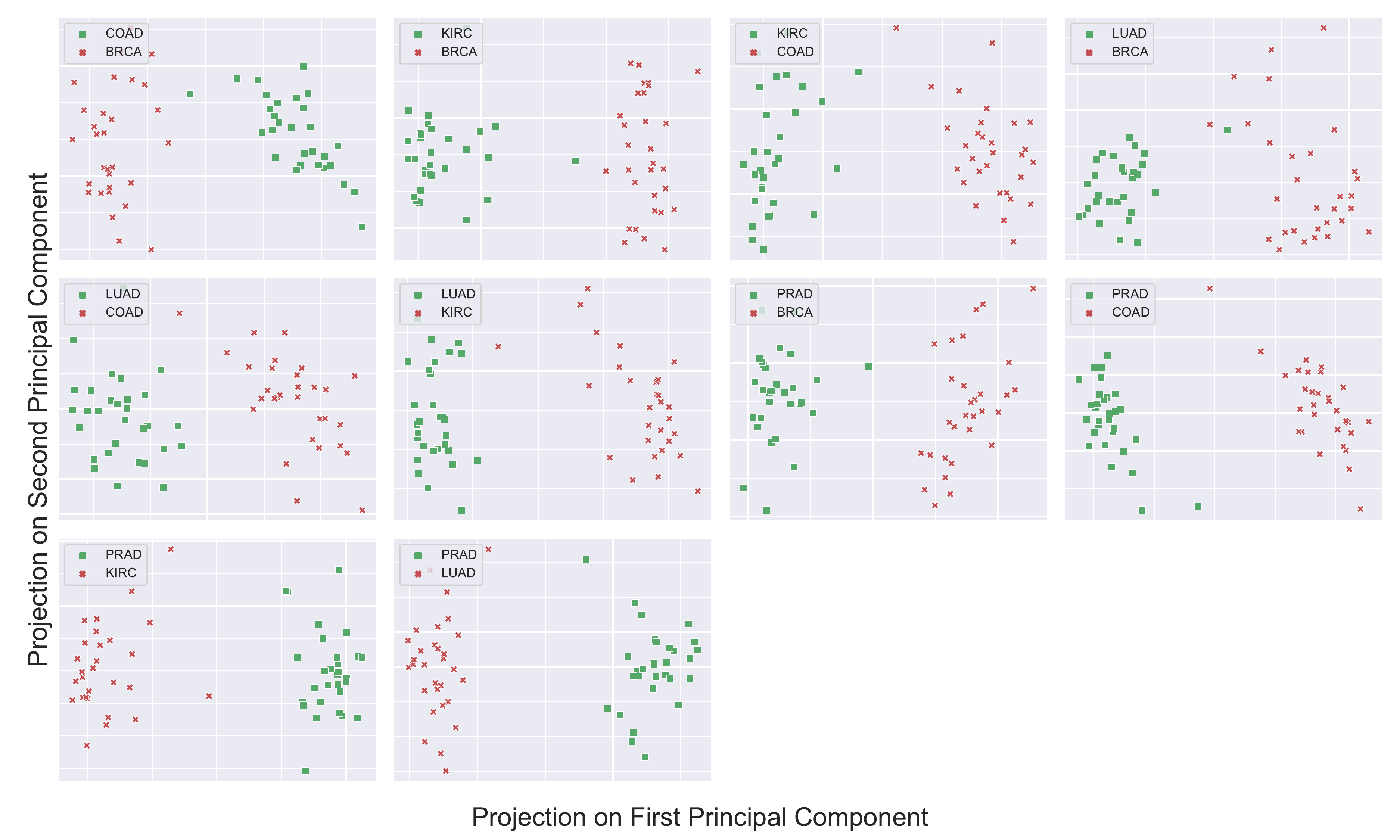}
\caption{PCA on subsets of RNA-Seq gene expression data: each time we select 30 datapoints from
each of the two cancer groups at random. Point colors and shapes stand for different cancer groups.
We plot the  projections of these datapoints onto the subspace defined by their first two 
principal components.}
\label{fig:cancer_PCA}
\end{figure} 

\begin{figure}[htp]
\centering
\includegraphics[width=16cm]{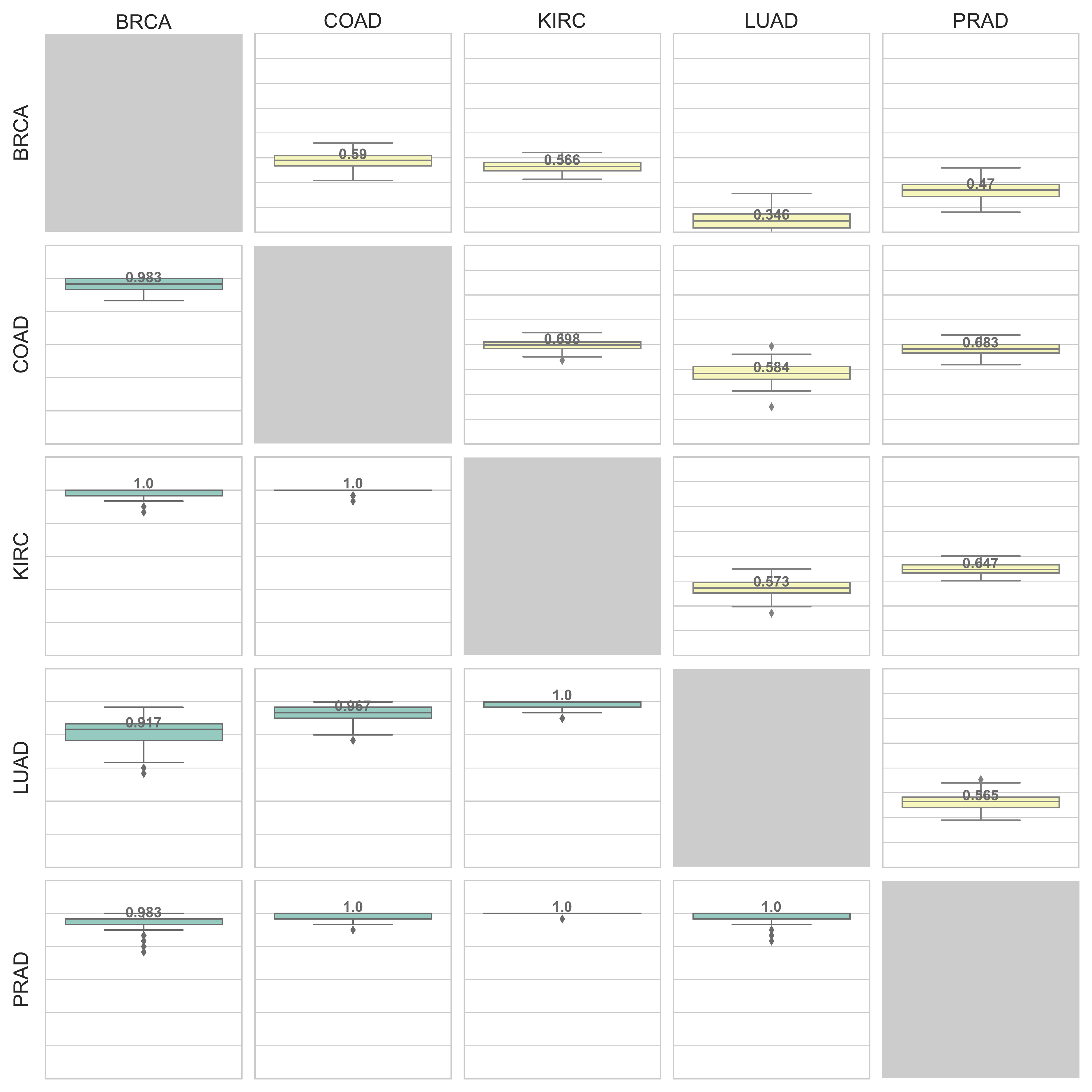}
\caption{Simulation results derived from 1000 independent experiments for the
UCI gene expression dataset.
 The boxplots in the upper triangle display the 
quantiles of the  normalized inner products between estimated cluster centers.
  The boxplots in the lower triangle display the quantiles of the accuracy
(overlap) in reconstructing the true clusters. Medians are annotated in the figures.}
\label{fig:cancer_boxplots}
\end{figure}

We carry out a similar experiment on gene expression data for different types of cancers from the UCI Machine Learning 
Repository\footnote{\url{https://archive.ics.uci.edu/ml/datasets/gene+expression+cancer+RNA-Seq}} 
\cite{Dua:2019}. The dataset contains 801 samples and 20531 attributes, with the 
predictors being 
RNA-Seq gene expression levels measured by the Illumina HiSeq platform. Before proceeding, 
again we apply additive Gaussian noise to the data matrix, with mean zero and variance 5. 
We consider five different cancer types, denoted by 
``COAD'', ``BRCA'', ``KIRC'', ``LUAD'' and ``PRAD''.

For each pair of cancer groups we subsample 30 subjects from each group, 
 to construct a $60 \times 20531$ data matrix.
We then center and rescale the columns of this matrix to unit norms. 
 A typical 
outcome of PCA is presented in Figure \ref{fig:cancer_PCA}. 
We observe that clusters corresponding to different cancer groups are well separated
for each of the pairs. 
In Figure \ref{fig:cancer_boxplots}, we report the overlaps between the 
labels obtained from K-means clustering and the ground truth labels. 
The overlaps are very high for 
all pairs. These plots summarize the results of 1000 independent repetitions of this 
experiment. 

In the upper half of the same figure, we present the 
maximum normalized inner products between the estimated cluster centers on 
two independent subsamples. The correlation is significantly different from zero, but 
far from being close to one. 
Once more, this is analogous to the strong signal regime in 
our analysis.

\section*{Acknowledgements}

This work was supported by the NSF through award DMS-2031883, the Simons Foundation
 through Award 814639 for the Collaboration on the Theoretical Foundations of Deep Learning, 
 the NSF grant CCF-2006489 and the ONR grant N00014-18-1-2729, and a grant from 
 Eric and Wendy Schmidt at the Institute for Advanced Studies. 
 Part of this work was carried out while Andrea Montanari was on partial leave
  from Stanford and a Chief Scientist at Ndata Inc dba Project N. The present 
  research is unrelated to AM’s activity while on leave.

\bibliographystyle{alpha}
\bibliography{bibliography.bib}

\newpage
\begin{appendices}
\section{Preliminaries}\label{sec:preliminaries}

\subsection{Further notations and conventions}\label{app:notation}


In this section, we present an incomplete summary of the notations and conventions that 
will be applied throughout the appendix.

For two sequences of random vectors $\{\bX_{n}\}_{n \in \NN_+} \subseteq \RR^k$ and $\{\bY_{n}\}_{n \in \NN_+}  \subseteq \RR^k$, we say $\bX_{n} \overset{P}{\simeq} \bY_{n}$ if and only if $\|\bX_{n} - \bY_{n}\| = o_p(1)$. 
For $n,k \in \NN_+$ and matrix $\bX \in \RR^{n \times k}$ with the $i$-th row denoted by $\bx_i \in \RR^k$, we let $\hat{p}_{\bX}$ be the empirical distribution of the $\bx_i$'s:
\begin{align*}
	\hat{p}_{\bX} := \frac{1}{n}\sum_{i = 1}^n\delta_{\bx_i}\, ,
\end{align*}
where $\delta_{\bx_i}$ is the point mass at $\bx_i$. 

\subsection{Wasserstein distance}

For two probability distributions $\mu_1, \mu_2$ over $\RR^r$, recall that the Wasserstein distance between $\mu_1$ and $\mu_2$ is defined as
\begin{align}\label{eq:wasserstein}
	W_2(\mu_1, \mu_2) := \left( \inf_{\gamma \in \Gamma(\mu_1, \mu_2)} \int_{\RR^r \times \RR^r} \|x - y\|^2 \dd \gamma(x, y) \right)^{1/2},
\end{align}
where $\Gamma(\mu_1, \mu_2)$ denotes the collection of all probability distributions over $\RR^r \times \RR^r$ with marginals $\mu_1$ and $\mu_2$ on the first and last $r$ coordinates, respectively. One observation is that for matrices $\bL_1, \bL_2 \in \RR^{n \times r}$, the $W_2$ distance between the empirical distributions of their rows is upper bounded by the Frobenius norm of their difference: $W_2(\hat{p}_{\bL_1}, \hat{p}_{\bL_2}) \leq \frac{1}{\sqrt{n}}\|\bL_1 - \bL_2\|_F$.  
 
\section{Technical lemmas}

\begin{lemma}\label{lemma:concentration-of-sample-covariance}
	Let $\bx_1, \cdots, \bx_n \in \mathbb{R}^p$ be independent $\tau^2$ sub-Gaussian random vectors, with mean 0 and covariance $\E[\bx_i\bx_i^T] = \bSigma$. We define the sample covariance matrix $\hat{\bSigma} = \frac{1}{n} \sum_{i = 1}^n \bx_i\bx_i^T$. Then for any $s \geq 100$, with probability at least $1 - 2e^{-ps^2/3200}$ we have
	\begin{align*}
		\| \hat{\bSigma} - \bSigma \|_{\rm{op}} \leq s\tau^2 \sqrt{\frac{p}{n}},
	\end{align*}
	provided that $s\sqrt{p/n} \leq 25$.
\end{lemma}

\begin{lemma}[Wedin's theorem \cite{wedin1972perturbation}]\label{lemma:wedin}
Let $\bA_0$, $\bA_1 \in \mathbb{R}^{m \times n}$ have singular value decomposition (for $a \in \{0, 1\}$)
\begin{align*}
	\bA_a = \bU_a \bSigma_a \bV_a^T, 
\end{align*}
with $\bSigma_a$ containing the singular values of $\bA_a$ in decreasing order. Furthermore, we let $\bU_{a,+} \in \mathbb{R}^{m \times k(a)}$, $\bV_{a, +} \in \mathbb{R}^{n \times k(a)}$, be formed by the first $k(a)$ columns of $\bU_a$, $\bV_a$, respectively, such that
\begin{align*}
	\bU_a = [\bU_{a, +} | \bU_{a, -}], \qquad \bV_a = [\bV_{a, +} | \bV_{a, -}].
\end{align*}
Let $\sigma_k(\cdot)$ denote the $k$-th largest singular value of a matrix. Finally assume $\Delta \equiv \sigma_{k(1)}(\bA_1) - \sigma_{k(0) + 1}(\bA_0) > 0$. Let $\bP_a = \bV_{a, +}\bV_{a, +}^T$(respectively, $\bQ_a = \bU_{a, +} \bU_{a, +}^T$) denote the projector onto the right singular space (left singular space) corresponding to the top $k(a)$ singular values of $\bA_a$. Then we have
\begin{align*}
	\|(\bI - \bP_0)\bP_1\|_{\rm{op}} \leq \frac{1}{\Delta} \left\{ \|(\bI - \bQ_0)(\bA_0 - \bA_1)\bP_1\|_{\rm{op}} \vee \|\bQ_1(\bA_0 - \bA_1)(\bI - \bP_0)\|_{\rm{op}} \right\}.
\end{align*}
If instead we have $\Delta \equiv \sigma_{k(0)}(\bA_0) - \sigma_{k(1) + 1}(\bA_1) > 0$, then 
\begin{align*}
	\|\bP_0(\bI - \bP_1)\|_{\rm{op}} \leq \frac{1}{\Delta} \left\{ \|(\bI - \bQ_1)(\bA_0 - \bA_1) \bP_0\|_{\rm{op}} \vee \|\bQ_0(\bA_0 - \bA_1)(\bI - \bP_1)\|_{\rm{op}} \right\}.
\end{align*}
\end{lemma}

\begin{lemma}[Nishimori identity, Proposition 16 in \cite{lelarge2019fundamental}]\label{lemma:nishimori}
	Let $(\bX, \bY)$ be a couple of random variables on a polish space. Let $k \geq 1$ and let $\bx^{(1)}, \cdots, \bx^{(k)}$ be $k$ i.i.d. samples (given $\bY$) from the distribution $\P(\bX = \cdot \mid \bY)$, independently of every other random variables. Let us denote $\langle \cdot \rangle$ the expectation with respect to $\P(\bX = \cdot \mid \bY)$ and $\E$ the expectation with respect to $(\bX, \bY)$. Then for all continuous bounded function $f$, 
	\begin{align*}
		\E\langle f(\bY, \bx^{(1)}, \cdots, \bx^{(k)}) \rangle = \E\langle f(\bY, \bx^{(1)}, \cdots, \bx^{(k- 1)}, \bX) \rangle.
	\end{align*}
\end{lemma}

\begin{lemma}\label{lemma:convex-derivative}
	Let $\{f_n\}_{n \in \NN_+}$ be a sequence of convex differentiable functions on $\RR$, and $f_n(x) \rightarrow f(x)$ for all $x \in \RR$. Let $D_f = \left\{x \in \RR: f \mbox{ is differentiable at }x \right\}$, then $f_n'(x) \rightarrow f'(x)$ for all $x \in D_f$.
\end{lemma}

\begin{lemma}\label{lemma:A10}
	If $f$ and $g$ are two differentiable convex functions, then for any $b > 0$, 
	\begin{align*}
		|f'(a) - g'(a)| \leq g'(a + b) - g'(a - b) + \frac{d}{b},
	\end{align*}
	where $d = |f(a + b) - g(a + b)| + |f(a - b) - g(a - b)| + |f(a) - g(a)|$. 
\end{lemma}
\begin{proof}
	See \cite{panchenko2013sherrington}, Lemma 3.2.
\end{proof}

\begin{lemma}\label{lemma:A14}
For $\bX \in \RR^{n \times d}$, we define $\|\bX\|_1 = \sum_{i \in [n], j \in [d]}|X_{ij}|$.	
Then for $\bX_1, \bX_2 \in \RR^{n \times d}$, we have $\|\bX_1 \bX_1^{\top} - \bX_2 \bX_2^{\top}\|_1 \leq n \|\bX_1 - \bX_2\|_F(\|\bX_1\|_F + \|\bX_2\|_F)$.
\end{lemma}
\begin{proof}
	we let $\bx_i^1 \in \RR^d$ be the $i$-th row of $\bX_1$ and we let $\bx_i^2 \in \RR^d$ be the $i$-th row of $\bX_2$. Then by triangle inequality,  
\begin{align*}
	\|\bX_1 \bX_1^{\top} - \bX_2 \bX_2^{\top}\|_1  \leq & \|\bX_1(\bX_1 - \bX_2)^{\top}\|_1 + \|\bX_2(\bX_1 - \bX_2)^{\top}\|_1 \\
	\leq & \sum_{i,j \in [n]}  (\|\bx_i^1\|_2 + \|\bx_i^2\|_2) \times \|\bx_j^1 - \bx_j^2\|_2 \\
	\leq & \sum_{i \in [n]} \|\bx_i^1\|_2 \times \sqrt{n\sum_{j \in [n]}\|\bx_j^1 - \bx_j^2\|_2^2} + \sum_{i \in [n]} \|\bx_i^2\|_2 \times \sqrt{n\sum_{j \in [n]}\|\bx_j^1 - \bx_j^2\|_2^2} \\
	\leq & \sqrt{n \sum_{i \in [n]} \|\bx_i^1\|_2^2}\times \sqrt{n\sum_{j \in [n]}\|\bx_j^1 - \bx_j^2\|_2^2} + \sqrt{n \sum_{i \in [n]} \|\bx_i^1\|_2^2}\times \sqrt{n\sum_{j \in [n]}\|\bx_j^1 - \bx_j^2\|_2^2} \\
	= & n\|\bX_1 - \bX_2\|_F(\|\bX_1\|_F + \|\bX_2\|_F).
\end{align*}
\end{proof}

\section{Proofs for the strong signal regime}
\subsection{Proof of \cref{thm:strong-signal-lambda-consistent}}\label{proof-of-thm:strong-signal-lambda-consistent}
\noindent
\textbf{Proof of claim 1}

\noindent
The top $r$ eigenvectors of $\bA\bA^{\sT}$ are also the top $r$ eigenvectors of the following matrix
\begin{align*}
	\frac{1}{d}{\bA\bA^\sT} - \id_n = \frac{\blambda \bTheta^{\top} \bTheta \blambda^\sT}{nd} + \frac{1}{d \sqrt{n}}\left( \blambda \btheta^\sT\bZ^\sT + \bZ\btheta \blambda^\sT \right) + \frac{1}{d}\bZ\bZ^\sT - \id_n. 
\end{align*}
We let $\bX =  \left(\blambda \btheta^\sT\bZ^\sT + \bZ \btheta \blambda^\sT \right) / {d\sqrt{n}} + \bZ\bZ^\sT / d - \id_n$. Applying in sequence triangle inequality, \cref{lemma:concentration-of-sample-covariance} and the law of large numbers, we conclude that there exists a constant $C > 0$, such that with probability $1 - o_n(1)$
\begin{align}
	\|\bX\|_{\rm{op}} \leq & \frac{2}{d \sqrt{n}} \|\blambda \btheta^T\bZ^T\|_{\rm{op}} + \|\frac{1}{d} \bZ\bZ^T - \id_n\|_{\rm{op}} \nonumber \\
	\leq & \frac{2}{d\sqrt{n}} \|\blambda\|_F \|\bZ \btheta\|_F + \|\frac{1}{d}\bZ\bZ^T - \id_n\|_{\rm{op}} \nonumber\\
	\leq &  C \sqrt{\frac{n}{d}}. \label{eq:Xop}
\end{align}
Using \cref{lemma:wedin}, we see that there exists another  constant $\tilde{C} > 0$, such that with probability $1 - o_n(1)$ we have 
\begin{align*}
	L^{\sin}(\hat{\blambda}_s, \blambda) \leq \tilde{C} \sqrt{\frac{n}{d}}. 
\end{align*}
This completes the proof of the first claim of the theorem since by assumption $d / n \to \infty$.

\vspace{0.3cm}

\noindent
\textbf{Proof of claim 2}

\noindent

We write $ \bA\bA^{\top} /  d -\id_n = {q_{\Theta}}\bLambda\bLambda^{\top} / n + \bW$, where $\bW$ is an $n \times n$ symmetric matrix. Using \cref{eq:Xop} and the law of large numbers, we find out that $\|\bW\|_{\rm op} = o_P(1)$.

We denote the unique eigenvalues of ${q_{\Theta}} \bQ_{\Lambda}$ by $u_1 > u_2 > \cdots > u_k > 0$, where $k \leq r$. The corresponding geometric multiplicities are denoted by $s_1, \cdots, s_k \in \NN_+$. We let $\delta_l$, $\hat{\delta}_l$ be the $l$-th largest eigenvalues of $q_{\Theta} \bLambda \bLambda^{\top} / n$ and $q_{\Theta} \bLambda \bLambda^{\top} / n + \bW$, respectively. We then see immediately that  $\delta_l, \hat{\delta}_l \toP \sum_{i = 1}^ku_i \mathbbm{1}\{ \sum_{j = 1}^{i - 1} s_j + 1  \leq l \leq \sum_{j = 1}^{i} s_j \}$ as $n,d \to \infty$. Let 
\begin{align*}
	& {\bD}_i = \diag\big({\delta}_{\sum_{j = 1}^{i - 1}s_j + 1}, \cdots, {\delta}_{\sum_{j = 1}^{i}s_j}\big) \in \RR^{s_i \times s_i}, \\
	& \hat{\bD}_i = \diag\big(\hat{\delta}_{\sum_{j = 1}^{i - 1}s_j + 1}, \cdots, \hat{\delta}_{\sum_{j = 1}^{i}s_j}\big) \in \RR^{s_i \times s_i}.
\end{align*}
The above arguments imply that $\bD_i, \hat{\bD}_i \toP u_i \id_{s_i}$. 

For $i \in [k]$, we define the matrices $\bV_i, \hat\bV_i \in \RR^{n \times s_i}$, such that the columns of $\bV_i / \sqrt{n}$, $\hat\bV_i / \sqrt{n}$ are the eigenvectors of ${q_{\Theta}} \bLambda \bLambda^{\top} / n$, ${q_{\Theta}} \bLambda \bLambda^{\top} / n + \bW$ that correspond to the top $\sum_{j = 1}^{i - 1} s_j +1$ to $\sum_{j = 1}^{i} s_j $ eigenvalues, respectively. 
By Wedin's theorem (\cref{lemma:wedin}), we see that $L^{\sin}(\bV_i, \hat{\bV}_i) \toP 0$ for all $i \in [k]$. Combining all arguments derived, we conclude that 
\begin{align*}
	\Big\|\frac{1}{n}\sum_{i = 1}^r \bV_i \bD_i \bV_i^{\top} -  \frac{1}{n}\sum_{i = 1}^r \hat\bV_i \hat\bD_i \hat\bV_i^{\top}\Big\|_F^2 \toP 0.
\end{align*}
Note that $\sum_{i = 1}^r \bV_i \bD_i \bV_i^{\top} / n = q_{\Theta} \bLambda \bLambda^{\top} / n$ and $\hat\bD_i$, $\hat\bV_i$ are functions of $\bA$, thus we have found an estimator $\hat{\bL} \in \RR^{n \times n}$ such that $\|\hat{\bL} - \bLambda \bLambda^{\top}\|_F^2 / n^2 = o_P(1)$. Based on this convergence, we only need to apply a standard truncation argument to show the expected mean square error vanishes. We skip the details here for the sake of simplicity.

\vspace{0.3cm} 

\noindent
\textbf{Proof of claim 3}

\noindent
By claim 2 of the theorem, we see that there exists an estimate $\hat\bL$ of $\bLambda \bLambda^{\top} $ which achieves consistency: $\|\hat\bL - \bLambda \bLambda^{\top}\|_F / n = o_P(1)$. Let
\begin{align*}
	\bR := \argmin_{\bX \in \RR^{n \times r}} \|\hat\bL - \bX \bX^{\top}\|_F.
\end{align*}
By definition, $\|\hat\bL - \bR \bR^{\top}\|_F \leq \|\hat\bL - \bLambda \bLambda^{\top}\|_F$, thus $\|\hat\bL - \bR \bR^{\top}\|_F / {n} = o_P(1)$. By triangle inequality we see that $\|\bLambda \bLambda^{\top} - \bR\bR^{\top}\|_F / {n} = o_P(1)$, from which we conclude that there exists $\bOmega_0 \in \mathcal{O}(r)$ such that 
\begin{align*}
	\frac{1}{\sqrt{n}}\|\bR \bOmega_0 - \bLambda\|_F = o_P(1). 
\end{align*}
We define
\begin{align*}
	\bOmega_{\ast} := \argmin_{\bOmega \in \mathcal{O}(r)} W_2(\hat{p}_{\bR\bOmega}, \mu_{\Lambda}), \qquad \bR_{\ast} := \bR \bOmega_{\ast}.
\end{align*}
Denote by $\mu_{\bOmega \Lambda}$ the distribution of $\bOmega\bLambdas$ for $\bOmega \in \mathcal{O}(r)$ and $\bLambdas \sim \mu_{\Lambda}$. We then have $\|\bLambda \bOmega_0^{\sT} \bOmega_{\ast} - \bR_{\ast} \|_F / \sqrt{n} = o_P(1)$, which implies $W_2(\hat{p}_{\bLambda \bOmega_0^{\sT} \bOmega_{\ast}}, \hat{p}_{\bR_{\ast}}) = o_P(1)$. Furthermore, 
\begin{align*}
	 W_2(\hat{p}_{\bR_{\ast}}, \mu_{\Lambda})  \leq & W_2(\hat{p}_{\bR\bOmega_0}, \mu_{\Lambda}) \\
	\leq & W_2(\hat{p}_{\bR\bOmega_0}, \hat{p}_{\bLambda}) + W_2(\hat{p}_{\bLambda}, \mu_{\Lambda}) \\
	\leq & \frac{1}{\sqrt{n}}\|\bR\bOmega_0 - \bLambda\|_F + W_2(\hat{p}_{\bLambda}, \mu_{\Lambda}) = o_P(1).
\end{align*}
Invoking triangle inequality, we have
\begin{align*}
	W_2(\mu_{\bOmega_{\ast}^{\top} \bOmega_0\Lambda}, \mu_{\Lambda}) \leq & W_2(\mu_{\Lambda}, \hat{p}_{\bR_{\ast}}) + W_2( \hat{p}_{\bR_{\ast}}, \hat{p}_{\bLambda \bOmega_0^{\sT} \bOmega_{\ast}}) + W_2(\hat{p}_{\bLambda \bOmega_0^{\sT} \bOmega_{\ast}}, \mu_{\bOmega_{\ast}^{\top} \bOmega_0\Lambda}) \\
	= & W_2(\mu_{\Lambda}, \hat{p}_{\bR_{\ast}}) + W_2( \hat{p}_{\bR_{\ast}}, \hat{p}_{\bLambda \bOmega_0^{\sT} \bOmega_{\ast}}) + W_2(\hat{p}_{\bLambda }, \mu_{\Lambda}).
\end{align*}
Combining the above results, we see that $W_2(\mu_{\bOmega_{\ast}^{\top} \bOmega_0\Lambda}, \mu_{\Lambda}) = o_P(1)$.
Notice that the mapping $\bOmega \mapsto W_2(\mu_{\Lambda}, \mu_{\bOmega\Lambda})$ is continuous on $\mathcal{O}(r)$, and $W_2(\mu_1, \mu_2) = 0$ if and only if $\mu_1 = \mu_2$. Therefore, by assumption we obtain that $\|\bOmega_{\ast}^{\top}\bOmega_0- \id_r\|_F = o_P(1)$, thus $\|\bR_{\ast} - \bLambda \|_F / \sqrt{n} = o_P(1)$. Notice that $\bR$ is a function of the observation $\bA$, thus $\bR_{\ast}$ is a function of $\bA$ as well. Therefore, we have constructed a consistent estimator for $\bLambda$ under the metric of vector mean square error. The rest parts of the proof again follow from a standard truncation argument.

\subsection{Proof of Theorem \ref{thm:strong-signal-theta-lower-bound}} \label{proof-of-thm:strong-signal-theta-lower-bound}

\noindent
\textbf{Proof of claim 1}

\noindent
	We first prove \cref{eq:6}. Define $\bA_0 := (\sum_{i = 1}^n \bLambda_i\bLambda_i^\sT / n)^{-1/2} \in S_r^+$. We note that under the assumptions of remark \ref{assumption:second-moment}, with high probability $\bA_0$ is well-defined. For $j \in [d]$, we let $\bB_0^j := \frac{1}{\sqrt{n}}\bA_0^{2} \sum_{i = 1}^n A_{ij}\bLambda_i \in \RR^r$. We denote by $\bM_r$ the set of symmetric invertible matrices in $\RR^{r \times r}$. Let $\bG \sim \normal(0, \id_r)$, independent of $\bThetas \sim \mu_{\Theta}$. We define the mapping $f_{\Theta}: \bM_r \times \RR^r  \rightarrow \mathbb{R}^r$ such that 
	\begin{align*}
		f_{\Theta}(\bA, \bB) := \E\left[\bThetas \mid \bThetas + \bA\bG = \bB\right] =\frac{\int\vtheta\exp\left(-\frac{1}{2} \vtheta^\sT \bA^{-2} \vtheta + \bB^\sT \bA^{-2}\vtheta \right)\mu_{\Theta}(\dd\vtheta)}{\int\exp\left(-\frac{1}{2} \vtheta^\sT \bA^{-2} \vtheta + \bB^\sT\bA^{-2}\vtheta \right)\mu_{\Theta}(\dd\vtheta)}. 
	\end{align*}
	Let $\hat{\bTheta}_j^B := \E[\bTheta_j \mid \bA, \bLambda] = f_{\Theta}(\bA_0, \bB_0^j)$, then $\hat{\bTheta}_j^B$ achieves Bayesian mean square error.

	Dominated convergence theorem reveals  that $f_{\Theta}(\cdot, \cdot)$ is continuous. By the law of large numbers and central limit theorem, we see that $(\bA_0, \bB_0^j) \overset{d}{\rightarrow} (\bQ_{\Lambda}^{-1/2}, \bThetas + \bQ_{\Lambda}^{-1/2}\bG)$ as $n,d \to \infty$. 
	Using Skorokhod's representation theorem, there exist $(\bA_n, \bB_n^j)$ and $(\bA_{\infty}, \bB_{\infty}^j)$ being random vectors defined on the same probability space, such that $(\bA_n, \bB_n^j) \overset{a.s.}{\rightarrow} (\bA_{\infty}, \bB_{\infty}^j)$, $(\bA_n, \bB_n^j) \overset{d}{=} (\bA_0, \bB_0^j)$, and $\bA_{\infty} = \bQ_{\Lambda}^{-1/2}$, $\bB_{\infty}^j \overset{d}{=} \bThetas + \bQ_{\Lambda}^{-1/2}\bG$. Therefore, $f_{\Theta}(\bA_0, \bB_0^j) \overset{d}{=} f_{\Theta}(\bA_n, \bB_n^j) \overset{a.s.}{\rightarrow} f_{\Theta}(\bA_{\infty}, \bB_{\infty}^j)$. Since $\|f_{\Theta}(\bA_n, \bB_n^j )\|^2 \overset{d}{=} \|\E[\bTheta_j \mid \bA, \bLambda]\|^2$, 
	we conclude that the set of random variables $\left\{ \|f_{\Theta}(\bA_n, \bB_n^j)\|^2: n \in \mathbb{N}_+ \right\}$ is uniformly integrable. Therefore, we have $\|f_{\Theta}(\bA_n, \bB_n^j )\|^2 \overset{L_1}{\rightarrow} \|f_{\Theta}(\bA_{\infty}, \bB_{\infty}^j )\|^2$. This further implies that as $n,d \rightarrow \infty$
	\begin{align*}
		\E[\|\bTheta_j - \hat{\bTheta}_j^B\|_F^2] = \E_{\bThetas \sim \mu_{\Theta}}[\|\bThetas\|^2] - \E\left[\| f_{\Theta}(\bA_{n}, \bB_{n}^j ) \|^2\right] \rightarrow & \E[\|\bThetas\|^2] - \E\left[\| f_{\Theta}(\bA_{\infty}, \bB_{\infty} ) \|^2\right]\\
		  =& \E[\|\bThetas\|^2] - \E\left[ \|\E[\bThetas | \bQ_{\Lambda}^{1/2} \bThetas + \bG ] \|^2\right], 
	\end{align*}
	thus completing the proof of the first claim.

\vspace{0.3cm} 

\noindent
\textbf{Proof of claim 2}

\noindent	
Next, we prove \cref{eq:7}.	For $k,j \in [d]$, $k \neq j$, notice that $\bTheta_k$ and $\bTheta_j$ are conditionally independent conditioning on $(\bA, \bLambda)$. Then we have $\E[\bTheta_k^{\top}\bTheta_j \mid \bA, \bLambda] = \E[\bTheta_k \mid \bA, \bLambda]^{\top} \E[\bTheta_j \mid \bA, \bLambda] = f_{\Theta}(\bA_0, \bB_0^k)^{\top}f_{\Theta}(\bA_0, \bB_0^j)$, thus
	\begin{align*}
		\E[(\bTheta_k^{\top}\bTheta_j - \E[\bTheta_k^{\top}\bTheta_j \mid \bA, \bLambda] )^2] = rq_{\Theta}^2 - \E[(f_{\Theta}(\bA_0, \bB_0^k)^{\top}f_{\Theta}(\bA_0, \bB_0^j))^2].
	\end{align*}
By the law of large numbers and the central limit theorem, we have $(\bA_0, \bB_0^j, \bB_0^k) \overset{d}{\rightarrow} (\bQ_{\Lambda}^{-1/2}, \bTheta_1 + \bQ_{\Lambda}^{-1/2} \bG_1, \bTheta_2 + \bQ_{\Lambda}^{-1/2} \bG_2 )$, where $\bTheta_1, \bTheta_2 \sim \mu_{\Theta}$, $\bG_1, \bG_2 \sim \normal(0, \id_r)$ are mutually independent. By Skorokhod's representation theorem, there exist $(\bA_n, \bB_n^j, \bB_n^k)$ and $(\bA_{\infty}, \bB_{\infty}^j, \bB_{\infty}^k)$ being random vectors on the same probability space, such that $(\bA_n, \bB_n^j, \bB_n^k) \overset{a.s.}{\rightarrow} (\bA_{\infty}, \bB_{\infty}^j, \bB_{\infty}^k)$, $(\bA_n, \bB_n^j, \bB_n^k) \overset{d}{=} (\bA_0, \bB_0^j, \bB_0^k)$, and $\bA_{\infty} = \bQ_{\Lambda}^{-1/2}$, $(\bB_{\infty}^j, \bB_{\infty}^k) \overset{d}{=} (\bTheta_1 + \bQ_{\Lambda}^{-1/2}\bG_1, \bTheta_2 + \bQ_{\Lambda}^{-1/2}\bG_2)$. Therefore, as $n,d \to \infty$
	\begin{align*}
		& (f_{\Theta}(\bA_0, \bB_0^k)^{\top}f_{\Theta}(\bA_0, \bB_0^j))^2 \\
		 \overset{d}{=} & (f_{\Theta}(\bA_n, \bB_n^k)^{\top} f_{\Theta}(\bA_n, \bB_n^j))^2 \overset{a.s.}{\rightarrow} (f_{\Theta}(\bA_{\infty}, \bB_{\infty}^k)^{\top}f_{\Theta}(\bA_{\infty}, \bB_{\infty}^j))^2.
	\end{align*}
	Notice that 
	\begin{align*}
		(f_{\Theta}(\bA_0, \bB_0^k)^{\top}f_{\Theta}(\bA_0, \bB_0^j))^2 \leq &  \|f_{\Theta}(\bA_0, \bB_0^k)\|^2 \|f_{\Theta}(\bA_0, \bB_0^j)\|^2 \\
		\leq & \E[\|\bTheta_k\|^2 \mid \bA, \bLambda] \E[\|\bTheta_j\|^2 \mid \bA, \bLambda] \\
		= & \E[\|\bTheta_k\|^2\|\bTheta_j\|^2 \mid \bA, \bLambda].
	\end{align*}
	Therefore, the set of random variables $\{(f_{\Theta}(\bA_0, \bB_0^k)^{\top}f_{\Theta}(\bA_0, \bB_0^j))^2: n \in \NN_+\}$ is uniformly integrable. This further implies that $(f_{\Theta}(\bA_0, \bB_0^k)^{\top}f_{\Theta}(\bA_0, \bB_0^j))^2 \overset{L_1}{\rightarrow} (f_{\Theta}(\bA_{\infty}, \bB_{\infty}^k)^{\top}f_{\Theta}(\bA_{\infty}, \bB_{\infty}^j))^2$, thus as $n,d \rightarrow \infty$
	\begin{align*}
		\E[(\bTheta_k^{\top}\bTheta_j - \E[\bTheta_k^{\top}\bTheta_j \mid \bA, \bLambda] )^2] \rightarrow rq_{\Theta}^2 - \Big\|\E\left[ \E[\bThetas \mid  \bQ_{\Lambda}^{1/2} \bThetas + \bG]\E[\bThetas \mid \bQ_{\Lambda}^{1/2} \bThetas + \bG]^{\top} \right]\Big\|_F^2,
	\end{align*} 
	which concludes the proof of the second claim of the theorem. 

\subsection{Proof of Theorem \ref{thm:strong-signal-theta-achieve-lower-bound}}\label{proof:thm:strong-signal-theta-achieve-lower-bound}

By \cref{thm:strong-signal-lambda-consistent} claim 3, we see that there exists estimate $\hat{\bLambda}$ of $\bLambda$, such that $\|\bLambda - \hat{\bLambda}\|_F / \sqrt{n} \toP 0$. Notice with high probability $\|\bA\|_{\rm op} \leq C \sqrt{d} $ for some constant $C > 0$ that depends uniquely on $(\mu_{\Lambda}, \mu_{\Theta})$, we then conclude that 
\begin{align*}
	\frac{1}{\sqrt{nd}} \| \bA^{\top} \hat\bLambda -  \bA^{\top} \bLambda\|_F \leq \frac{\|\bA\|_{op}}{\sqrt{d}} \cdot \frac{\|\bLambda - \hat{\bLambda}\|_F}{\sqrt{n}} = o_P(1). 
\end{align*}
Since $\bZ$ is independent of $\bLambda$, we immediately see that there exists $\bg \in \RR^{d \times r}$ that has i.i.d. standard Gaussian entries and is independent of $(\bLambda, \bTheta)$, such that
\begin{align}\label{eq:Theta-g}
	\frac{1}{\sqrt{d}}\Big\| \frac{1}{\sqrt{n}} \bA^{\top} \hat{\bLambda}\bQ_{\Lambda}^{-1/2} - \bTheta \bQ_{\Lambda}^{1/2} - \bg  \Big\|_F = o_P(1). 
\end{align}

\subsubsection*{Proof of the first result}

We let $\bG \sim \normal(0, \id_r)$, $\bThetas \sim \mu_{\Theta}$, independent of each other. Define the mapping $F: \RR^r \rightarrow \RR^r$, such that
	\begin{align*}
		F(\by) := \E[\bThetas \mid \bQ_{\Lambda}^{1/2}\bThetas + \bG = \by].
	\end{align*}
Dominated convergence theorem straightforwardly implies that $F$ is continuous on $\RR^r$. Therefore, for any $w \in (0,1)$, we see that there exists a mapping $F^w: \RR^r \to \RR^r$, such that $F^w$ is Lipschitz continuous. In addition, 
\begin{align*}
	& \E\left[ \| F(\bQ_{\Lambda}^{1/2} \bThetas + \bG) - F^w(\bQ_{\Lambda}^{1/2} \bThetas + \bG)  \|^2 \right] \leq w^2
\end{align*}
We denote the Lipschitz constant of $F^w$ by $L_w > 0$. 
Let $\bg_i \in \RR^r$ be the $i$-th row of $\bg$. The law of large numbers gives the following convergence:
\begin{align*}
	\frac{1}{d}\sum_{i = 1}^d\| \bTheta_i - F( \bQ_{\Lambda}^{1/2}\bTheta_i + \bg_i)\|^2 \toP \E[\|\bThetas\|^2] -  \E\left[ \left\| \E[\bThetas \mid \bQ_{\Lambda}^{1/2} \bThetas + \bG ] \right\|^2 \right].
\end{align*}
Therefore, as $n, d \to \infty$
\begin{align*}
	\left|\frac{1}{d}\sum_{i = 1}^d\| \bTheta_i - F^w( \bQ_{\Lambda}^{1/2}\bTheta_i + \bg_i)\|^2 - \E[\|\bThetas\|^2] + \E\left[ \left\| \E[\bThetas \mid \bQ_{\Lambda}^{1/2} \bThetas + \bG ] \right\|^2 \right] \right| \leq Cw + o_P(1),
\end{align*}
where $C > 0$ is a constant depending only on $\mu_{\Theta}$. We denote by $\bv_i$ the $i$-th row of $ \bA^{\top} \hat{\bLambda}\bQ_{\Lambda}^{-1/2} / \sqrt{n}$. By assumption we have 
\begin{align*}
	 \frac{1}{d} \sum_{i = 1}^d \|F^w( \bQ_{\Lambda}^{1/2}\bTheta_i + \bg_i) - F^w(\bv_i) \|^2 \leq \frac{L_w^2}{d} \sum_{i = 1}^d \| \bQ_{\Lambda}^{1/2} \bTheta_i + \bg_i - \bv_i \|^2,
\end{align*}
which is $o_P(1)$ according to \cref{eq:Theta-g}. Combining the above analysis, we conclude that for any $w \in (0,1)$, there exists $n_w \in \NN_+$, such that for $n \geq n_w$, there exists estimator $\hat{\bTheta}_w \in \RR^{d \times r}$, such that with probability at least $1 - w$
\begin{align*}
	\frac{1}{d} \sum_{i = 1}^d \|\bTheta_i - F^w(\bv_i) \|^2 \leq \E[\|\bThetas\|^2] -  \E\left[ \left\| \E[\bThetas \mid \bQ_{\Lambda}^{1/2} \bThetas + \bG ] \right\|^2 \right] + 2w + o_P(1), 
\end{align*}
Since $w$ is arbitrary, the rest parts of the proof follow from a simple truncation argument.

\subsubsection*{Proof of the second result}

By analyzing the second moment we obtain that
\begin{align*}
	& \frac{1}{d^2} \sum_{i,j \in [d]} |\bTheta_i^{\top} \bTheta_j - F(\bQ_{\Lambda}^{1/2} \bTheta_i +  \bg_i)^{\top} F(\bQ_{\Lambda}^{1/2} \bTheta_j +  \bg_j)|^2 \\
	= &\, rq_{\Theta}^2 - \Big\| \E\left[\E[\bThetas \mid \bQ_{\Lambda}^{1/2} \bThetas + \bG] \E[\bThetas \mid \bQ_{\Lambda}^{1/2} \bThetas + \bG]^{\top} \right] \Big\|_F^2 + o_P(1).
\end{align*}
Since $F$ is continuous, then for any $w \in (0,1)$, there exists $\tilde F^w: \RR^{2r} \to \RR$ such that $\tilde F^w$ is $\tilde L_w$-Lipschitz continuous. Furthermore, 
\begin{align*}
	\E\left[ |F(\bQ_{\Lambda}^{1/2} \bTheta_i +  \bg_i)^{\top} F(\bQ_{\Lambda}^{1/2} \bTheta_j +  \bg_j) - \tilde F^w(\bQ_{\Lambda}^{1/2} \bTheta_i +  \bg_i, \bQ_{\Lambda}^{1/2} \bTheta_j +  \bg_j)|^2  \right] \leq w^2. 
\end{align*}
Again through analysis of the second moment we have
\begin{align*}
	& \frac{1}{d^2} \sum_{i,j \in [d]} |\bTheta_i^{\top} \bTheta_j - \tilde F^w(\bQ_{\Lambda}^{1/2} \bTheta_i +  \bg_i, \bQ_{\Lambda}^{1/2} \bTheta_j +  \bg_j)|^2 \\
	\leq & rq_{\Theta}^2 - \Big\| \E\left[\E[\bThetas \mid \bQ_{\Lambda}^{1/2} \bThetas + \bG] \E[\bThetas \mid \bQ_{\Lambda}^{1/2} \bThetas + \bG]^{\top} \right] \Big\|_F^2 + \tilde Cw + o_P(1),
\end{align*} 
where $\tilde C > 0$ is a constant depending uniquely on $\mu_{\Theta}$. By Lipschitzness we have
\begin{align*}
	& \frac{1}{d^2} \sum_{i,j \in [d]} |\tilde F^w(\bQ_{\Lambda}^{1/2} \bTheta_i +  \bg_i, \bQ_{\Lambda}^{1/2} \bTheta_j +  \bg_j) - \tilde F^w(\bv_i, \bv_j)|^2 \\
	\leq & \frac{\tilde L_w^2}{d^2} \sum_{i,j \in [d]} \left\{ \|\bQ_{\Lambda}^{1/2} \bTheta_i +  \bg_i - \bv_i\|^2 + \|\bQ_{\Lambda}^{1/2} \bTheta_j +  \bg_j - \bv_j\|^2 \right\},
\end{align*}
which by \cref{eq:Theta-g} is $o_P(1)$. Since $w$ is arbitrary, again the claim follows by applying standard truncation argument.

\section{Proof outlines for the weak signal regime}
\subsection{Proof of  \cref{thm:weak-signal-theta}}\label{proof:thm:weak-signal-theta}
	Assume $\blambda$ is given, then for any $j \in [d]$, the posterior distribution of $\bTheta_j$ given $(\bA, \bLambda)$ can be expressed as 
	\begin{align*}
		p(\dd\vtheta_j | \blambda, \bA) \propto \exp\left( -\frac{1}{2\sqrt{nd}}\sum\limits_{i = 1}^n \langle \bLambda_i, \vtheta_j \rangle^2 + \frac{1}{\sqrt[4]{nd}}\sum\limits_{i = 1}^n  A_{ij} \langle \bLambda_i, \vtheta_j \rangle^2 \right) \mu_{\Theta}(\dd\vtheta_j).
	\end{align*}
	From the above equation we see that the posterior of $\bTheta$ given $(\bA, \bLambda)$ is a product distribution over $\RR^d$, thus greatly simplifies the analysis. The rest of the proof is similar to that of Theorem \ref{thm:strong-signal-theta-lower-bound}, and we skip it for simplicity.  

\subsection{Proof outline of \cref{thm:free-energy}}\label{sec:proof-of-thm-lower-bound-lower-bound}

In this section we outline the proof of  \cref{thm:free-energy}. We leave the proofs of technical lemmas to Appendix \ref{sec:tec-lemmas-lower-bound}. For the sake of simplicity, here we consider only $r = 1$. We comment that cases with $r \geq 2$ can be proven similarly. 

\subsubsection{Free energy density}

Note that the posterior distributions that correspond to the symmetric and asymmetric models can be expressed as follows:
\begin{align*}
	& \dd\P(\bLambda = \vlambda \mid \bY) = \frac{e^{H_{s,n}(\vlambda)} \tensorl}{\int e^{H_{s,n}(\vlambda)} \tensorl }, \\
	& \dd\P(\bLambda = \vlambda, \bTheta = \vtheta \mid \bA) = \frac{e^{H_n(\vlambda, \vtheta)} \tensort\tensorl}{\int e^{H_n(\vlambda, \vtheta)} \tensort\tensorl},
\end{align*} 
where $ \tensorl$ ($\tensort$) is the product distribution over $\RR^n$ ($\RR^d$) with each coordinate having marginal distribution $\mu_{\Lambda}$ ($\mu_{\Theta}$), and $H_{s,n}, H_n$ are the Hamiltonians that correspond to models \eqref{model:symmetric} and \eqref{model:weak-signal}, respectively:
\begin{align}\label{eq:Hs-and-H}
\begin{split}
	& H_{s,n}(\vlambda ) := \, \frac{q_{\Theta}^2}{2n}\langle \bLambda, \vlambda \rangle^2 + \frac{q_{\Theta}}{2}\vlambda^\intercal\bW \vlambda - \frac{q_{\Theta}^2}{4n}\|\vlambda\|^4, \\
	& H_n(\vlambda, \vtheta) := \frac{1}{\sqrt{nd}}\langle \bLambda, \vlambda \rangle \langle \bTheta, \vtheta \rangle + \frac{1}{\sqrt[4]{nd}}  \vlambda^{\top}\bZ \vtheta - \frac{1}{2\sqrt{nd}}\|\vlambda\|^2 \|\vtheta\|^2.
\end{split}
\end{align}
Following the terminology of statistical mechanics, the \emph{free energy density} is defined as the expected log-partition function (also known as log normalizing constant):
\begin{align*}
	& \Psi_n^s := \frac{1}{n} \E \log \int e^{H_{s,n}(\vlambda) } \tensorl, \\
	& \Psi_n := \frac{1}{n} \E \log \int e^{H_n(\vlambda, \vtheta)} \tensorl \tensort.  
\end{align*}
The lemma below connects free energy densities with the corresponding mutual informations. 
\begin{lemma}\label{lemma:I-F}
	The following equations hold:
	\begin{align*}
		& \Psi_n^s = \frac{q_{\Theta}^2\E[\bLambdas^2]^2}{4} -\,\Infosy_n(\mu_{\Lambda};q_{\Theta})  + o_n(1), \\
		& \Psi_n = \frac{q_{\Theta}^2\E[\bLambdas^2]^2}{4} -\, \Infoas_n(\mu_{\Lambda},\mu_{\Theta}) + o_n(1).
	\end{align*}
\end{lemma}
\begin{proof}
	By definition, the mutual information that corresponds to the symmetric model can be reformulated as 
	\begin{align*}
		\Infosy_n(\mu_{\Lambda};q_{\Theta}) = & \frac{1}{n}\E \left\{ \log \frac{\dd \mu_{\Lambda}^{\otimes n}(\bLambda) \cdot \exp(H_{s,n}(\bLambda))}{\dd \mu_{\Lambda}^{\otimes n}(\bLambda) \cdot \int \exp(H_{s,n}(\vlambda)) \dd \mu_{\Lambda}^{\otimes n}(\vlambda)} \right\} \\
		= & \frac{q_{\Theta}^2\E[\bLambdas^2]^2}{4} - \Psi_n^s. 
	\end{align*}
	The asymmetric mutual information is a slightly more complicated, which we write below
	\begin{align*}
		& \Infoas_n(\mu_{\Lambda},\mu_{\Theta}) \\ = & \frac{1}{n} \E\left\{ \log \frac{\dd \mu_{\Lambda}^{\otimes n}(\bLambda) \cdot \int \exp(H_n(\bLambda, \vtheta)) \dd\mu_{\Theta}^{\otimes n}(\vtheta)}{\dd \mu_{\Lambda}^{\otimes n}(\bLambda) \cdot \int \exp(H_n(\vlambda, \vtheta))  \dd \mu_{\Lambda}^{\otimes n}(\vlambda) \dd \mu_{\Theta}^{\otimes n}(\vtheta)} \right\} \\
		= & \frac{1}{n} \sum_{i = 1}^d \E\Big\{  \log \int {\exp\left( \frac{1}{\sqrt{nd}} \|\bLambda\|^2 \bTheta_i \vtheta_i + \frac{1}{\sqrt[4]{nd}} \langle \bZ_{\cdot i}, \bLambda\rangle \vtheta_i - \frac{1}{2\sqrt{nd}}\|\bLambda\|^2 \vtheta_i^2 \right)}  \mu_{\Theta}(\dd \vtheta_i)\Big\} - \Psi_n \\
		= & \frac{d}{n}\, \E\left\{ \log \int \exp\left( \frac{1}{\sqrt{nd}}\|\bLambda\|^2 \bTheta_1 \vtheta_1 + \frac{1}{\sqrt[4]{nd}} \langle \bZ_{\cdot 1}, \bLambda\rangle \vtheta_1 - \frac{1}{2\sqrt{nd}}\|\bLambda\|^2 \vtheta_1^2  \right)\mu_{\Theta}(\dd \vtheta_1) \right\} - \Psi_n.  
	\end{align*}
	Define
	\begin{align*}
		F(q) = \E\left\{ \log \int \exp\left( q \bThetas \vtheta + \sqrt{q} \bG \vtheta - \frac{1}{2} q\vtheta^2  \right)\mu_{\Theta}(\dd \vtheta) \right\},
	\end{align*}
	where the expectation is taken over $\bThetas \sim \mu_{\Theta}, \bG \sim \normal(0,1)$ that are independent of each other. Applying Stein's lemma, we obtain that for $q > 0$, $F(q)$ has second order continuous derivatives satisfying
	\begin{align*}
		F'(q) 
		=& \frac{1}{2} \E \left\{ \E[\bThetas \mid \sqrt{q} \bThetas + \bG]^2\right\}, \\
		F''(q) = & \E\left\{ \frac{1}{2} \E[\bThetas^2 \mid \sqrt{q} \bThetas + \bG]^2   + \frac{1}{2}  \E[\bThetas \mid \sqrt{q} \bThetas + \bG]^4   -   \E[\bThetas \mid \sqrt{q} \bThetas + \bG]^2  \E[\bThetas^2 \mid \sqrt{q} \bThetas + \bG]\right\}.
	\end{align*}
	Since $\mu_{\Theta}$ has mean zero, we conclude that $F$ also has second order continuous derivatives at zero, and $F'(0) = 0$, $F''(0)  = q_{\Theta}^2 / 2$. These arguments imply that
	\begin{align*}
		\Infoas_n(\mu_{\Lambda},\mu_{\Theta}) = \frac{d}{n}\E\left\{ F\left(\frac{1}{\sqrt{nd}}\|\bLambda\|^2 \right) \right\} - \Psi_n = \frac{q_{\Theta}^2\E[\bLambdas^2]^2}{4} - \Psi_n + o_n(1),
	\end{align*}
	which concludes the proof of the lemma.
	
	%
	%
	%
	%
\end{proof}
From \cref{lemma:I-F} we see that in order to prove the theorem, it suffices to show that the free energy densities agree asymptotically:
\begin{align}\label{eq:Psis-Psi}
	\lim_{n \to \infty} \Psi_n^s = \lim_{n,d \to \infty} \Psi_n.
\end{align}

\subsubsection{Asymptotic equivalence of free energy densities}

We then proceed to prove \cref{eq:Psis-Psi}. We will start with the additional constraint that $\mu_{\Lambda}$ has bounded support. Later in Appendix \ref{section:reduction-to-bdd-support}, we show that proofs for general $\mu_{\Lambda}$ can be reduced to the bounded case. 
\begin{assumption}\label{assumption:bounded-support}
	We assume that support$\,(\mu_{\Lambda}) \subseteq [-K, K]$, with $K > 0$ being some fixed constant that is independent of $n,d$. 
\end{assumption}
%
%
%
For $h,s \geq 0$, we define the perturbations
\begin{align*}
	& \bY'(h) = \frac{\sqrt{h}}{n} \bLambda \bLambda^{\top} + \bW', \\
	& \bx'(s) = \sqrt{s} \bLambda + \bg',
\end{align*}
where $\bW' \overset{d}{=}  \GOE(n)$ and $\bg' \overset{d}{=} \normal(0, \id_n)$. Furthermore, we require that $(\bW', \bg', \bLambda, \bTheta, \bZ, \bW)$ are mutually independent. We define the Hamiltonians associated with the perturbations $\bY'(h)$ and $\bx'(s)$ respectively as follows:
\begin{align}
	 & H_n(\vlambda ; \bY'(h)) := \, \frac{h}{2n}\langle \bLambda, \vlambda \rangle^2 + \frac{\sqrt{h}}{2}\vlambda^\sT\bW' \vlambda - \frac{h}{4n}\|\vlambda\|^4, \label{eq:Hamiltonian-Yp} \\
	 & H_n(\vlambda ; \bx'(s)) := \, \sqrt{s} \langle \vlambda, \bg' \rangle + s\langle \bLambda, \vlambda \rangle - \frac{s}{2}\|\vlambda\|^2.\label{eq:Hamiltonian-xp}
\end{align}
The posterior distribution of $\bLambda$ given $(\bA, \bY'(h), \bx'(s))$ can be expressed as 
\begin{align*}
	\mu(\dd \vlambda) = \frac{1}{Z_n(h,s)} \tensorl \int \exp\left(H_n(\vlambda, \vtheta) + H_n(\vlambda ; \bY'(h)) + H_n(\vlambda ; \bx'(s)) \right)\tensort,
\end{align*}
where $Z_n(h,s)$ is the normalizing constant:
\begin{align*}
	Z_n(h,s) = \int \exp\left(H_n(\vlambda, \vtheta) + H_n(\vlambda ; \bY'(h)) + H_n(\vlambda ; \bx'(s)) \right)\tensort \tensorl.
\end{align*}
Note that $Z_n(h,s)$ is random and depends on $(\bA, \bY'(h), \bx'(s))$. We define the  free energy density that corresponds to observations $(\bA, \bY'(h), \bx'(s))$ as 
\begin{align}\label{eq:C25-1}
	\Phi_n(h,s) := \frac{1}{n}\E\left[ \log Z_n(h,s) \right]. 
\end{align}
%
The next equation follows from Gaussian integration by parts and Nishimori identity (\cref{lemma:nishimori}):
\begin{align}\label{eq:C25-2}
	\frac{\partial}{ \partial h} \Phi_n(h,s) = \frac{1}{4n^2}\E\left[ \langle \bLambda \bLambda^{\top}, \E[ \bLambda \bLambda^{\top} \mid \bA, \bY'(h), \bx'(s) ]\rangle \right].
\end{align}
\cref{eq:C25-2} holds for all   $h > 0, s \geq 0$, and is directly related to the MMSE in the perturbed model. 
The rest parts of the proof will be devoted to proving convergence of $\Phi_n(h,s)$ as $n,d \rightarrow \infty$. 

To this end, we first show that asymptotically speaking, the free energy density depends on $\mu_{\Theta}$ only through its second moment. More precisely, we can replace $\mu_{\Theta}$ with a Gaussian distribution which has mean zero and variance $q_{\Theta}$. This vastly simplifies further computation. 

\begin{lemma}\label{lemma:free-energy-1}
For $k \in [d]$, we let $P_{\Theta, k}$ be a distribution over $\RR^d$ with independent coordinates, such that $(\vtheta_1, \vtheta_2, \cdots, \vtheta_d) \sim P_{\Theta, k}$ if and only if $\vtheta_1, \cdots, \vtheta_k \iidsim \mu_{\Theta}$ and $\vtheta_{k + 1}, \cdots, \vtheta_d \iidsim \normal(0,q_{\Theta})$. We define 
\begin{align*}
	\Phi_n^{(k)}(h,s) := \frac{1}{n} \E\left[ \log \left( \int \exp(H_n(\vlambda, \vtheta)+ H_n(\vlambda ; \bY'(h)) + H_n(\vlambda ; \bx'(s))) \tensorl \dd P_{\Theta, k}(\vtheta) \right) \right].
\end{align*}	
In the above expression, the expectation is taken over $(\bLambda, \bTheta, \bZ, \bW', \bg')$. Notice that by definition $\Phi_n(h,s) = \Phi_n^{(d)}(h,s)$. Then under the conditions of \cref{thm:free-energy} and in addition Assumption \ref{assumption:bounded-support}, as $n,d \rightarrow \infty$ we have $\Phi_n(h,s) - \Phi_n^{(0)}(h, s) = o_n(1)$ for all fixed $h,s \geq 0$. 
\end{lemma}
\cref{lemma:free-energy-1} can be proved via a Lindeberg type argument, and we postpone the details to Appendix \ref{sec:proof-free-energy-density1}. According to \cref{lemma:free-energy-1}, in order to derive the limiting expression of $\Phi_n(h,s)$, it suffices to compute the limit of $\Phi_n^{(0)}(h,s)$ instead, which can be done via Gaussian integration techniques. 

\begin{lemma}\label{lemma:free-energy-2}
For fixed $h,s \geq 0$, we define 
\begin{align*}
	\tilde{H}_n(\vlambda; \bY'(h), \bx'(s)) := & \frac{q_{\Theta}^2}{2n}\langle \bLambda, \vlambda \rangle^2 + \frac{q_{\Theta}}{2\sqrt{nd}} \|\bZ^{\top} \vlambda\|^2 - \frac{dq_{\Theta}}{2\sqrt{nd}} \|\vlambda\|^2 - \frac{q_{\Theta}^2}{4n}\|\vlambda\|^4 \\
	& + H_n(\vlambda ; \bY'(h)) + H_n(\vlambda ; \bx'(s)),  \\
	\tilde{\Phi}_n(h,s) := & \frac{1}{n} \E\Big[ \log \Big( \int \exp\Big(\tilde{H}_n(\vlambda; \bY'(h), \bx'(s)) \Big) \tensorl \Big) \Big].
\end{align*}
Then under the conditions of \cref{thm:free-energy} and Assumption \ref{assumption:bounded-support}, as $n,d \rightarrow \infty$, we have $\tilde{\Phi}_n(h,s) - \Phi_n^{(0)}(h,s) = o_n(1)$.
\end{lemma}
We defer the proof of \cref{lemma:free-energy-2} to Appendix \ref{sec:proof-free-energy-density2}. Under the asymptotics $n,d \rightarrow \infty$, $d / n \to \infty$, according to \cite{bai1988convergence}, the matrix $\left( \bZ\bZ^\sT - d\id_n \right) / \sqrt{nd}$ behaves like a GOE$(n)$ matrix. Replacing $\left( \bZ\bZ^\sT - d\id_n \right) / \sqrt{nd}$ with a $\GOE(n)$ matrix in the definition of $\tilde{\Phi}_n(h,s)$, we see that this allows us to approximate $\tilde{\Phi}_n(h,s)$ via the  free energy density of the symmetric model \eqref{model:symmetric}. Such heuristics can be made rigorous via the following lemma:
\begin{lemma}\label{lemma:free-energy-3}
	Recall that $\bY$ is defined in \cref{model:symmetric}.  For $h,s \geq 0$, we define the free energy density $\Phi_n^Y(h,s)$ that corresponds to the observations $(\bY, \bY'(h), \bx'(s))$ as 
	\begin{align}\label{eq:26}
	& \Phi_n^Y(h,s) := \frac{1}{n} \E\left[ \log \left( \int \exp\left( H_n^Y(\vlambda) + H_n(\vlambda ; \bY'(h)) + H_n(\vlambda ; \bx'(s)) \right)\tensorl \right) \right], \\ \nonumber 
	& H_n^Y(\vlambda) := \frac{q_{\Theta}^2}{2n} \langle \bLambda, \vlambda \rangle^2 + \frac{q_{\Theta}}{2} \vlambda^\sT\bW\vlambda - \frac{q_{\Theta}^2}{4n}\|\vlambda\|^4.
\end{align}
Then under the conditions of \cref{thm:free-energy} and Assumption \ref{assumption:bounded-support}, as $n,d \rightarrow \infty$, we have $\Phi_n^Y(h,s) - \tilde{\Phi}_n(h,s) = o_n(1)$.
\end{lemma}
We defer the proof of \cref{lemma:free-energy-3} to Appendix \ref{sec:proof-free-energy-density3}. 
Combining \cref{lemma:free-energy-1,lemma:free-energy-2,lemma:free-energy-3}, we conclude that as $n, d \rightarrow \infty$, for all fixed $h,s \geq 0$, we have $\Phi_n(h,s) - \Phi_n^Y(h,s) = o_n(1)$. This relates the asymmetric model  to the symmetric model through their free energy densities. The following lemma summarizes this result and lists several additional useful properties for future reference. 
\begin{lemma}\label{lemma:free-energy-4}
	Under the conditions of \cref{thm:free-energy} and Assumption \ref{assumption:bounded-support}, for all fixed $h,s \geq 0$, the following claims hold:
	\begin{enumerate}
		\item As $n,d \rightarrow \infty$,  we have $\Phi_n(h,s) = \Phi^Y_n(h,s) + o_n(1)$.
		\item The following mappings $x \mapsto \Phi_n(h, x)$, $x \mapsto \Phi_n(x, s)$, $x \mapsto \Phi_n^Y(h, x)$, $x \mapsto \Phi_n^Y(x, s)$ are all convex on $[0, \infty)$ and differentiable on $(0, \infty)$.
		\item  $\lim_{n \rightarrow \infty}\Phi^Y_n(0,s)$ exists for all 
		$$(q_{\Theta}^2,s) \in \{(tx, (1 - t)xq^{\ast}(x)): x \geq 0, q^{\ast}(x) \mbox{ exists and is unique, }  t \in [0,1]\},$$ where $\cF(\cdot, \cdot)$ is defined in \cref{eq:FsQ} and
		\begin{align}\label{qstar}
			q^{\ast}(x) := \argmax_{q \geq 0} \cF(x,q).
		\end{align} 
	\end{enumerate}
\end{lemma}
\begin{remark}
	By \cite[Proposition 17]{lelarge2019fundamental}, $q^{\ast}(x)$ exists and is unique for all but countably many $x > 0$. Claims 2 and 3 do not rely on Assumption \ref{assumption:bounded-support}.
\end{remark}
We delay the proof of \cref{lemma:free-energy-4}  to Appendix \ref{sec:proof-free-energy-density4}. We note that \cref{thm:free-energy} is an immediate consequence of \cref{lemma:free-energy-4}.

\subsection{Proof of \cref{thm:lower-bound}}\label{sec:proof-of-thm:lower-bound}

In this section, we will apply \cref{lemma:free-energy-4} to prove \cref{thm:lower-bound}. Using Lemma \ref{lemma:nishimori} and Gaussian integration by parts, for all $h > 0$ we have
\begin{align}\label{eq:C27}
\begin{split}
	& \frac{\partial}{\partial h}\Phi_n(h,0) = \frac{1}{4n^2} \E \left[ \langle \bLambda \bLambda^{\top}, \E[\bLambda \bLambda^{\top} \mid \bA, \bY'(h)] \right] , \\
	&  \frac{\partial}{\partial h}{\Phi_n^{Y }}(h, 0) = \frac{1}{4n^2} \E \left[ \langle \bLambda \bLambda^{\top}, \E[\bLambda \bLambda^{\top} \mid \bY, \bY'(h)]\rangle \right]. 
\end{split}
\end{align} 
%
%
Recall that $\cF$ is defined in \cref{eq:FsQ}. We let 
\begin{align}\label{eq:setD}
	D := \left\{s > 0 \mid \cF(s, \cdot) \mbox{ has a unique maximizer } q^{\ast}(s) \right\}.
\end{align}
 By proposition 17 in \cite{lelarge2019fundamental}, $D$ is equal to $(0, +\infty)$ minus a countable set, and is precisely the set of $s > 0$ at which the function $\phi: s \mapsto \sup_{q \geq 0}\cF(s, q)$ is differentiable. Furthermore, by \cite[Theorem 13]{lelarge2019fundamental}, for all $h \geq 0$, 
 \begin{align}\label{eq:PhiY-cF}
 	\lim_{n \rightarrow \infty}\Phi_n^Y(h, 0) = \sup_{q \geq 0} \cF(q_{\Theta}^2 + h, q).
 \end{align}
%
By the first claim of \cref{lemma:free-energy-4}, $\Phi_n(h,0) = \Phi_n^Y(h,0) + o_n(1)$, thus $\lim_{n \rightarrow \infty}\Phi_n(h, 0) = \sup_{q \geq 0} \cF(q_{\Theta}^2 + h, q)$. By the second claim of \cref{lemma:free-energy-4}, the mappings $h \mapsto \Phi_n(h,0)$, $h \mapsto \Phi_n^Y(h,0)$ are convex and differentiable on $(0, \infty)$. Next, we apply \cref{lemma:convex-derivative} to function sequences $\{h\mapsto \Phi_n(h,0)\}_{n \geq 1}$, $\{h\mapsto \Phi_n^Y(h,0)\}_{n \geq 1}$, and conclude that  for all but countably many values of $h > 0$, 
\begin{align}\label{eq:Phi'-phi'}
	\lim_{n \rightarrow \infty}\frac{\partial}{\partial h}\Phi_n(h, 0) = \lim_{n \rightarrow \infty} \frac{\partial}{\partial h}{\Phi_n^Y}(h, 0) = \phi'(h + q_{\Theta}^2).
\end{align}
From the above equation we see that the mapping $\lambda \mapsto \phi'(\lambda)$ is non-decreasing on $D$. Therefore, for all but countably many $q_{\Theta} > 0$, $\phi'$ is continuous at $q_{\Theta}^2 \in D$. For these $q_{\Theta}$, we immediately see that for any $\ep > 0$, there exists $h_{\ep} > 0$ depending uniquely on $(q_{\Theta}, \ep, \mu_{\Lambda})$, such that $\phi'(q_{\Theta}^2 + h_{\eps}) \leq \phi'(q_{\Theta}^2) + \eps$, and $\phi$ is differentiable at $q_{\Theta}^2 + h_{\eps}$. According to \cref{eq:Phi'-phi'}, there exists $n_{\eps} \in \NN_+$, such that for all $n \geq n_{\eps}$, 
$$\Big|\frac{\partial}{\partial h}\Phi_n(h_{\eps}, 0) - \phi'(h_{\eps} + q_{\Theta}^2)\Big| \leq \eps. $$
According to \cref{eq:C27},
\begin{align*}
	\frac{\partial}{\partial h}\Phi_n(h, 0) \geq & \frac{1}{4n^2} \E \left[ \langle \bLambda \bLambda^{\top}, \E[\bLambda \bLambda^{\top} \mid \bA] \rangle\right] \\
	 = & \frac{1}{4n^2}\left( n\E_{\bLambdas \sim \mu_{\Lambda}}[\bLambdas^4] + n(n - 1) \E_{\bLambdas \sim \mu_{\Lambda}}[\bLambdas^2]^2 - n^2\MMSEas_n(\mu_{\Lambda}, \mu_{\Theta}) \right).
\end{align*}
Invoking Proposition 17 and Corollary 18 from \cite{lelarge2019fundamental}, for all $q_{\Theta}^2 \in D$
\begin{align*}
	\phi'(q_{\Theta}^2) = \frac{1}{4} q^{\ast}(q_{\Theta}^2)^2 =  \frac{1}{4}\E_{\bLambdas \sim \mu_{\Lambda}}[\bLambdas^2]^2 - \frac{1}{4} \lim_{n \rightarrow \infty} \MMSEsy_n(\mu_{\Lambda}; q_{\Theta}) + o_n(1).
\end{align*}
Combining all arguments above, we obtain that
\begin{align*}
	\liminf_{n,d \rightarrow \infty}\MMSEas_n(\mu_{\Lambda}, \mu_{\Theta}) \geq \lim_{n \rightarrow \infty} \MMSEsy_n(\mu_{\Lambda}; q_{\Theta}) - 8\eps.
\end{align*}
Since $\eps$ is arbitrary, we then complete the proof of the first claim of the theorem.

We then proceed to prove the second claim. For $q_{\Theta}^2, q_{\Theta}^2 + \eta \in D$ satisfying $0 < \eta < \eps$, by \cref{eq:C27,eq:Phi'-phi'} we have
\begin{align*}
	\lim_{n,d \rightarrow \infty}\MMSEas_n(\mu_{\Lambda}, \mu_{\Theta}; \eta) =&  \E_{\bLambdas \sim \mu_{\Lambda}}[\bLambdas^2]^2 - 4\phi'(q_{\Theta}^2 + \eta) \\
	\leq &\E_{\bLambdas \sim \mu_{\Lambda}}[\bLambdas^2]^2 - 4\phi'(q_{\Theta}^2) \\
	= & \lim_{n \rightarrow \infty} \MMSEsy_n(\mu_{\Lambda}; q_{\Theta}).
\end{align*}
Note that $\lim_{n,d \rightarrow \infty}\MMSEas_n(\mu_{\Lambda}, \mu_{\Theta}; \eta) \geq \limsup_{n,d \rightarrow \infty}\MMSEas_n(\mu_{\Lambda},\mu_{\Theta};\eps)$, the proof of the second claim immediately follows.  


\subsection{Reduction to bounded prior}\label{section:reduction-to-bdd-support}
In this section, we show that in order to prove \cref{thm:free-energy} and \cref{thm:lower-bound}, it suffices to prove the theorems under Assumption \ref{assumption:bounded-support}.  

Since $\mu_{\Lambda}$ is sub-Gaussian, for any $\eps > 0$, there exists $K_{\ep} > 0$, such that if we let $\bar\bLambda_0 := \bLambdas \mathbbm{1}\{|\bLambdas| \leq K_{\eps}\}$, then $\E_{\bLambdas \sim \mu_{\Lambda}}[(\bLambdas - \bar\bLambda_0)^4] < \eps$ and $\mu_{\Lambda}([-K_{\eps}, K_{\eps}]) > 1 - \eps^2$. For all $i \in [n]$, we define $\bar\bLambda_i := \bLambda_i \mathbbm{1}\{|\bLambda_i| \leq K_{\eps}\}$ and $\bar\vlambda_i := \vlambda_i \mathbbm{1}\{|\vlambda_i| \leq K_{\eps}\}$. Let $\bar\bLambda = (\bar\bLambda_i)_{i \leq n} \in \RR^n$ and $\bar\vlambda = (\bar\vlambda_i)_{i \leq n} \in \RR^n$. We introduce the truncated Hamiltonians:
\begin{align*}
	& \bar{H}_n^{\eps}(\bar\vlambda, \vtheta) := \frac{1}{\sqrt{nd}}\langle \barbLambda, \bar\vlambda \rangle \langle \bTheta, \vtheta \rangle + \frac{1}{\sqrt[4]{nd}} \bar\vlambda^{\sT} \bZ \vtheta - \frac{1}{2\sqrt{nd}}\|\bar\vlambda\|^2\|\vtheta\|^2, \\
	& \bar H_n^{Y, \eps}(\barblambda) := \frac{q_{\Theta}^2}{2n}\langle \bar\bLambda, \bar\vlambda \rangle^2 + \frac{q_{\Theta}}{2}\bar\vlambda^\sT\bW \bar\vlambda - \frac{q_{\Theta}^2}{4n}\|\bar\vlambda\|^4.
\end{align*}
Recall that $\bW' \overset{d}{=} \GOE(n)$ and is independent of $(\bW, \bZ)$. For $s,q, h \geq 0$, we define the truncated versions of $\Phi_n^Y$, $\Phi_n$ and $\cF$ as
\begin{align*}
	& \bar{\Phi}^{Y, \eps}_n(h) := \frac{1}{n} \E\left[ \log \left( \int \exp\Big(\bar H_n^{Y, \eps}(\barblambda) + \frac{h}{2n}\langle \bar\bLambda, \bar\vlambda \rangle^2 + \frac{\sqrt{h}}{2}\bar\vlambda^\sT\bW' \bar\vlambda - \frac{h}{4n}\|\bar\vlambda\|^4\Big) \tensorlb  \right) \right], \\
	& \bar{\Phi}^{\eps}_n(h) := \frac{1}{n} \E\left[ \log \left( \int \exp\Big(\bar{H}^{\eps}_n(\bar\vlambda, \vtheta) + \frac{h}{2n}\langle \bar\bLambda, \bar\vlambda \rangle^2 + \frac{\sqrt{h}}{2}\bar\vlambda^\sT\bW' \bar\vlambda - \frac{h}{4n}\|\bar\vlambda\|^4 \Big) \tensorlb \tensort \right) \right], \\
	& \bar{\cF}^{\eps}(s, q) := - \frac{s}{4}q^2 + \E_{Z \sim \normal(0,1), \bLambdas \sim \mu_{\Lambda}} \left[ \log \left(\int \exp\left( \sqrt{sq}Z\bar\lambda + sq\bar\lambda \bar\bLambda_0 - \frac{s}{2}q\bar\lambda^2 \right) \mu_{\bar\Lambda}(\dd\bar\lambda) \right) \right].
\end{align*}
In the above display, $\mu_{\bar\Lambda}$ stands for the law of $\bar\bLambda_0$. The following lemma states that $\bar\Phi_n^{\eps}(h)$ is close to $\Phi_n(h,0)$ for small $\eps$. 
\begin{lemma}\label{lemma:bounded-approx-free-energy}
	Under the conditions of \cref{thm:free-energy}, there exists a constant $C_0 > 0$, which is a function of $(\mu_{\Lambda}, \mu_{\Theta})$ only, such that for $n,d$ large enough, the following inequality holds for all $h \in [0,1]$:
	\begin{align*}
		\left| \Phi_n(h, 0) - \bar{\Phi}^{\eps}_n(h)  \right| \leq C_0\sqrt[4]{\eps}.
	\end{align*}
\end{lemma}
The proof of \cref{lemma:bounded-approx-free-energy} is deferred to \cref{proof:lemma:bounded-approx-free-energy}. Furthermore, according to Lemma 46 from \cite{lelarge2019fundamental}, $\cF$ is also close to $\bar\cF^{\ep}$ for $\ep$ small. 
\begin{lemma}[Lemma 46 from \cite{lelarge2019fundamental}]\label{lemma:from-LM19}
Under the conditions of \cref{thm:free-energy}, there exists a constant $K' > 0$ that depends only on $\mu_{\Lambda}$, such that
\begin{align*}
		\left| \sup\limits_{q \geq 0} \cF(s, q) - \sup\limits_{q \geq 0} \bar{\cF}^{\eps}(s,q) \right| \leq sK'\eps.
\end{align*}
\end{lemma}

Invoking the convergence results of free energy density for the symmetric spiked model \cite{lelarge2019fundamental}, we have 
\begin{align*}
	\left|\bar{\Phi}^{Y, \eps}_n(h) -\sup_{q \geq 0} \bar\cF^{\eps}(q_{\Theta}^2 + h, q) \right| = o_n(1).
\end{align*}
Applying \cref{lemma:free-energy-4} to the truncated distribution $\mu_{\bar\Lambda}$, we obtain that $|\bar{\Phi}^{Y, \eps}_n(h) - \bar\Phi_n^{\eps}(h)| = o_n(1)$ for all $h \geq 0$. Using this result and \cref{lemma:bounded-approx-free-energy}, \ref{lemma:from-LM19}, we derive that for all $h \in [0,1]$,
%
%
\begin{align*}
	& \big|\Phi_n(h, 0) - \Phi_n^Y(h, 0)\big| \\
	\leq & \big|\Phi_n(h,0) - \bar{\Phi}_n^{\eps}(h)\big| + \big|\bar{\Phi}_n^{\eps}(h) - \bar{\Phi}_n^{Y, \eps}(h)\big| + \big|\bar{\Phi}_n^{Y, \eps}(h) -  \sup\limits_{q \geq 0} \bar{\cF}^{\eps}(q_{\Theta}^2 + h,q)\big|  \\
	& + \big|\sup\limits_{q \geq 0} \bar{\cF}^{\eps}(q_{\Theta}^2 + h,q) - \sup\limits_{q \geq 0} {\cF}(q_{\Theta}^2 + h,q)\big| + \big|\sup\limits_{q \geq 0} {\cF}(q_{\Theta}^2 + h,q) - \Phi_n^Y(h,0)\big| \\
	\leq & C_0\sqrt[4]{\eps} + (q_{\Theta}^2 + 1)K'\eps + o_n(1).
\end{align*}
Since $\eps$ is arbitrary, we then have the following lemma:
\begin{lemma}\label{lemma:free-energy-converge}
	Under the conditions of Theorem \ref{thm:free-energy}, for all $h \in [0,1]$, as $n,d \rightarrow \infty$ we have
	\begin{align*}
		\lim_{n,d \rightarrow \infty}|\Phi_n(h, 0) - \Phi_n^Y(h, 0)| = 0.
	\end{align*}
\end{lemma}
\cref{thm:free-energy} is a direct consequence of \cref{lemma:free-energy-converge}. The remainder proof of \cref{thm:lower-bound} follows exactly the same procedure as stated in \cref{sec:proof-of-thm:lower-bound}, and here we skip it for the sake of simplicity. 


\subsection{Proof outline of \cref{thm:upper-bound}}\label{sec:proof-thm:upper-bound}
We state the proof outline of \cref{thm:upper-bound} in this section. The proofs of supporting lemmas are delayed to Appendix \ref{sec:tec-lemmas-upper-bound}. For the sake of simplicity, here we only consider the rank-one case $r = 1$. We comment that proof for $r \geq 2$ can be conducted analogously.

\subsubsection{Proof outline of \cref{thm:upper-bound} under condition (a)}\label{sec:cond(a)}
Since $\mu_{\Lambda}$ is sub-Gaussian, there exists a constant $K_0 > 0$ depending only on $\mu_{\Lambda}$, such that for all $x > 0$
\begin{align}\label{eq:C4-31}
	 \P(|\bLambdas| \geq x) \leq 2 \exp(-x^2 / K_0^2). 
\end{align}
For all $i \in [n]$, we define $\bar\bLambda_i :=\bLambda_i\mathbbm{1}\{|\bLambda_i| \leq 2K_0 \sqrt{\log n}\} $, $\bar\bLambda := (\bar\bLambda_1, \cdots, \bar\bLambda_n)^{\top} \in \RR^n$ and ${\bar\bA} :=  \bar\bLambda \bTheta^{\top} / \sqrt[4]{nd} + {\bZ} \in \RR^{n \times d}$.
The next lemma says that truncation does not decrease the MMSE too much. 
\begin{lemma}\label{lemma:log-truncate2}
Under the conditions of Theorem \ref{thm:upper-bound} (a) , as $n,d \rightarrow \infty$ we have 
\begin{align*}
	& \frac{1}{n^2} \E\left[\left\| \bLambda \bLambda^{\top} - \E[\bLambda \bLambda^{\top} \mid \bA] \right\|_F^2 \right] \leq \frac{1}{n^2}\E\left[\left\| \barbLambda \barbLambda^{\top} - \E[\barbLambda \barbLambda^{\top} \mid {\bar\bA}] \right\|_F^2 \right] + o_n(1).
\end{align*}
\end{lemma} 
We leave the proof of \cref{lemma:log-truncate2} to Appendix \ref{sec:proof-of-lemma:log-truncate2}. By \cref{lemma:log-truncate2}, in order to prove the theorem, it suffices to show that under the current conditions, for all but countably many values of $q_{\Theta} > 0$,
\begin{align}\label{eq:after-truncation}
	\limsup_{n,d \rightarrow \infty}\frac{1}{n^2}\E\left[\left\| \barbLambda \barbLambda^{\top} - \E[\barbLambda \barbLambda^{\top} \mid {\bar\bA}] \right\|_F^2 \right] \leq \lim_{n\rightarrow \infty} \MMSEsy_n(\mu_{\Lambda}; q_{\Theta}).
\end{align}
We let $\bZ\bTheta = \|\bTheta\| \cdot \bg$, where $\bg \sim \normal(0, \id_n)$ is independent of $(\bTheta, \bLambda)$. We denote by $\bZ'$ an independent copy of $\bZ$, such that $\bZ'$ is further independent of $(\bTheta, \bLambda, \bg)$. Furthermore, we can choose $\bZ'$ such that $(\bTheta, \bLambda, \bg, \bZ \Pj_{\bTheta}^{\perp} \bZ^{\top}) = (\bTheta, \bLambda, \bg, \bZ' \Pj_{\bTheta}^{\perp} {\bZ'}^{\top})$, where $\Pj_{\bTheta}^{\perp}$ denotes the projection onto the null space of $\bTheta$. 

We define $\bY_1 := (\bar\bA \bar\bA^{\top} - d\id_n) / \sqrt{d}$. Then $\bY_1$ admits the following decomposition:
\begin{align*}
	\bY_1 
	= & \frac{\|\bTheta\|^2}{\sqrt{n}d} \barbLambda\barbLambda^{\top} + \frac{\|\bTheta\|}{n^{1/4}d^{3/4}} \barbLambda \bg^{\top} + \frac{\|\bTheta\|}{n^{1/4}d^{3/4}} \bg \barbLambda^{\top}   + \frac{1}{\sqrt{d}}(\bZ \Pj_{\bTheta}^{\perp} \bZ^{\top} - d\id_n) + \frac{1}{\sqrt{d}} \bg \bg^{\top} \\
	= & \frac{1}{\sqrt{n}}\left( q_{\Theta}^{1/2} \barbLambda + r_n^{-1} \bg \right)\left( q_{\Theta}^{1/2} \barbLambda + r_n^{-1} \bg \right)^{\top} + \frac{1}{\sqrt{d}}(\bZ'{\bZ'}^{\top} - d\id_n) + \bE. 
\end{align*}
In the above display, the $n \times n$ symmetric matrix $\bE$ is defined as follows:
\begin{align*}
	\bE := &  - \frac{1}{\sqrt{d}\|\bTheta\|^2} {\bZ'} \bTheta \bTheta^{\top} {\bZ'}^{\top} + \frac{1}{\sqrt{n}}\left(\frac{\|\bTheta\|^2}{d}  - q_{\Theta}\right) \barbLambda \barbLambda^{\top} \\
	&+ \frac{1}{n^{1/4}d^{1/4}}\left( \frac{\|\bTheta\|}{\sqrt{d}} - q_{\Theta}^{1/2}\right)\left(\barbLambda \bg^{\top} + \bg \barbLambda^{\top} \right) + \frac{1}{\sqrt{d}} \bg\bg^{\top}.
\end{align*}
We define the set
\begin{align*}
	\Omega_1 := \left\{ \Big| \frac{1}{d}\|\bTheta\|^2 - q_{\Theta}\Big| \leq \frac{C_1\sqrt{\log n}}{\sqrt{d}}, \frac{1}{\sqrt{d}}|(\bZ'\bTheta)_i|  \leq C_1 \sqrt{\log n}, \frac{1}{\sqrt{d}}|g_i| \leq C_1 \sqrt{\log n}\,  \mbox{ for all }i \in [n] \right\},
\end{align*}
where $C_1 > 0$ is a constant depending only on $\mu_{\Theta}$. Since $\mu_{\Theta}$, $\mu_{\Lambda}$ are sub-Gaussian distributions, if we choose $C_1$ large enough, then we have $\P(\Omega_1) = 1 - o_n(1)$. Using the definition of $\Omega_1$, we conclude that there exists a constant $C_2 > 0$ that depends only on $(C_1, K_0, \mu_{\Theta})$, such that on $\Omega_1$ we have $|E_{ij}| \leq {C_2\log n} / {\sqrt{d}}$ for all $i,j \in [n]$. For some absolute constant $C_3 > 0$, we let $\bar\bg \in \RR^n$ such that $\bar g_i := g_i \mathbbm{1}\{|g_i| \leq C_3 \sqrt{\log n}\}$ for all $i \in [n]$. Direct computation reveals that for $C_3$ large enough, we have  $\P(\bar\bg \neq \bg) \rightarrow 0$ as $n,d \rightarrow \infty$.

Define
\begin{align*}
	\bY_2 := \frac{1}{\sqrt{n}}\left( q_{\Theta}^{1/2} \barbLambda + r_n^{-1} \bar\bg \right)\left( q_{\Theta}^{1/2} \barbLambda + r_n^{-1} \bar\bg \right)^{\top} + \bG,
\end{align*}
where $\bG \sim \sqrt{n}\GOE(n)$ and is independent of $(\bLambda, \barbLambda, \bTheta,  \bg, \bar\bg)$.  By \cite[Theorem 4]{bubeck2016testing}, under condition (a), there exists a coupling such that as $n,d \rightarrow \infty$, with probability $1 - o_n(1)$ we have $(\bZ' {\bZ'}^{\top} - d\id_n) / \sqrt{d} = \bG$. We define $\Omega_2 := \Omega_1 \cap \{\bY_2 = \bY_1 - \bE\}$, then we see that $\P(\Omega_2) \rightarrow 1$ as $n,d \to \infty$.

For $\bX \in \RR^{n \times n}$, we define
\begin{align*}
	\bM_n(\bX) := \frac{1}{n}\E\left[\big( q_{\Theta}^{1/2} \barbLambda + r_n^{-1} \bar\bg \big)\big( q_{\Theta}^{1/2} \barbLambda + r_n^{-1} \bar\bg \big)^{\top} \big| \bY_2 = \bX \right].
\end{align*}
Then for any $i,j ,k,s\in [n]$, we have
\begin{align*}
	&\frac{\partial \bM_n(\bX)_{ks}}{\partial X_{ij}}  =  \frac{1}{n^{3/2}}\E\big[(q_{\Theta}^{1/2} \bar\bLambda_i + r_n^{-1}\bar g_i)(q_{\Theta}^{1/2} \bar\bLambda_j +r_n^{-1}\bar g_j)(q_{\Theta}^{1/2} \bar\bLambda_k +r_n^{-1}\bar g_k)(q_{\Theta}^{1/2} \bar\bLambda_s +r_n^{-1}\bar g_s) \mid \bY_2 = \bX\big] \\
	& - \frac{1}{n^{3/2}} \E\big[ (q_{\Theta}^{1/2} \bar\bLambda_i +r_n^{-1}\bar g_i)(q_{\Theta}^{1/2} \bar\bLambda_j +r_n^{-1}\bar g_j) \mid \bY_2 = \bX\big]\E\big[ (q_{\Theta}^{1/2} \bar\bLambda_k +r_n^{-1}\bar g_k)(q_{\Theta}^{1/2} \bar\bLambda_s +r_n^{-1}\bar g_s) \mid \bY_2 = \bX\big].
\end{align*}
Note that on $\Omega_2$ we have $|\bar\bLambda_i| \leq 2K_0\sqrt{\log n}$, $|\bar g_i| \leq C_3\sqrt{\log n}$, and $|E_{ij}| \leq C_2 \log n / \sqrt{d}$ for all $i,j \in [n]$. Therefore, we conclude that there exists a constant $C_4 > 0$ depending only on $(K_0, \mu_{\Theta}, C_1, C_2, C_3)$, such that on $\Omega_2$ we have
\begin{align*}
	|\bM_n(\bY_1)_{ks} - \bM_n(\bY_1 - \bE)_{ks}| \leq \frac{C_4\sqrt{n}(\log n)^{3}}{\sqrt{d}}. 
\end{align*}
Therefore, we obtain that
\begin{align*}
	\E\left[ \| \bM_n(\bY_1 - \bE) - \bM_n(\bY_1)\|_F^2 \mathbbm{1}_{\Omega_2} \right]  \leq  \frac{C_4^2n^3(\log n)^{6}}{d}.
\end{align*}
The right hand side of the above equation vanishes as $n,d \rightarrow \infty$ under condition (a). Therefore, we derive that $\|\bM_n(\bY_2) - \bM_n(\bY_1)\|_F = o_P(1)$. 
Note that $\bY_1$ is a function of $\bar\bA$. Using standard truncation argument, we conclude that in order to prove \cref{eq:after-truncation}, it suffices to show 
\begin{align*}
	\limsup_{n,d \rightarrow \infty} \E\left[ \Big\| q_{\Theta}^{-1}\bM_n(\bY_2) - \frac{1}{n}  \barbLambda \barbLambda^{\top} \Big\|_F^2 \right]  \leq \lim_{n \rightarrow \infty} \MMSEsy_n(\mu_{\Lambda}; q_{\Theta}),
\end{align*}
which we prove in \cref{lemma:toMMSE} below. The proof of this lemma is deferred to Appendix \ref{sec:proof-of-lemma:toMMSE}.
\begin{lemma}\label{lemma:toMMSE}
	Under the conditions of \cref{thm:upper-bound} (a), for all but countably many values of $q_{\Theta} > 0$, we have
	\begin{align*}
		\lim\limits_{n \rightarrow \infty}\E\left[ \Big\| q_{\Theta}^{-1} \bM_n(\bY_2) - \frac{1}{n}\bar\bLambda \bar\bLambda^{\top}  \Big\|_F^2 \right] = \lim_{n \rightarrow \infty} \MMSEsy_n(\mu_{\Lambda}; q_{\Theta}). 
	\end{align*}
\end{lemma}

\subsubsection{Proof outline of \cref{thm:upper-bound} under condition (b)}\label{sec:cond(b)}

\subsubsection*{Truncation}
By assumption, there exists $0 < K_1 < \infty$ such that support$(\mu_\Lambda) \subseteq [-K_1, K_1]$.  Since $\mu_{\Theta}$ is sub-Gaussian, there exists $K_2 > 0$ which depends only on $ \mu_{\Theta}$, such that for all $x > 0$
\begin{align}\label{eq:37}
	 \P_{\bThetas \sim \mu_{\Theta}}(|\bThetas| \geq x) \leq 2 \exp(-x^2 / K_2^2). 
\end{align}
For $j \in \{0\} \cup [d]$, we define $\bar\bTheta_j := \bTheta_j\mathbbm{1}\{|\bTheta_j| \leq 2K_2 \sqrt{\log d}\} $, $\bar\bTheta := (\bar\bTheta_1, \cdots, \bar\bTheta_d)^{\top} \in \RR^d$ and $\bar\bA := \bLambda \barbTheta^{\top} / \sqrt[4]{nd} + {\bZ}$. Note that $\bar\bA$ defined here is not to be confused with $\bar\bA$ defined in Appendix \ref{sec:cond(a)}. The following lemma states that truncation does not decrease the asymptotic matrix MMSE. 
\begin{lemma}\label{lemma:log-truncate}
Under the conditions of Theorem \ref{thm:upper-bound} (b), as $n,d \rightarrow \infty$ we have 
\begin{align*}
	\frac{1}{n^2} \E\left[\left\| \bLambda \bLambda^{\top} - \E[\bLambda \bLambda^{\top} \mid \bA] \right\|_F^2 \right] \leq \frac{1}{n^2}\E\left[\left\| \bLambda \bLambda^{\top} - \E[\bLambda \bLambda^{\top} \mid \bar\bA] \right\|_F^2 \right] + o_n(1).
\end{align*}
\end{lemma} 
We postpone the proof of the lemma to Appendix \ref{sec:proof-of-lemma:log-truncate}. By \cref{lemma:log-truncate}, in order to prove \cref{thm:upper-bound} under condition (b), we only need to show for all but countably many values of $q_{\Theta} > 0$,
\begin{align*}
	\limsup_{n,d \rightarrow \infty}\frac{1}{n^2} \E\left[\left\| \bLambda \bLambda^{\top} - \E[\bLambda \bLambda^{\top} \mid \bar\bA]\right\|_F^2 \right] \leq \lim_{n \rightarrow \infty} \MMSEsy_n(\mu_{\Lambda}; q_{\Theta}). 
\end{align*}
\subsubsection*{Model with extra perturbation}
For $s,h, a,a' \geq 0$ and $\{\eps_n\}_{n \geq 1}, \{\eps'_n\}_{n \geq 1} \subseteq \RR_+$, we introduce a perturbed model sequence, such that for each $n$, we observe $(\bar\bA(s), \bx'(a'), \bar\bx(a), \bY'(h))$ defined as follows:
\begin{align}
	& \bar\bA(s) := \frac{\sqrt{s}}{\sqrt[4]{nd}} \bLambda \barbTheta^{\top} + \bZ, \label{model:As} \\
	& \bx'(a') := a'\sqrt{\ep_n'} \bLambda + \bg', \label{model:perturb-lambda} \\
	& \bar\bx(a) := a\sqrt{\frac{n\ep_n}{d}} \barbTheta + \bg\label{model:perturb-theta}, \\
	& \bY'(h) := \frac{\sqrt{h}}{n} \bLambda \bLambda^{\top} + \bW',
\end{align}
where $\bg \sim \normal(0, \id_d)$, $\bg' \sim \normal(0, \id_n)$, $\bW' \sim  \GOE(n)$, mutually independent and are independent of everything else. Note that $\bar\bA(1) = \bar\bA$. Furthermore, we assume that $\eps_n, \eps_n' \rightarrow 0^+$ as $n,d \rightarrow \infty$. 
We can associate to the observation \eqref{model:As} the Hamiltonian 
\begin{align}
	\bar H_n^{[s]}(\vlambda, \barbtheta) = \sum\limits_{i \in [n], j \in [d]} \left\{ \frac{s}{\sqrt{nd}} \bLambda_i \vlambda_i \bar\bTheta_j \bar\vtheta_j + \frac{\sqrt{s}}{\sqrt[4]{nd}} Z_{ij} \vlambda_i \bar\vtheta_j - \frac{s}{2\sqrt{nd}} \vlambda_i^2 \bar\vtheta_j^2 \right\}, \label{eq:Hamiltonian-barHs}
\end{align}
where $\bar\vtheta_j = \vtheta_j \mathbbm{1}\{ |\vtheta_j| \leq 2K_2 \sqrt{\log d} \}$. For the sake of simplicity, we let $\bar H_n(\vlambda, \barbtheta) = \bar H_n^{[1]}(\vlambda, \barbtheta) $. Similarly, we can associate to the observations \eqref{model:perturb-lambda} and \eqref{model:perturb-theta} the following Hamiltonians, respectively:
\begin{align*}
	& \pertl(\vlambda) = \sum\limits_{i = 1}^n \left\{\sqrt{\ep_n'} a' \vlambda_i g_i' + {a'}^2 \ep_n' \bLambda_i \vlambda_i - \frac{{a'}^2 \ep_n'}{2} \vlambda_i^2 \right\}, \\
	& \pertt(\barbtheta) = \sum\limits_{j = 1}^d \left\{\sqrt{\frac{n\ep_n}{d}}a\bar\vtheta_j g_j + \frac{ a^2n\ep_n}{d} \bar\bTheta_j \bar\vtheta_j - \frac{a^2n\ep_n}{2d} \bar\vtheta_j^2 \right\}. 
\end{align*}
We then define the ``total'' Hamiltonian, which corresponds to all observations in the perturbed model as
\begin{align*}
	\perttot(\vlambda, \barbtheta) := \bar H_n^{[s]}(\vlambda, \barbtheta) + \pertl(\vlambda) + \pertt(\barbtheta) + H_n(\vlambda; \bY'(h) ),
\end{align*}
where we recall that $H_n(\vlambda; \bY'(h) )$ is defined in \cref{eq:Hamiltonian-Yp}. The posterior distribution of $(\blambda, \bar\bTheta)$ given observations $(\bar\bA(s), \bx'(a'), \bar\bx(a), \bY'(h))$ can be expressed as 
\begin{align*}
	\mu(\dd \vlambda, \dd \bar\vtheta \mid \bar\bA(s), \bx'(a'), \bar\bx(a), \bY'(h)) \propto \exp(\bar H_n^{(tot)}(\vlambda, \barbtheta))\tensorl \tensortb.
\end{align*}
We define the free energy functionals corresponding to the total Hamiltonian as
\begin{align}\label{eq:40}
\begin{split}
	 & \bar\phi_n(s, a, a',h) := \frac{1}{n}\log \int \exp(\perttot(\vlambda, \barbtheta) )\tensorl \tensortb, \\
	 &  \bar\Phi_n(s,a, a',h) := \E\left[ \bar\phi_n(s, a, a',h) \right],
\end{split}
\end{align}
where $\mu_{\bar\Theta}^{\otimes d}$ is the product distribution over $\RR^d$ with each coordinate distributed as $\bar{\bTheta}_0$. 
\subsubsection*{Truncation does not change the asymptotic free energy density}
Next, we show that truncation does not change the asymptotic free energy density. 
\begin{lemma}\label{lemma:C9}
For $s,h \geq 0$, we define the following unperturbed Hamiltonian and free energy density 
\begin{align*}
	&  H_n^{[s]}(\vlambda, \vtheta) := \sum\limits_{i \in [n], j \in [d]} \left\{ \frac{s}{\sqrt{nd}} \bLambda_i \vlambda_i \bTheta_j \vtheta_j + \frac{\sqrt{s}}{\sqrt[4]{nd}} Z_{ij} \vlambda_i \vtheta_j - \frac{s}{2\sqrt{nd}} \vlambda_i^2 \vtheta_j^2 \right\}, \\
	& {\Phi}_n(s, 0, 0, h) := \frac{1}{n}\E\left[ \log \int \exp( H_n^{[s]}(\vlambda, \vtheta) + H_n(\vlambda; \bY'(h) )) \tensorl \tensort\right]. 
\end{align*}
Then for all fixed $S_0 > 0$, under the conditions of Theorem \ref{thm:upper-bound} (b), as $n,d \rightarrow \infty$ we have 
	$$\sup_{h \geq 0, S_0 \geq s \geq 0}|{\Phi}_n(s, 0, 0, h)  - {\bar\Phi}_n(s, 0, 0, h) | = o_n(1).$$
\end{lemma}

The proof of \cref{lemma:C9} is given in Appendix \ref{sec:proof-of-lemma:C9}. We further characterize the convergence of free energy density in \cref{lemma:G1}, again postponing the proof to Appendix \ref{sec:proof-of-lemma:G1}. 

\begin{lemma}\label{lemma:G1}
	Recall that $\cF(\cdot,\cdot)$ is defined in \cref{eq:FsQ}. Under the conditions of \cref{thm:upper-bound} (b), if we further assume that $\ep_n, \ep_n' \rightarrow 0^+$, then for all fixed $s,h \geq 0$, as $n,d \to \infty$ 
	\begin{align*}
		\lim\limits_{n,d \rightarrow \infty}\sup\limits_{a,a' \in [0,10]} \left|\bar\Phi_n(s,a,a', h) - \sup_{q \geq 0} \cF(q_{\Theta}^2s^2 + h, q) \right| = 0.
	\end{align*} 
\end{lemma}
 Recall that $D$ is defined in \cref{eq:setD}. For all $h + q_{\Theta}^2 \in D$, the mapping 
$s \mapsto \sup_{q \geq 0}\cF(q_{\Theta}^2s^2 + h, q)$ is differentiable at $s = 1$. Furthermore, for all fixed $a,a',h \geq 0$, the mappings $s \mapsto \bar\Phi_n(s,a,a', h)$,  $s \mapsto \bar\Phi_n(s,0,0, h)$ are convex differentiable on $(0, \infty)$. By \cref{lemma:C9,lemma:G1}, as $n, d \rightarrow \infty$, $\bar\Phi_n(s,a,a',h) \rightarrow \sup_{q \geq 0}\cF(q_{\Theta}^2s^2 + h, q)$ and $\bar\Phi_n(s,0,0,h) \rightarrow \sup_{q \geq 0}\cF(q_{\Theta}^2s^2 + h, q)$. Then we apply \cref{lemma:convex-derivative} and conclude that for $h + q_{\Theta}^2 \in D$, we have
\begin{align*}
	\left|\frac{\partial}{\partial s}\bar\Phi_n(s,a,a',h) \Big\vert_{s = 1} - \frac{\partial}{\partial s}\bar\Phi_n(s,0,0,h)\Big\vert_{s = 1}\right| = o_n(1).
\end{align*}
Using Gaussian integration by parts and Nishimori identity, we further derive that for $h + q_{\Theta}^2 \in D$
\begin{align}\label{eq:G187}
	\lim\limits_{n,d \rightarrow \infty}\frac{1}{2n\sqrt{nd}}\E\left[ \| \E[\bLambda \barbTheta^{\top} \mid \bar\bA(1),  \bar\bx(a), \bx'(a'),\bY'(h)] - \E[\bLambda \barbTheta^{\top} \mid \bar\bA(1), \bY'(h)]  \|_F^2   \right] = 0. 
\end{align}
%
By Jensen's inequality, for all $a,a' \in [1,2]$ 
\begin{align*}
	&\frac{1}{2n\sqrt{nd}}\E\left[ \| \E[\bLambda \barbTheta^{\top} \mid \bar\bA(1), \bar\bx(a), \bx'(a'), \bY'(h)] - \E[\bLambda \barbTheta^{\top} \mid \bar \bA(1), \bY'(h)]  \|_F^2   \right] \\
	\leq & \frac{1}{2n\sqrt{nd}}\E\left[ \| \E[\bLambda \barbTheta^{\top} \mid \bar\bA(1), \bar\bx(2), \bx'(2), \bY'(h)] - \E[\bLambda \barbTheta^{\top} \mid \bar \bA(1), \bY'(h)]  \|_F^2   \right].
\end{align*}
By \cref{eq:G187}, the last line above converges to zero as $n,d \to \infty$ for all $q_{\Theta}^2 + h \in D$. In this case, we have
\begin{align}\label{eq:C34}
	\lim\limits_{n,d \rightarrow \infty}\int_1^2\int_1^2 \frac{1}{2n\sqrt{nd}}\E\left[ \| \E[\bLambda \barbTheta^{\top} \mid \bar\bA(1), \bar\bx(a), \bx'(a'), \bY'(h)] - \E[\bLambda \barbTheta^{\top} \mid \bar\bA(1), \bY'(h)]  \|_F^2   \right] \dd a \dd a' = 0. 
\end{align}
We define 
\begin{align*}
	& \bLambda_{s, a, a', h} := \E[\bLambda \mid \bar\bA(1), \bar\bx(a), \bx'(a'), \bY'(h)] \in \RR^n, \\
	& \barbTheta_{s, a, a', h} := \E[\barbTheta \mid \bar\bA(1), \bar\bx(a), \bx'(a'), \bY'(h)] \in \RR^d, \\
	& \bM_{s, a, a', h} := \E[ \bLambda \bar\bTheta^{\top} \mid \bar\bA(1), \bar\bx(a), \bx'(a'), \bY'(h) ] \in \RR^{n \times d}.
\end{align*}
Invoking Stein's lemma, we see that the following equation holds:
\begin{align}
	 \frac{\partial}{\partial a}\E\left[ \|\bLambda_{1, a, a', h}\|^2 \right] =  \frac{2an\eps_n}{d} \E\left[ \| \bM_{1, a, a', h}- \bLambda_{1, a, a', h}\barbTheta_{1, a, a', h}^{\top}\|_F^2 \right]. \nonumber 
\end{align}
%
Using the above equation, we obtain that
%
%
%
\begin{align}
	& \int_1^2\int_1^2 \frac{1}{2n\sqrt{nd}} \E\left[ \|\bM_{1, a, a', h} - \bLambda_{1, a, a', h}\barbTheta_{1, a, a', h}\|_F^2 \right] \dd a \dd a' \nonumber\\
	 \leq & \frac{d^{1/2}}{2\ep_n n^{5/2} } \int_1^2  \left\{ \E[\|\bLambda_{1, 2, a', h}\|^2] - \E[\|\bLambda_{1, 1, a', h} \|^2]  \right\} \dd a' \nonumber\\
	 \leq & \frac{d^{1/2}\E_{\bLambdas \sim \mu_{\Lambda}}[\bLambdas^2]}{\ep_n n^{3/2}}  .\label{eq:C35}
\end{align}
%
%
%
\subsubsection*{Overlap concentration}
The next lemmas show that if we draw two independent samples from the posterior distribution of $\barbTheta$ ($\bLambda$) given $(\bar\bA(1), \bar\bx(a), \bx'(a'), \bY'(h))$, then their normalized inner product concentrates. This phenomenon is referred to as \emph{overlap concentration} in the literature of statistical mechanics. In what follows, we prove overlap concentration for $\bar\bTheta$. Since the proof is similar, in order to avoid redundancy, we skip the counterpart proof for $\bLambda$. For the sake of simplicity, we denote by $\langle \cdot\rangle_{s,a,a',h}$ the expectation with respect to the posterior distribution $\P(\cdot \mid \bar\bA(s), \bx'(a'), \bar\bx(a), \bY'(h))$.
\begin{lemma}\label{lemma:C11}
For $\barbtheta \in \RR^d$, we define
\begin{align*}
	U(\barbtheta) := \sum\limits_{j = 1}^d \left\{ \frac{1}{\sqrt{\ep_n nd}}\bar\vtheta_j g_j + \frac{2a}{d} \bar\vtheta_j \bar\bTheta_j - \frac{a}{d} \bar\vtheta_j^2 \right\}.
\end{align*}
Let $\barbtheta^{(1)}$, $\barbtheta^{(2)}, \barbtheta \in \RR^d$ be independent samples drawn from the posterior distribution $\P(\barbTheta = \cdot \mid \bar\bA(1), \bar\bx(a), \bx'(a'), \bY'(h))$. Then under the conditions of \cref{thm:upper-bound} (b), for all $a,a' \in [0.1,10]$ and $h \geq 0$, we have
\begin{align}\label{eq:K2log}
	& \E[\langle ((\barbtheta^{(1)})^{\top} \barbtheta^{(2)} / d - \E[\langle (\barbtheta^{(1)})^{\top} \barbtheta^{(2)} / d\rangle_{1,a,a',h}])^2 \rangle_{1,a,a',h}] \nonumber \\
	\leq & 40K_2^2 \log d\E[\langle | U(\barbtheta) - \E[\langle U(\barbtheta) \rangle_{1,a,a',h}]|\rangle_{1,a,a',h}]. 
\end{align}
\end{lemma}
Using \cref{eq:K2log}, we see that in order to prove overlap concentration, we only need to show that the right hand side of \cref{eq:K2log} is sufficiently small, which is accomplished via the following lemmas:
\begin{lemma}\label{lemma:G3}
	We let $\barbtheta \in \RR^d$ be a sample drawn from the posterior distribution $\P(\barbTheta = \cdot \mid \bar\bA(1), \bar\bx(a), \bx'(a'), \bY'(h))$. For $h \geq 0$, we define 
	$$v_n(h) := \sup_{1/2 \leq a, a' \leq 3}\E[|\bar\phi_n(1,a,a',h) - \E[\bar\phi_n(1,a,a',h)]|],$$
	where we recall that $\bar\phi_n$ is defined in \cref{eq:40}. Then under the conditions of \cref{thm:upper-bound} (b), if we further assume that $\ep_n \rightarrow 0^+$ and $nd^{-1/2}\eps_n \rightarrow \infty$ as $n,d \to \infty$, then there exists a numerical constant $C > 0$, such that for $n,d$ large enough
	\begin{align*}
		\int_1^2\int_1^2 \E[\langle |U(\barbtheta) - \E[\langle U(\barbtheta)\rangle_{1,a,a',h}]| \rangle_{1,a,a',h}]\dd a \dd a' \leq CK_2 \sqrt{(v_n(h) + n^{-1})\ep_n^{-1} \log d}.
	\end{align*}
\end{lemma}
\begin{lemma}\label{lemma:concentration-of-free-energy}
	Under the conditions of Theorem \ref{thm:upper-bound} (b), if we further assume $\eps_n, \eps_n' \rightarrow 0^+$ as $n,d \to \infty$, then there exists a numerical constant $C_1 > 0$ such that for all $n,d$ large enough and $0 \leq h \leq 1$ 
	\begin{align*}
		v_n(h) \leq C_1 K_1^2 K_2^2 d^{1/2} n^{-1} \log d . 
	\end{align*}
\end{lemma}
We defer the proofs of \cref{lemma:C11,lemma:G3,lemma:concentration-of-free-energy} to \cref{sec:proof-of-lemma:C11,sec:proof-of-lemma:G3,sec:proof-of-lemma:concentration-of-free-energy}, respectively. Combining \cref{lemma:C11,lemma:G3,lemma:concentration-of-free-energy}, we deduce that under the conditions of these lemmas, for $\bar\vtheta^{(1)}, \bar\vtheta^{(2)} \in \RR^d$ that are two independent samples from the posterior distribution $\P(\barbTheta = \cdot \mid \bar\bA(1), \bx'(a'), \bar\bx(a), \bY'(h))$, there exists a numerical constant $C > 0$, such that for $n,d$ large enough
\begin{align*}
	\frac{d}{n}\int_1^2 \int_1^2 \E[\langle (\barbtheta_1^{\top} \barbtheta_2 / d - \E[\langle \barbtheta_1^{\top} \barbtheta_2 / d \rangle_{1,a,a',h}])^2  \rangle_{1,a,a',h}] \dd a \dd a' \leq CK_1K_2^4\ep_n^{-1/2}(\log d)^2d^{5/4}n^{-3/2}.
\end{align*}
Under condition (b), we see that there exists $\ep_n \rightarrow 0^+$, such that $\ep_n^{-1/2}(\log d)^2d^{5/4}n^{-3/2} \rightarrow 0$ and $nd^{-1/2} \eps_n \rightarrow \infty$ as $n,d \rightarrow \infty$. We summarize the overlap concentration results in \cref{thm:overlap-concentration} below, which also contains concentration argument for $\bLambda$ (that we skip the proof).

%
\begin{theorem}[Overlap concentration]\label{thm:overlap-concentration}

Let $\bar\vtheta^{(1)}, \bar\vtheta^{(2)} \in \RR^d$ be two independent samples drawn from the posterior distribution $\P(\barbTheta = \cdot \mid \bar\bA(1), \bx'(a'), \bar\bx(a), \bY'(h))$, and $\vlambda^{(1)}, \vlambda^{(2)} \in \RR^n$ be two independent samples drawn from the posterior distribution $\P(\bLambda = \cdot \mid \bar\bA(1), \bx'(a'), \bar\bx(a), \bY'(h))$. Under the conditions of Theorem \ref{thm:upper-bound} (b), there exist $\ep_n, \ep_n' \rightarrow 0^+$ such that for all $h \in [0,1]$, as $n,d \rightarrow \infty$ 
	\begin{align*}
		& \frac{(\log d)^2d^{5/4}}{n^{3/2}\ep_n^{1/2}} \rightarrow 0, \qquad \frac{n \eps_n}{d^{1/2}} \rightarrow \infty , \qquad \frac{(\log d)^{1/2}d^{1/4}}{n^{1/2}{\ep_n'}^{1/2}} \rightarrow 0, \qquad \frac{n\eps_n'}{d^{1/2}\log d}  \rightarrow \infty, \\
		& \frac{d}{n}\int_1^2 \int_1^2 \E[\langle (\barbtheta_1^{\top} \barbtheta_2 / d - \E[\langle \barbtheta_1^{\top} \barbtheta_2 / d \rangle_{1,a,a',h}])^2  \rangle_{1,a,a',h}] \dd a \dd a' \rightarrow 0, \\
		& \int_1^2 \int_1^2 \E[\langle (\vlambda_1^{\top} \vlambda_2 / n - \E[\langle \vlambda_1^{\top} \vlambda_2 / n \rangle_{1,a,a',h}])^2  \rangle_{1,a,a',h} ]\dd a \dd a' \rightarrow 0.
	\end{align*}
\end{theorem}
\begin{corollary}\label{cor:overlap-concentration}
	Under the conditions of \cref{thm:overlap-concentration}, for all $h \in [0,1]$, as $n,d \rightarrow \infty$ we have
	\begin{align*}
		& \frac{d}{n}\int_1^2 \int_1^2 \E[\langle (\langle\barbtheta_1^{\top} \barbtheta_2 / d \rangle_{1,a,a',h} - \E[\langle \barbtheta_1^{\top} \barbtheta_2 / d \rangle_{1,a,a',h}])^2  \rangle_{1,a,a',h}] \dd a \dd a' \rightarrow 0, \\
		& \int_1^2 \int_1^2 \E[\langle (\langle \vlambda_1^{\top} \vlambda_2 / n \rangle_{1,a,a',h} - \E[\langle \vlambda_1^{\top} \vlambda_2 / n \rangle_{1,a,a',h}])^2  \rangle_{1,a,a',h}] \dd a \dd a' \rightarrow 0.
	\end{align*}
\end{corollary}
\begin{remark}\label{remark:overlap-concentration}
In \cref{thm:overlap-concentration} and \cref{cor:overlap-concentration}, we can replace the interval $[1,2]$ with $[a,b]$ for any fixed $0 < a < b < \infty$. 
\end{remark}

\subsubsection*{Proof of the theorem}

In the rest parts of the proof, we always assume that $\{\ep_n\}_{n \in \NN_+}$ and $\{\ep_n'\}_{n \in \NN_+}$ are chosen as in \cref{thm:overlap-concentration}. Under this assumption we have $\ep_n^{-1}(\log d)^4d^{5/2}n^{-3} \rightarrow 0$ and $d^{-2}n^{3/2}(\log d)^{-4} \rightarrow 0$, thus $d^{1/2}n^{-3/2}\ep_n^{-1} \rightarrow 0$ as $n,d \to \infty$. Plugging this result into \cref{eq:C34,eq:C35}, we obtain that 

%
\begin{align}\label{eq:44}
	 \lim\limits_{n,d \rightarrow \infty}\int_1^2\int_1^2 \frac{1}{2n\sqrt{nd}}\E\left[ \| \bLambda_{1,a,a',h} \barbTheta_{1,a,a',h} - \bM_{1,0,0,h}  \|_F^2   \right] \dd a \dd a' = 0. 
\end{align} 
%

By \cref{lemma:C9} and \cref{lemma:G1} we see that $|\Phi_n(s, 0, 0, h) - \sup_{q \geq 0} \cF(q_{\Theta}^2s^2 + h, q)| = o_n(1)$ for all fixed $s, h \geq 0$. Notice that $s \mapsto \Phi_n(s, 0, 0, h)$ is convex and differentiable. Furthermore, for all but countably many values of $q_{\Theta}^2 + h > 0$, the mapping $s \mapsto \sup_{q \geq 0} \cF(q_{\Theta}^2s^2 + h, q)$ is differentiable at $s = 1$. Using Gaussian integration by parts, \cref{lemma:convex-derivative,lemma:nishimori} we conclude that for these $q_{\Theta}^2 + h$ we have 
%
%
\begin{align}\label{eq:C45}
	\lim\limits_{n,d \rightarrow \infty} \frac{1}{2n\sqrt{nd}} \E\left[ \| \E[\bLambda \barbTheta^{\top} \mid \bar\bA(1), \bY'(h)] \|_F^2 \right] = \frac{\partial}{\partial s} \sup_{q \geq 0} \cF(s^2q_{\Theta}^2 + h, q) \Big\vert_{s = 1}.
\end{align}
Define $D_{\Theta}(h) := \frac{\partial}{\partial s} \sup_{q \geq 0} \cF(s^2q_{\Theta}^2 + h, q)\Big\vert_{s = 1}$. For all but countably many $q_{\Theta} > 0$ the mapping $s \mapsto \sup_{q \geq 0} \cF(q_{\Theta}^2s^2, q)$ is differentiable at $s = 1$, thus $D_{\Theta}(0)$ is well-defined. In this case, if $D_{\Theta}(0) = 0$, then by \cite{lelarge2019fundamental} we have $\lim_{n \rightarrow \infty} \MMSEsy_n(\mu_{\Lambda}, q_{\Theta}) = \E_{\bLambdas \sim \mu_{\Lambda}}[\bLambdas^2]^2$, which is achieved by $\mathbf{0}_{n \times n}$. Using \cref{thm:lower-bound} we deduce that $\lim_{n,d \rightarrow \infty} \MMSEas_n(\mu_{\Lambda}, \mu_{\Theta}) = \E_{\bLambdas \sim \mu_{\Lambda}}[\bLambdas^2]^2$, which concludes the proof of the theorem.


In the following parts of the proof we will assume $D_{\Theta}(0) > 0$. Let $a_1 < a_2$, $a_1' < a_2'$, and $a \sim \Unif[a_1,a_2]$ and $a' \sim \Unif[a_1',a_2']$. Similar to the derivation of \cref{eq:44} we have
%
%
\begin{align}\label{eq:C47}
	\lim\limits_{n,d \rightarrow \infty} \frac{1}{2n\sqrt{nd}}\E\left[ \| \bLambda_{1,a,a',h}\barbTheta_{1,a,a',h} - \bM_{1,0,0,h}  \|_F^2   \right]  = 0, 
\end{align}
where the expectation is taken over $a \sim \Unif[a_1,a_2]$ and $a' \sim \Unif[a_1',a_2']$. By \cref{eq:C45}, \cref{eq:C47} and triangle inequality, for all but countably many $q_{\Theta}^2$, $q_{\Theta}^2 + h$ we have 
\begin{align}
	& \limsup\limits_{n,d \rightarrow \infty} \left\{ \frac{1}{2\sqrt{n^3d}}\E\left[ \|\bLambda_{1,a,a',h}\barbTheta_{1,a,a',h} - \bM_{1,0,0,0}\|_F^2 \right]\right\}^{1/2} \nonumber\\
	\leq & \frac{1}{2\sqrt{n^3d}} \Big(\limsup\limits_{n,d \rightarrow \infty} \left\{ \E\left[ \| \bLambda_{1,a,a',h}\barbTheta_{1,a,a',h} - \bM_{1,0,0,h} \|_F^2   \right]\right\}^{1/2}  + \limsup\limits_{n,d \rightarrow \infty} \left\{ \E\left[\| \bM_{1,0,0,h}  - \bM_{1,0,0,0}  \|_F^2  \right]  \right\}^{1/2} \Big) \nonumber \\
	= & (D_{\Theta}(h) - D_{\Theta}(0))^{1/2}. \label{eq:C64}
\end{align} 
The following lemmas establish several useful properties of $\bLambda_{1,a,a',h}$ and $\barbTheta_{1,a,a',h}$.
\begin{lemma}\label{lemma:C14}
	Recall that $D$ is defined in \cref{eq:setD}. For all fixed $a_{\ast}, a_{\ast}' \in (0,5]$, $h \in [0,1]$, under the assumptions of \cref{thm:upper-bound} (b), if we further assume that $h + q_{\Theta}^2 \in D$, then as $n,d \to \infty$ we have
\begin{align*}
	& \frac{1}{n}\E\left[ \big\|\E[\bLambda \mid \bar\bA(1), \bx'(a_{\ast}'), \bar\bx(a_{\ast}), \bY'(h)]\big\|^2 \right] = 2\left(\frac{\partial}{\partial h}\sup_{q \geq 0} \cF(q_{\Theta}^2 + h, q)\right)^{1/2} + o_n(1), \\
	& \frac{1}{\sqrt{nd}} \E\left[ \|\E[\barbTheta \mid \bar\bA(1), \bx'(a'_{\ast}), \bar\bx(a_{\ast}), \bY'(h)]\|^2 \right] = 2q_{\Theta}^2\left(\frac{\partial}{\partial h}\sup_{q \geq 0} \cF(q_{\Theta}^2 + h, q)\right)^{1/2} + o_n(1).
\end{align*}

\end{lemma}
%
%
%
The proof of \cref{lemma:C14} is deferred to \cref{sec:proof-of-lemma:C14,sec:proof-of-lemma:C15}, respectively. 
By \cref{lemma:C14} and \cref{cor:overlap-concentration}, we conclude that for all $0 < a_1 < a_2 < 5$ and $0 < a_1' < a_2' < 5$, if we let $a \sim \Unif[a_1, a_2]$ and $a' \sim \Unif[a_1', a_2']$, then for all $h + q_{\Theta}^2 \in D$ we have
\begin{align}\label{eq:bTheta-conc}
	\frac{1}{\sqrt{nd}} \big\| \E[\barbTheta \mid \bar\bA(1), \bx'(a'), \bar\bx(a), \bY'(h)] \big\|^2 = 2q_{\Theta}^2\left(\frac{\partial}{\partial h}\sup_{q \geq 0} \cF(q_{\Theta}^2 + h, q)\right)^{1/2} + o_P(1). 
\end{align}
Let $C_0(q_{\Theta}, h) := 2q_{\Theta}^2\left(\frac{\partial}{\partial h}\sup_{q \geq 0} \cF(q_{\Theta}^2 + h, q)\right)^{1/2}$. 
We define the mapping $\bM: \RR^{n \times n} \times \RR^+ \mapsto \RR^{n \times n}$, such that $\bM(\bX; b)_{ij} = X_{ij} \mathbbm{1}\{|X_{ij}| \leq b\}$. We further define $\bM_0: \RR^{n \times d} \rightarrow \RR^{n \times n}$ such that for $\bX \in \RR^{n \times d}$
\begin{align*}
	\bM_0(\bX) = \bM\left(\mbox{$\frac{1}{\sqrt{nd}}$} \E[\bLambda \barbTheta^{\top} \mid \bar\bA = \bX]\E[ \barbTheta \barbLambda^{\top} \mid \bar\bA = \bX];\, 2K_1^2 C_0(q_{\Theta}, h)\right). 
\end{align*}
By triangle inequality
\begin{align}
	& \frac{1}{n}\E\left[ \|\bM_0(\bar\bA) - C_0(q_{\Theta}, h)\E[\bLambda \bLambda^{\top} \mid \bar\bA, \bY'(h)]\|_F^2 \right]^{1/2} \nonumber \\
	\leq & \frac{C_0(q_{\Theta}, h)}{n}\E\left[ \|\bLambda_{1,a,a',h}\bLambda_{1,a,a',h}^{\top}  - \E[\bLambda \bLambda^{\top} \mid \bar\bA, \bY'(h)]\|_F^2 \right]^{1/2} \nonumber \\
	& + \frac{1}{n}\E\left[ \|C_0(q_{\Theta}, h)\bLambda_{1,a,a',h}\bLambda_{1,a,a',h}^{\top} - \bM_0(\bar\bA)\|_F^2 \right]^{1/2}, \label{eq:C58}
\end{align}
where the expectation is taken over $(a,a',\bLambda, \barbTheta, \bZ, \bg, \bg', \bW')$. Using Gaussian integration by parts we obtain
\begin{align}
	 \frac{\partial }{\partial {a'}} \E\left[ \|\bLambda_{1,a,a',h}\|^2 \right] =  \ep_n'a' \E\left[ \|\E[\bLambda \bLambda^{\top} \mid \bar\bA, \bar\bx(a), \bar\bx'(a'), \bY'(h)] - \bLambda_{1,a,a',h} \bLambda_{1,a,a',h}^{\top}\|_F^2  \right]. \label{eq:C60}
\end{align}
By \cref{eq:C60} we have
\begin{align}\label{eq:C61}
	& \frac{1}{n^2}\E\left[ \|\bLambda_{1,a,a',h} \bLambda_{1,a,a',h} - \E[\bLambda \bLambda^{\top} \mid \bar\bA, \bx'(a'),\bar\bx(a), \bY'(h)]\|_F^2 \right] \nonumber \\
	= & \frac{1}{n^2(a_2 - a_1)(a_2'-a_1')} \int_{a_1}^{a_2} \int_{a_1'}^{a_2'} (\eps_n'a')^{-1} \frac{\partial }{\partial {a'}} \E\left[ \|\bLambda_{1,a,a',h}\|^2 \right] \dd a'\dd a \nonumber\\
	\leq & \frac{1}{n^2\eps_n'a_1'(a_2 - a_1)(a_2'-a_1')} \int_{a_1}^{a_2} \E\left[ \|\bLambda_{1,a,a_2',h}\|^2 \right]\dd a \nonumber \\
	 \leq & \frac{K_1^2}{n(a_2' - a_1')\eps_n'a_1'}, 
\end{align}
which vanishes as $n,d \to \infty$. For $\bX \in \RR^{n \times n}$, we define $\|\bX\|_1 := \sum_{i,j \in [n]} |X_{ij}|$. 
Then we have
\begin{align}\label{eq:C62}
	& \frac{1}{n^2}\E\left[ \|C_0(q_{\Theta}, h)\bLambda_{1,a,a',h}\bLambda_{1,a,a',h}^{\top} - {\bM_0}(\bar\bA)\|_F^2 \right] \nonumber\\
	\overset{(i)}{\leq} & \frac{3K_1^2 C_0(q_{\Theta}, h)}{n^2}\E\left[ \|C_0(q_{\Theta}, h)\bLambda_{1,a,a',h}\bLambda_{1,a,a',h}^{\top}- \bM_0(\bar\bA)\|_1 \right] \nonumber\\
	\overset{(ii)}{\leq} & \frac{3K_1^2 C_0(q_{\Theta}, h)}{n^2}\E\left[ \Big\|C_0(q_{\Theta}, h)\bLambda_{1,a,a',h}\bLambda_{1,a,a',h}^{\top} - \frac{1}{\sqrt{nd}}\bM_{1,0,0,0} \bM_{1,0,0,0}^{\top} \Big\|_1 \right] \nonumber\\
	\overset{(iii)}{\leq} & \frac{3K_1^2 C_0(q_{\Theta}, h)}{n^2}\E\left[ \Big\|\frac{\|\barbTheta_{1,a,a',h}\|^2}{\sqrt{nd}}\bLambda_{1,a,a',h}\bLambda_{1,a,a',h}^{\top} - \frac{1}{\sqrt{nd}}\bM_{1,0,0,0}\bM_{1,0,0,0}^{\top} \Big\|_1 \right] \\
	& + 3K_1^4 C_0(q_{\Theta}, h)\E\left[ \left| \frac{\|\barbTheta_{1,a,a',h}\|^2}{\sqrt{nd}} - C_0(q_{\Theta}, h) \right| \right],  
\end{align}
where in \emph{(i)} we use the assumption that $\mu_{\Lambda}$ has bounded support, and \emph{(ii)} is by the fact that for all $|x| \leq C_0(q_{\Theta}, h)K_1^2$, $y \in \RR$, 
\begin{align*}
	\big|x - y\mathbbm{1}\{|y| \leq 2K_1^2 C_0(q_{\Theta, h})\}\big| \leq \big|x - y\big|.
\end{align*}
Lastly, \emph{(iii)} is by triangle inequality. Applying \cref{lemma:A14} and \Holder's inequality we see that 
\begin{align}
	& \frac{3K_1^2 C_0(q_{\Theta}, h)}{n^2} \times \E\Big[ \Big\|\frac{\|\barbTheta_{1,a,a',h}\|^2}{\sqrt{nd}}\bLambda_{1,a,a',h} \bLambda_{1,a,a',h}^{\top} - \frac{1}{\sqrt{nd}}\bM_{1,0,0,0} \bM_{1,0,0,0}^{\top} \Big\|_1 \Big] \nonumber \\
	\leq & \frac{6K_1^2 C_0(q_{\Theta}, h)}{n\sqrt{nd}} \times  \E\left[\big\| \bM_{1,0,0,0} - \bLambda_{1,a,a',h} \barbTheta_{1,a,a',h}^{\top} \big\|_F^2\right]^{1/2} \nonumber \\
	&  \times  \Big( \E\left[ \big\| \bM_{1,0,0,0} \big\|_F^2 \right]^{1/2} + \E\left[ \big\| \bLambda_{1,a,a',h}\barbTheta_{1,a,a',h}^{\top} \big\| _F^2\right]^{1/2} \Big). \label{eq:C63}
\end{align}
By \cref{lemma:G1}, for all $a,a' \in [0,5]$, as $n,d \rightarrow \infty$ we have $\bar\Phi_n(1,a,a',h) \rightarrow \sup_{q \geq 0} \cF(q_{\Theta}^2 + h, q)$. Notice that the mapping $h \mapsto \bar\Phi_n(1,a,a',h)$ is convex and differentiable, thus for $q_{\Theta}^2 + h \in D$ we can apply \cref{lemma:convex-derivative} and conclude that 
\begin{align}\label{eq:54}
	\lim_{n,d \rightarrow \infty}\frac{1}{4n^2}\E[\| \E[\bLambda\bLambda^{\top} \mid \bar\bA, \bx'(a'), \bar{\bx}(a), \bY'(h)] \|_F^2] = \frac{\partial}{\partial h} \sup_{q \geq 0}\cF(q_{\Theta}^2 + h, q).
\end{align}
Leveraging triangle inequality, \cref{eq:C45,eq:C64}, we obtain that for all $a,a' \in [0,5]$, $q_{\Theta}^2 + h, q_{\Theta}^2 \in D$
\begin{align}\label{eq:75}
\begin{split}
	 & \limsup_{n,d \rightarrow \infty}\frac{1}{n\sqrt{nd}}\E\big[ \big\| \bM_{1,0,0,0} \big\|_F^2 \big] = 2D_{\Theta}(0), \\
	 & \limsup_{n,d \rightarrow \infty}\frac{1}{n\sqrt{nd}}\E\big[ \big\| \bLambda_{1,a,a',h}\barbTheta_{1,a,a',h}^{\top} \big\| _F^2\big] = 2D_{\Theta}(h), \\
	 &\lim_{n,d \rightarrow \infty} \frac{1}{n\sqrt{nd}}\E\big[\big\| \bM_{1,0,0,0}  - \bLambda_{1,a,a',h}\barbTheta_{1,a,a',h}  \big\|_F^2\big] = 2(D_{\Theta}(h) - D_{\Theta}(0)).
\end{split}
\end{align}
By \cref{thm:overlap-concentration} and \cref{eq:bTheta-conc} we have
\begin{align}\label{eq:76}
	\limsup_{n,d \rightarrow \infty}\E\left[ \left| \frac{1}{\sqrt{nd}}{\|\barbTheta_{1,a,a',h}\|^2} - C_0(q_{\Theta}, h) \right| \right] = 0.
\end{align}
We plug \cref{eq:C63,eq:75,eq:76} into \cref{eq:C62} and obtain that 
\begin{align}\label{eq:77}
	& \limsup_{n,d \rightarrow \infty} \frac{1}{n^2}\E\left[ \|C_0(q_{\Theta}, h)\bLambda_{1,a,a',h}\bLambda_{1,a,a',h}^{\top} - \bM_0(\bar\bA)\|_F^2 \right] \nonumber \\
	 \leq & 24K_1^2 C_0(q_{\Theta}, h)\times \big(D_{\Theta}(h) - D_{\Theta}(0)\big)^{1/2}.
\end{align}
Using triangle inequality
\begin{align*}
	& \frac{1}{n}\E\left[ \|\bLambda\bLambda^{\top} - C_0(q_{\Theta}, h)^{-1} \bM_0(\bar\bA)\|_F^2  \right]^{1/2} \\
	\leq & \frac{1}{n} \E\left[ \| \bLambda \bLambda^{\top} - \E[\bLambda \bLambda^{\top} \mid \bar\bA, \bx'(a'), \bar\bx(a), \bY'(h)] \|_F^2 \right]^{1/2} + \frac{1}{n} \E\left[ \|\bLambda_{1,a,a',h} \bLambda_{1,a,a',h}^{\top} - C_0(q_{\Theta}, h)^{-1} \bM_0(\bar\bA)\|_F^2 \right]^{1/2}  \\
	& + \frac{1}{n} \E\left[ \|\E[\bLambda \bLambda^{\top} \mid \bar\bA, \bx'(a'), \bar\bx(a), \bY'(h)] - \bLambda_{1,a,a',h} \bLambda_{1,a,a',h}^{\top}\|_F^2 \right]^{1/2}.
\end{align*}
We plug \cref{eq:54,eq:C61,eq:77} into the above equation and conclude that
%
\begin{align*}
	& \limsup_{n,d \rightarrow \infty}\frac{1}{n^2}\E\left[ \|\bLambda\bLambda^{\top} - C_0(q_{\Theta}, h)^{-1} \hat{\bM}(\bar\bA)\|_F^2  \right] \\
	\leq & \E_{\bLambdas \sim \mu_{\Lambda}}[\bLambdas^2]^2 - 4 \frac{\partial}{\partial h} \sup_{q \geq 0}\cF(q_{\Theta}^2 + h, q) + 2K_1^2\sqrt{24K_1^2 C_0(q_{\Theta}, h)^{-1}(D_{\Theta}(h) - D_{\Theta}(0))^{1/2}} \\
	& + 24K_1^2 C_0(q_{\Theta}, h)^{-1}(D_{\Theta}(h) - D_{\Theta}(0))^{1/2},  
\end{align*} 
which is an upper bound for $\limsup_{n,d \rightarrow \infty} \E\left[\left\| \bLambda \bLambda^{\top} - \E[\bLambda \bLambda^{\top} \mid \bar\bA]\right\|_F^2 \right] / n^2$. Recall that $D_{\Theta}(0) > 0$, thus $C_0(q_{\Theta}, h)^{-1} < C_0(q_{\Theta}, 0)^{-1} < \infty$. For all but countably many $q_{\Theta} > 0$ the mapping $h \mapsto D_{\Theta}(h)$ is continuous at $0$. For these $q_{\Theta}$, if we take $h \to 0^+$ while maintaining $h + q_{\Theta}^2 \in D$ then we derive that 
%
%
%
\begin{align*}
	\limsup_{n,d \rightarrow \infty}\frac{1}{n^2} \E\left[\left\| \bLambda \bLambda^{\top} - \E[\bLambda \bLambda^{\top} \mid \bar\bA]\right\|_F^2 \right] \leq  \E_{\bLambdas \sim \mu_{\Lambda}}[\bLambdas^2]^2- 4 \frac{\partial}{\partial h} \sup_{q \geq 0}\cF(q_{\Theta}^2 + h, q) \Big|_{h = 0},
\end{align*}
thus concludes the proof of the theorem using \cref{lemma:log-truncate}.

\subsubsection{Proof of Theorem \ref{thm:upper-bound} under condition (c)}

We define $\tilde \bY = (\bA\bA^{\top} - d\id_n) / {\sqrt{nd}}$ and $\tilde{\bY}' = {q_{\Theta}} \bLambda \bLambda^{\top} / n + (\bZ^{\top} \bZ - d \id_n) / {\sqrt{nd}}$. One can verify that as $n,d \rightarrow \infty$ we have $\|\tilde \bY - \tilde{\bY}'\|_{op} = o_P(1)$. We then run rotationally invariant AMP with spectral initialization based on $\tilde{\bY}$. According to \cite{mondelli2021pca}, for large enough number of iterations this algorithm achieves Bayesian MMSE, thus completing the proof of the theorem under condition (c). 



\section{Convergence of free energy density}\label{sec:tec-lemmas-lower-bound}
\subsection{Proof of  \cref{lemma:free-energy-1}}\label{sec:proof-free-energy-density1}

For $k \in [d]$, we define
\begin{align*}
	H_n^{(k)}(\vlambda, \vtheta) := &  \frac{1}{\sqrt{nd}}\sum\limits_{i = 1}^n \sum\limits_{j = 1, j \neq k}^d \bLambda_i\vlambda_i \bTheta_j \vtheta_j + \frac{1}{\sqrt[4]{nd}}\sum\limits_{i = 1}^n \sum\limits_{j = 1, j \neq k}^d Z_{ij} \vlambda_i \vtheta_j- \frac{1}{2\sqrt{nd}}\sum\limits_{i = 1}^n \sum\limits_{j = 1, j \neq k}^d \vlambda_i^2 \vtheta_j^2 \\
	& + H_n(\vlambda ; \bY'(h)) + H_n(\vlambda ; \bx'(s)).
\end{align*}
Furthermore, we introduce the following distributions
\begin{align*}
	& \mu_n^{(k,+)}(\dd\vlambda, \dd\vtheta) := \frac{1}{Z_n^{(k,+)}} \exp\left( H_n^{(k)}(\vlambda, \vtheta) \right) \tensorl P_{\Theta, k}(\dd \vtheta), \\
	& Z_n^{(k, +)} := \int \exp\left( H_n^{(k)}(\vlambda, \vtheta) \right) \tensorl  P_{\Theta, k}(\dd\vtheta), \\
	& \mu_n^{(k,-)}(\dd\vlambda, \dd\vtheta) := \frac{1}{Z_n^{(k,-)}} \exp\left( H_n^{(k)}(\vlambda, \vtheta)\right) \tensorl  P_{\Theta, k - 1}(\dd\vtheta), \\
	& Z_n^{(k, -)} := \int \exp\left( H_n^{(k)}(\vlambda, \vtheta) \right) \tensorl  P_{\Theta, k - 1}(\dd\vtheta).
\end{align*}
Note that $\mu_n^{(k,+)}$, $ \mu_n^{(k,-)}$, $Z_n^{(k, +)}$ and $Z_n^{(k, -)}$ are all random objects. The following lemma is a straightforward consequence of the above definitions. 
\begin{lemma}\label{lemma:some-facts}
	The following statements are true for all $k \in [d]$:
	\begin{enumerate}
	\item $Z_n^{(k, +)} = Z_n^{(k, -)}$.
	\item We denote by $\bZ_{\cdot k} \in \RR^n$ the $k$-th column of $\bZ$, then $(\mu_n^{(k,+)}, \mu_n^{(k,-)}, Z_n^{(k,+)}, Z_n^{(k,-)})$ are independent of $(\bZ_{\cdot k}, \bTheta_k)$.
	\item  We let $\vtheta_{-k} = (\vtheta_1, \vtheta_2, \cdots, \vtheta_{k - 1}, \vtheta_{k + 1}, \cdots, \vtheta_d) \in \RR^{d - 1}$. For $(\vlambda, \vtheta) \sim\mu_n^{(k, +)}$, we have $\vtheta_k \sim \mu_{\Theta}$ and is independent of $(\vtheta_{-k}, \vlambda)$. Similarly, for $(\vlambda, \vtheta) \sim\mu_n^{(k, -)}$, we have $\vtheta_k \sim \normal(0, q_{\Theta})$, and is independent of $(\vtheta_{-k}, \vlambda)$.
\end{enumerate}
\end{lemma}
%
%
We define 
\begin{align*}
	h_n^{(k)}(\vlambda, \vtheta) := \frac{1}{\sqrt{nd}} \sum_{i = 1}^n \bLambda_i\vlambda_i \bTheta_k \vtheta_k + \frac{1}{\sqrt[4]{nd}}\sum_{i = 1}^n Z_{ik} \vlambda_i \vtheta_k - \frac{1}{2\sqrt{nd}} \sum_{i = 1}^n \vlambda_i^2 \vtheta_k^2. 
\end{align*}
%
For some random variable $X$, we denote by $ \mu_n^{(k, +)}[X]$, $ \mu_n^{(k, -)}[X]$  the  expectations of $X$ evaluated under distributions $\mu_n^{(k, +)}$ and $\mu_n^{(k, -)}$, respectively. Using \cref{lemma:some-facts}, we obtain that
\begin{align}
	& \Phi_n^{(k)} - \Phi_n^{(k - 1)} \nonumber \\
	= & \frac{1}{n} \E\left[ \log \left( \int \exp(h_n^{(k)}(\vlambda, \vtheta))\mu_n^{(k,+)}(\dd\vlambda, \dd\vtheta) Z_n^{(k,+)} \right) \right] - \frac{1}{n} \E\left[ \log \left( \int \exp(h_n^{(k)}(\vlambda, \vtheta))\mu_n^{(k,-)}(\dd\vlambda, \dd\vtheta) Z_n^{(k,-)} \right) \right] \nonumber  \\
	\overset{(i)}{=}& \frac{1}{n} \E\left[ \log \left( \mu_n^{(k, +)}\left[\exp(h_n^{(k)}(\vlambda, \vtheta)  )\right]\right)  \right] - \frac{1}{n}\E\left[ \log \left( \mu_n^{(k, -)}\left[\exp(h_n^{(k)}(\vlambda, \vtheta)  ) \right] \right)\right],\label{eq:diff-of-Phi}
\end{align}
where \emph{(i)} is by result 1 in \cref{lemma:some-facts}. We consider the following Taylor expansion:
\begin{align}\label{eq:exp}
\exp\left(h_n^{(k)}(\vlambda, \vtheta) \right) = \sum\limits_{l = 0}^{\infty} \frac{1}{l!}h_n^{(k)}(\vlambda, \vtheta)^l = 1 + \sum\limits_{l = 1}^4 c_l^{(k)} \vtheta_k^l + R^{(k)} + \sum\limits_{l = 5}^{\infty} \frac{1}{l!}h_n^{(k)}(\vlambda, \vtheta)^l,
\end{align}
where 
\begin{align*}
	 c_1^{(k)} =& \frac{1}{\sqrt{nd}} \langle \bLambda, \vlambda \rangle \bTheta_k + \frac{1}{\sqrt[4]{nd}} \langle \bZ_{\cdot k}, \vlambda \rangle, \\
	 c_2^{(k)} =& \frac{1}{2}\left( \frac{1}{\sqrt{nd}} \langle \bLambda, \vlambda \rangle \bTheta_k + \frac{1}{\sqrt[4]{nd}} \langle \bZ_{\cdot k}, \vlambda\rangle \right)^2 - \frac{1}{2\sqrt{nd}} \langle \vlambda, \vlambda \rangle, \\
	 c_3^{(k)} =& -\frac{1}{2nd}  \langle \bLambda, \vlambda \rangle  \langle \vlambda, \vlambda \rangle \bTheta_k  - \frac{1}{2n^{3/4}d^{3/4}}  \langle  \bZ_{\cdot k} , \vlambda \rangle \langle  \vlambda, \vlambda \rangle  + \frac{1}{6} \left( \frac{1}{\sqrt{nd}} \langle \bLambda, \vlambda \rangle \bTheta_k + \frac{1}{\sqrt[4]{nd}} \langle \bZ_{\cdot k}, \vlambda \rangle \right)^3, \\
	c_4^{(k)} = & \frac{1}{8nd}\langle \vlambda, \vlambda \rangle^2 - \frac{1}{4\sqrt{nd}} \langle \vlambda, \vlambda \rangle\left( \frac{1}{\sqrt{nd}} \langle \bLambda, \vlambda \rangle \bTheta_k + \frac{1}{\sqrt[4]{nd}} \langle \bZ_{\cdot k}, \vlambda  \rangle \right)^2  \\
	& + \frac{1}{24}\left(\frac{1}{\sqrt{nd}} \langle \bLambda, \vlambda \rangle \bTheta_k + \frac{1}{\sqrt[4]{nd}} \langle \bZ_{\cdot k}, \vlambda \rangle \right)^4, \\
	R^{(k)}  = & - \frac{\langle \vlambda, \vlambda \rangle^3}{48(nd)^{3/2}}  \vtheta_k^6 + \frac{1}{8nd} \langle \vlambda, \vlambda \rangle^2 \left(\frac{1}{\sqrt{nd}} \langle \bLambda, \vlambda \rangle \bTheta_k + \frac{1}{\sqrt[4]{nd}} \langle \bZ_{\cdot k}, \vlambda \rangle \right) \vtheta_k^5 + \frac{1}{384n^2d^2} \langle \vlambda, \vlambda \rangle^4 \vtheta_k^8\\
	&  - \frac{1}{12\sqrt{nd}} \langle \vlambda, \vlambda \rangle   \left( \frac{1}{\sqrt{nd}} \langle \bLambda, \vlambda \rangle \bTheta_k + \frac{1}{\sqrt[4]{nd}} \langle \bZ_{\cdot k}, \vlambda \rangle \right)^3 \vtheta_k^5 \\
	& + \frac{1}{16nd} \langle \vlambda, \vlambda \rangle^2 \left( \frac{1}{\sqrt{nd}} \langle \bLambda, \vlambda \rangle \bTheta_k + \frac{1}{\sqrt[4]{nd}} \langle \bZ_{\cdot k}, \vlambda \rangle \right)^2 \vtheta_k^6\\
	&   -\frac{1}{48(nd)^{3/2}}  \langle \vlambda, \vlambda \rangle^3 \left( \frac{1}{\sqrt{nd}} \langle \bLambda, \vlambda \rangle \bTheta_k + \frac{1}{\sqrt[4]{nd}} \langle \bZ_{\cdot k} , \vlambda \rangle \right) \vtheta_k^7.
\end{align*}
%
The next lemma characterizes convergence of power series:
\begin{lemma}\label{lemma:finite-sum}
	For $n,d$ large enough, the following quantities almost surely exist and are finite:
	\begin{align*}
		\sum\limits_{l = 5}^{\infty} \frac{1}{l!}\mu_n^{(k,+)}\left[ |h_n^{(k)}(\vlambda, \vtheta)|^l\right], \qquad \sum\limits_{l = 5}^{\infty} \frac{1}{l!}\mu_n^{(k,-)}\left[ |h_n^{(k)}(\vlambda, \vtheta)|^l\right].
	\end{align*}
\end{lemma}
\begin{proof}
	We will only prove the lemma for $\mu_n^{(k,-)}$. Proof for $\mu_n^{(k,+)}$ is analogous and we skip it for simplicity. By the power mean inequality we have
	\begin{align*}
		 & \sum\limits_{l = 5}^{\infty} \frac{1}{l!} |h_n^{(k)}(\vlambda, \vtheta)|^l \\
		 \leq & \sum\limits_{l = 5}^{\infty} \frac{3^l}{l!} \left\{ \Big|\frac{1}{\sqrt{nd}}\langle \bLambda, \vlambda \rangle \bTheta_k \vtheta_k \Big|^l + \Big| \frac{1}{\sqrt[4]{nd}}\langle \bZ_{\cdot k}, \vlambda \rangle \vtheta_k\Big|^l + \Big| \frac{1}{2\sqrt{nd}} \langle \vlambda, \vlambda \rangle \vtheta_k^2 \Big|^l \right\}.
	\end{align*}
	Next, we take the expectation of the last line above with respect to $\mu_n^{(k,-)}$, which gives
	\begin{align*}
		& \sum\limits_{l = 5}^{\infty} \frac{1}{l!}\mu_n^{(k,-)}\left[ |h_n^{(k)}(\vlambda, \vtheta)|^l\right] \\
		\overset{(i)}{\leq} & \sum_{l = 5}^{\infty}\left\{ \frac{3^ln^{l / 2}K^{2l}|\bTheta_k|^l l!!q_{\Theta}^{l / 2}}{l!d^{l / 2}} + \frac{3^ln^{3l/4}\|\bZ_{\cdot k}\|_{\infty}^lK^l l!!q_{\Theta}^{l / 2}}{l!d^{l/4}} + \frac{3^ln^{l / 2}K^{2l}(2l)!!q_{\Theta}^{l}}{l!d^{l / 2}}\right\} \overset{(ii)}{<}   \infty,
	\end{align*}
where \emph{(i)} is by Assumption \ref{assumption:bounded-support} and the third result of \cref{lemma:some-facts}. In order to prove \emph{(ii)}, we only need to use the following fact: 
For $n,d$ large enough we have
	\begin{align*}
		\frac{6q_{\Theta}K^{2}n^{1/2}}{d^{1/2}} < 1.
	\end{align*}
\end{proof}
According to \cref{lemma:finite-sum}, we can take the expectation of  \cref{eq:exp} with respect to $\mu_n^{(k,+)}$, which gives 
\begin{align}\label{eq:17}
	& \mu_n^{(k, +)}\left[ \exp\left(h_n^{(k)}(\vlambda, \vtheta) \right) \right] \nonumber \\
	 = & 1 + \mu_n^{(k, +)}[c_2^{(k)}]q_{\Theta} + \mu_n^{(k, +)}[c_4^{(k)}]\E[\bThetas^4] + \mu_n^{(k, +)}[R^{(k)}] +\mu_n^{(k, +)} \Big[ \sum\limits_{l = 5}^{\infty} \frac{1}{l!}h_n^{(k)}(\vlambda, \vtheta)^l\Big].
\end{align}
In the above derivation, we use the fact that under $\mu_n^{(k,+)}$, we have $\vtheta_k \overset{d}{=} \mu_{\Theta}$, which has zero first and third moments.  Using Assumption \ref{assumption:bounded-support}, we conclude that  $c_2^{(k)} \geq -\frac{1}{2} \sqrt{\frac{n}{d}}K^2$. Furthermore, notice that
\begin{align*}
	 \frac{3}{8nd} \langle \vlambda, \vlambda \rangle^2- \frac{1}{4\sqrt{nd}}\langle \vlambda, \vlambda \rangle \left( \frac{1}{\sqrt{nd}} \langle \bLambda, \vlambda  \rangle \bTheta_k + \frac{1}{\sqrt[4]{nd}} \langle \bZ_{\cdot k}, \vlambda \rangle \right)^2 +\frac{1}{24}\left(\frac{1}{\sqrt{nd}} \langle \bLambda, \vlambda \rangle \bTheta_k + \frac{1}{\sqrt[4]{nd}} \langle \bZ_{k \cdot}, \vlambda \rangle \right)^4 
\end{align*}
is non-negative, thus we have $c_4^{(k)} \geq -\frac{1}{4nd} \langle \vlambda, \vlambda \rangle^2 \geq -\frac{n}{4d} K^4$. Since $\vtheta_k$ has zero expectation under $\mu_n^{(k,+)}$, we then conclude that $\mu_n^{(k,+)}[h_n^{(k)}(\vlambda, \vtheta)] = -\frac{1}{2\sqrt{nd}}\mu_n^{(k,+)}[\langle \vlambda, \vlambda \rangle \vtheta_k^2] \geq - \frac{1}{2}\sqrt{\frac{n}{d}}K^2 q_{\Theta}$. By Jensen's inequality
\begin{align*}
	\mu_n^{(k, +)}\left[ \exp\left(h_n^{(k)}(\vlambda, \vtheta) \right) \right] \geq \exp\left( \mu_n^{(k, +)}\left[ h_n^{(k)}(\vlambda, \vtheta) \right] \right) \geq \exp\left( - \frac{1}{2}\sqrt{\frac{n}{d}}K^2 q_{\Theta} \right).
\end{align*}
Note that the function $x \mapsto \log(x)$ is concave. We next plug the lower bounds derived above into \cref{eq:17}, and obtain 
\begin{align}
	& \left|\log\left( \mu_n^{(k, +)}\left[ \exp\left(h_n^{(k)}(\vlambda, \vtheta) \right) \right] \right) - \log\left( 1 + \mu_n^{(k, +)}[c_2^{(k)}]q_{\Theta} + \mu_n^{(k, +)}[c_4^{(k)}]\E[\bThetas^4] \right)\right|\nonumber \\
	\leq & \left| \sum\limits_{l = 5}^{\infty} \mu_n^{(k, +)} \Big[\frac{1}{l!}h_n^{(k)}(\vlambda, \vtheta)^l\Big] + \mu_n^{(k, +)}[R^{(k)}] \right| \times \nonumber \\
	& \max\left\{ {\mu_n^{(k, +)}\left[ \exp\left(h_n^{(k)}(\vlambda, \vtheta) \right) \right]}^{-1},\, \left( 1 + \mu_n^{(k, +)}[c_2^{(k)}]q_{\Theta} + \mu_n^{(k, +)}[c_4^{(k)}]\E[\bThetas^4] \right)^{-1} \right\} \nonumber\\
	\leq & \underbrace{\left|\sum\limits_{l = 5}^{\infty} \mu_n^{(k, +)} \Big[ \frac{1}{l!}h_n^{(k)}(\vlambda, \vtheta)^l\Big] + \mu_n^{(k, +)}[R^{(k)}] \right|}_{I} \times \nonumber  \\
	& \underbrace{\max\left\{ \exp\left(  \frac{1}{2}\sqrt{\frac{n}{d}}K^2 q_{\Theta} \right), \left( 1  -\frac{1}{2} \sqrt{\frac{n}{d}}K^2q_{\Theta} -\frac{n}{4d}K^4\E[\bThetas^4] \right)^{-1}    \right\}}_{II}. \label{eq:19}
\end{align}
Since $d \gg n$ and $(K, q_{\Theta}, \E[\bThetas^4])$ are independent of $(n,d)$, we obtain that term II above converges to 1 as $n,d \rightarrow \infty$. Next, we will provide an upper bound for term I. To this end, we upper bound $\E[ |\mu_n^{(k, +)} [ \sum_{l = 5}^{\infty} \frac{1}{l!}h_n^{(k)}(\vlambda, \vtheta)^l ] |  ]$ and $\E[|\mu_n^{(k, +)}[R^{(k)}]|]$ in \cref{sec:upper-bound1} and \cref{sec:upper-bound2}, respectively, and combine them to finish the proof in \cref{sec:upper-bound3}. 

\subsubsection{Upper bounding $\E[ |\mu_n^{(k, +)} [ \sum_{l = 5}^{\infty} \frac{1}{l!}h_n^{(k)}(\vlambda, \vtheta)^l ] |  ]$}\label{sec:upper-bound1}

Since $\mu_{\Theta}$ is sub-Gaussian, there exists a constant $C > 0$ depending only on $\mu_{\Theta}$, such that for all $p \in \NN_+$, $\E_{\bThetas \sim \mu_{\Theta}}[|\bThetas|^p] \leq C^p p^{p/2}$ and $\E_{G \sim \normal(0, 1)}[|G|^p] \leq C^p p^{p/2}$. Then for $n,d$ large enough, we have
\begin{align}
	& \E \left[ \left|\sum\limits_{l = 5}^{\infty} \frac{1}{l!}\mu_n^{(k, +)} \Big[ h_n^{(k)}(\vlambda, \vtheta)^l\Big] \right|  \right] \nonumber \\
	\overset{(i)}{\leq} & \sum\limits_{l = 5}^{\infty} \frac{1}{l!} \E\left[ \mu_n^{(k, +)}\left[ \left|  \frac{1}{\sqrt{nd}} \langle \bLambda, \vlambda \rangle \bTheta_k \vtheta_k + \frac{1}{\sqrt[4]{nd}}\langle \bZ_{\cdot k}, \vlambda \rangle \vtheta_k - \frac{1}{2\sqrt{nd}}\langle \vlambda, \vlambda \rangle \vtheta_k^2 \right|^l \right]   \right]  \nonumber\\
	\overset{(ii)}{\leq} & \sum\limits_{l = 5}^{\infty} \frac{1}{l!}\E\left[ \mu_n^{(k, +)}\left[ \left|  \frac{1}{\sqrt{nd}} \langle \bLambda, \vlambda \rangle \bTheta_k \vtheta_k + \frac{1}{\sqrt[4]{nd}} \langle \bZ_{\cdot k}, \vlambda \rangle \vtheta_k - \frac{1}{2\sqrt{nd}}\langle \vlambda, \vlambda \rangle \vtheta_k^2 \right|^{2l}\right]   \right]^{1/2} \nonumber \\
	\overset{(iii)}{\leq} & \sum\limits_{l = 5}^{\infty} \frac{3^l}{l!}\times \left\{ \E\left[\mu_n^{(k, +)}\left[\left| \frac{1}{\sqrt{nd}} \langle \bLambda, \vlambda \rangle\bTheta_k \vtheta_k \right|^{2l} \right] \right]^{1/2} + \E\left[\mu_n^{(k, +)}\left[\left| \frac{1}{\sqrt[4]{nd}} \langle \bZ_{\cdot k}, \vlambda \rangle \vtheta_k \right|^{2l} \right] \right]^{1/2} + \right.  \nonumber\\
	& \left. \E\left[\mu_n^{(k, +)}\left[\left| \frac{1}{2\sqrt{nd}}\langle \vlambda, \vlambda \rangle \vtheta_k^2 \right|^{2l} \right] \right]^{1/2}  \right\} \nonumber \\
	\overset{(iv)}{\leq} & \sum\limits_{l = 5}^{\infty} \frac{3^l}{l!}\times \left\{\frac{K^{2l}C^{2l}n^{l/2}}{d^{l/2}} \times (2l)^l + \frac{K^lC^{2l}n^{l/4}}{d^{l/4}} \times (2l)^l + \frac{K^{2l}C^{2l}n^{l/2}}{d^{l/2}}\times (2l)^l \right\} \nonumber \\
	\overset{(v)}{\leq} & F_{K,C} \times \frac{n^{5/4}}{d^{5/4}}, \label{eq:20}
\end{align}
where $F_{K,C}>0$ is a constant that depends only on $K$ and $C$. In the above inequalities, \emph{(i)} is by triangle inequality, \emph{(ii)} is by \Holder's inequality and \emph{(iii)} is by power mean inequality. Argument \emph{(iv)} is via a combination of the following facts: (1) Support$(\Lambda) \subseteq [-K, K]$, (2) $\mu_{\Theta}$ is sub-Gaussian, (3) the random distribution $\mu_n^{(k,+)}$ is independent of  $(\bZ_{\cdot k}, \bTheta_k)$. 

For illustration, in the following parts of the proof we upper bound the second summand in the second to last line of \cref{eq:20}. The proofs for the first and third summands follow analogously. 

By \cref{lemma:some-facts}, we see that $\mu_n^{(k,+)}$ is independent of $\bZ_{\cdot k}$, and $\vtheta_k$ is independent of $\vlambda$ under $\mu_n^{(k,+)}$. Therefore, we have
\begin{align}\label{eq:18}
	\E\left[\mu_n^{(k, +)}\left[\left| \frac{1}{\sqrt[4]{nd}} \langle \bZ_{\cdot k}, \vlambda \rangle \vtheta_k \right|^{2l} \right] \right] = \sum\limits_{s_1 = 1}^n \cdots \sum\limits_{s_{2l} = 1}^n  \E\left[Z_{s_1k}Z_{s_2k} \cdots Z_{s_{2l}k}\right] C_{s_1,s_2,\cdots,s_{2l}} , 
\end{align}
where $C_{s_1,s_2, \cdots, s_{2l}} = {n^{-l/2}d^{-l/2}} \E[ \mu_n^{(k,+)}[\lambda_{s_1}\lambda_{s_2} \cdots \lambda_{s_{2l}}]]\E[\bThetas^{2l}]$. By sub-Gaussian property we have 
$$|C_{s_1,s_2, \cdots, s_{2l}}| \leq K^{2l}C^{2l}(2l)^l n^{-l/2}d^{-l/2}.$$
Consider all terms that take the form of $\E\left[Z_{s_1k}Z_{s_2k} \cdots Z_{s_{2l}k}\right]$. If such term is non-zero, then it must be positive. Using property of Gaussian distribution, we have
\begin{align*}
	 \sum\limits_{s_1 = 1}^n \cdots \sum\limits_{s_{2l} = 1}^n  \E\left[Z_{s_1k}Z_{s_2k} \cdots Z_{s_{2l}k}\right]= \E[(Z_{1k} + \cdots + Z_{nk})^{2l}] = n^l(2l - 1)!! \leq n^lC^{2l}(2l)^l.
\end{align*}
Therefore, the right hand side of \cref{eq:18} has value no larger than 
$$n^lC^{2l}(2l)^l \times K^{2l}C^{2l}(2l)^l n^{-l/2}d^{-l/2} = K^{2l}C^{4l}(2l)^{2l}n^{l/2}d^{-l/2},$$ 
which leads to the desired upper bound for the second summand. Finally, we use Stirling formula and the assumption that $d \gg n$ to prove argument \emph{(v)}.


\subsubsection{Upper bounding $\E[|\mu_n^{(k, +)}[R^{(k)}]|]$}\label{sec:upper-bound2}

Similar to the proof in \cref{sec:upper-bound1}, we conclude that for $n,d$ large enough
\begin{align}\label{eq:21}
	\E[|\mu_n^{(k,+)}[R^{(k)}]|] \leq F'_{K,C} \times \frac{n^{5/4}}{d^{5/4}},
\end{align}
where $F_{K,C}' > 0$ is a constant depending only on $K$ and $C$. The derivation of the above upper bound is similar to the derivation of the upper bound for $\E[ |\mu_n^{(k, +)} [ \sum_{l = 5}^{\infty} \frac{1}{l!}h_n^{(k)}(\vlambda, \vtheta)^l ] |  ]$ given in \cref{eq:20}, and we skip the details here for the sake of simplicity.
\subsubsection{Combining the upper bounds}\label{sec:upper-bound3}
%
Combining \cref{eq:19,eq:20,eq:21}, we obtain that for $n,d$ large enough
\begin{align}\label{eq:219}
	& \left|\sum\limits_{k = 1}^d \frac{1}{n} \E\left[ \log\left( \mu_n^{(k, +)}\left[ \exp\left(h_n^{(k)}(\vlambda, \vtheta) \right) \right] \right) - \log\left( 1 + \mu_n^{(k, +)}[c_2^{(k)}]q_{\Theta} + \mu_n^{(k, +)}[c_4^{(k)}]\E[\bThetas^4] \right) \right]\right| \nonumber \\
	\leq & 2(F_{K,C} + F_{K,C}') \times \frac{n^{1/4}}{d^{1/4}}.
\end{align}
In what follows, we show that the following quantity is small:
\begin{align}\label{eq:log-diff}
	\left|\sum\limits_{k = 1}^d \frac{1}{n} \E\left[ \log \left(1 + \mu_n^{(k, +)}[c_2^{(k)}]q_{\Theta} + \mu_n^{(k, +)}[c_4^{(k)}] \E[\bThetas^4] \right) \right] - \sum\limits_{k = 1}^d \frac{1}{n} \E\left[ \log \left(1 + \mu_n^{(k, +)}[c_2^{(k)}]q_{\Theta}  \right) \right]\right|.
\end{align}
Again we use the concavity of the mapping $x \mapsto \log(x)$, which gives
\begin{align}
	&  \log\left( 1 + \mu_n^{(k, +)}[c_2^{(k)}] q_{\Theta} \right) +  \mu_n^{(k, +)}[c_4^{(k)}]\E[\bThetas^4] - \frac{\Big|\mu_n^{(k, +)}[c_4^{(k)}]\E[\bThetas^4]\left(\mu_n^{(k, +)}[c_2^{(k)}] q_{\Theta} + \mu_n^{(k, +)}[c_4^{(k)}] \E[\bThetas^4]\right)\Big|}{1 + \mu_n^{(k, +)}[c_2^{(k)}] q_{\Theta} + \mu_n^{(k, +)}[c_4^{(k)}] \E[\bThetas^4]} \nonumber \\
	\leq & \log\left( 1 + \mu_n^{(k, +)}[c_2^{(k)}] q_{\Theta} \right) + \frac{\mu_n^{(k, +)}[c_4^{(k)}]\E[\bThetas^4]}{1 + \mu_n^{(k, +)}[c_2^{(k)}] q_{\Theta} + \mu_n^{(k, +)}[c_4^{(k)}] \E[\bThetas^4]} \nonumber\\
	\leq & \log \left(1 + \mu_n^{(k, +)}[c_2^{(k)}]q_{\Theta} + \mu_n^{(k, +)}[c_4^{(k)}] \E[\bThetas^4] \right) \leq \log\left( 1 + \mu_n^{(k, +)}[c_2^{(k)}] q_{\Theta} \right) + \frac{\mu_n^{(k, +)}[c_4^{(k)}]\E[\bThetas^4]}{1 + \mu_n^{(k, +)}[c_2^{(k)}]q_{\Theta} }\nonumber \\
	 \leq & \log\left( 1 + \mu_n^{(k, +)}[c_2^{(k)}] q_{\Theta} \right) +  \mu_n^{(k, +)}[c_4^{(k)}]\E[\bThetas^4] + \frac{|\mu_n^{(k, +)}[c_2^{(k)}]\mu_n^{(k, +)}[c_4^{(k)}]q_{\Theta}\E[\bThetas^4]|}{1 + \mu_n^{(k, +)}[c_2^{(k)}]q_{\Theta}}. \label{eq:220}
\end{align}
%
By \cref{lemma:some-facts}, $\mu_n^{(k,+)}$ and $\bZ_{\cdot k}$ are independent of each other, thus
\begin{align}\label{eq:24}
	\E\left[\mu_n^{(k, +)}\left[ \frac{1}{8nd} \langle \vlambda, \vlambda \rangle^2 - \frac{1}{4nd} \langle \vlambda, \vlambda \rangle \langle\bZ_{\cdot k}, \vlambda \rangle^2 + \frac{1}{24nd}\langle \bZ_{\cdot k}, \vlambda \rangle^4  \right]  \right] = 0.
\end{align}
Next, we plug \cref{eq:24} into the definition of $c_4^{(k)}$, then apply \cref{lemma:some-facts} claim 3, which gives
\begin{align}\label{eq:221}
	& \left| \E[\mu_n^{(k, +)}[c_4^{(k)}]]  \right| \nonumber \\
	=& \left| -\frac{\E\left[ \mu_n^{(k,+)}\left[ \langle \vlambda, \vlambda \rangle \langle \vlambda, \bLambda \rangle^2 \bTheta_k^2 \right] \right]}{4n^{3/2}d^{3/2}} + \frac{\E\left[ \mu_n^{(k,+)}\left[ \langle \bLambda, \vlambda \rangle^4 \bTheta_k^4 \right] \right]}{24n^2d^2}  + \frac{\E\left[ \mu_n^{(k,+)}\left[ \langle \vlambda, \bLambda \rangle^2 \langle \bZ_{\cdot k}, \vlambda \rangle^2 \bTheta_k^2 \right] \right]}{4n^{3/2}d^{3/2}} \right| \nonumber \\ 
	\leq & \frac{n^{3/2}}{2d^{3/2}}K^6q_{\Theta} + \frac{n^2}{24d^2}K^8 \E[\bThetas^4].
\end{align}
In addition, we have the following lemma:
\begin{lemma}\label{lemma:C3}
	There exist constants $A_1(K, \mu_{\Theta}), A_2(K, \mu_{\Theta}) > 0$, which are functions of $(K, \mu_{\Theta})$ only, such that
	\begin{align*}
		\E[\mu_n^{(k,+)}[|c_2^{(k)}|^2]] \leq A_1(K, \mu_{\Theta}) \times \frac{n}{d}, \qquad \E[\mu_n^{(k,+)}[|c_4^{(k)}|^2]] \leq A_2(K, \mu_{\Theta}) \times \frac{n^2}{d^2}. 
	\end{align*}
\end{lemma}
\begin{proof}
	Straightforward computation reveals that there exist $A_1'(K, \mu_{\Theta}), A_2'(K, \mu_{\Theta}) > 0$ depending only on $(K, \mu_{\Theta})$, such that
	\begin{align*}
		\E[\mu_n^{(k,+)}[|c_2^{(k)}|^2]] \leq &  A_1'(K, \mu_{\Theta})\E\left[ \mu_n^{(k,+)}\left[ \frac{1}{n^2d^2}\langle \bLambda, \vlambda \rangle^4 \Theta_k^4  +\frac{1}{{nd}} \langle \bZ_{\cdot k}, \vlambda \rangle^4 +\frac{1}{{nd}} \langle \vlambda, \vlambda \rangle^2 \right] \right] \\
		\E[\mu_n^{(k,+)}[|c_4^{(k)}|^2]] \leq & A_2'(K, \mu_{\Theta})\E\left[\mu_n^{(k,+)}\left[\frac{1}{n^2d^2}\langle \vlambda, \vlambda \rangle^4 +  \frac{1}{n^{3}d^{3}}\langle \vlambda, \vlambda \rangle^2 \langle  \bLambda, \vlambda \rangle^4 \bTheta_k^4  +   \frac{1}{n^2d^2}\langle \vlambda, \vlambda \rangle^2 \langle \bZ_{\cdot k}, \vlambda \rangle^4  + \right.\right. \\
		& \left. \left.  \frac{1}{n^4d^4} \langle \bLambda,\vlambda \rangle^8 \bTheta_k^8  + \frac{1}{n^2d^2} \langle \bZ_{\cdot k}, \vlambda \rangle^8  \right]\right]. 
	\end{align*}
	The rest of the proof follows from \cref{lemma:some-facts} and the assumption that $d \gg n$. 
\end{proof}
Recall that $c_2^{(k)} \geq -\frac{1}{2} \sqrt{\frac{n}{d}}K^2$ and $c_4^{(k)} \geq -\frac{n}{4d}K^4$. Then for $n,d$ large enough, using \cref{lemma:C3}, we obtain 
\begin{align}
	& \E\left[\frac{|\mu_n^{(k, +)}[c_2^{(k)}]\mu_n^{(k, +)}[c_4^{(k)}]q_{\Theta}\E[\bThetas^4]|}{1 + \mu_n^{(k, +)}[c_2^{(k)}]q_{\Theta}} \right] \leq \frac{q_{\Theta}\E[\bThetas^4]}{1 - \frac{1}{2} \sqrt{\frac{n}{d}}K^2q_{\Theta}} \times \E[\mu_n^{(k,+)}[|c_2^{(k)}|^2]]^{1/2}\E[\mu_n^{(k,+)}[|c_4^{(k)}|^2]]^{1/2}, \nonumber\\
	& \qquad \qquad \qquad \qquad \qquad \qquad \qquad \; \; \; \; \;  \leq \frac{2q_{\Theta}\E[\bThetas^4]A_1(K, \mu_{\Theta})^{1/2}A_2(K, \mu_{\Theta})^{1/2}n^{3/2}}{d^{3/2}}, \label{eq:222}
\end{align}
\begin{align}
	&\E\left[ \frac{|\mu_n^{(k, +)}[c_4^{(k)}]\E[\bThetas^4]\left(\mu_n^{(k, +)}[c_2^{(k)}] q_{\Theta} + \mu_n^{(k, +)}[c_4^{(k)}] \E[\bThetas^4]\right)|}{1 + \mu_n^{(k, +)}[c_2^{(k)}] q_{\Theta} + \mu_n^{(k, +)}[c_4^{(k)}] \E[\bThetas^4]} \right]  \nonumber  \\
	 \leq & \frac{\E[\bThetas^4]^2\E[\mu_n^{(k,+)}[|c_4^{(k)}|^2]]}{1 -\frac{1}{2} \sqrt{\frac{n}{d}}K^2q_{\Theta} -\frac{n}{4d} K^4 \E[\bThetas^4]}  + \frac{\E[\bThetas^4] q_{\Theta}\E[\mu_n^{(k,+)}[|c_2^{(k)}|^2]]^{1/2}\E[\mu_n^{(k,+)}[|c_4^{(k)}|^2]]^{1/2}}{1 -\frac{1}{2} \sqrt{\frac{n}{d}}K^2q_{\Theta} -\frac{n}{4d} K^4 \E[\bThetas^4]} \nonumber\\
	 \leq & \frac{2\E[\bThetas^4]^2 A_2(K, \mu_{\Theta})n^2}{d^2} +  \frac{2\E[\bThetas^4] q_{\Theta}A_1(K, \mu_{\Theta})^{1/2}A_2(K, \mu_{\Theta})^{1/2}n^{3/2}}{d^{3/2}}. \label{eq:223}	
\end{align}
Next, we plug \cref{eq:221,eq:222,eq:223} into \cref{eq:220}, then sum over $k \in [d]$. This implies the existence of $C(K, \mu_{\Theta}) > 0$, which is a constant depending only on $(K, \mu_{\Theta})$, such that for $n,d$ large enough
\begin{align}\label{eq:C87}
\begin{split}
	& \sum\limits_{k = 1}^d \frac{1}{n} \E\left[\log\left( 1 + \mu_n^{(k, +)}[c_2^{(k)}] q_{\Theta} \right) \right] - \frac{C(K, \mu_{\Theta})n^{1/2}}{d^{1/2}} \\
	\leq & \sum\limits_{k = 1}^d \frac{1}{n} \E\left[ \log \left(1 + \mu_n^{(k, +)}[c_2^{(k)}]q_{\Theta} + \mu_n^{(k, +)}[c_4^{(k)}] \E[\bThetas^4] \right) \right] \\
	 \leq & \sum\limits_{k = 1}^d \frac{1}{n} \E\left[\log\left( 1 + \mu_n^{(k, +)}[c_2^{(k)}] q_{\Theta} \right) \right] + \frac{C(K, \mu_{\Theta})n^{1/2}}{d^{1/2}}.
\end{split}
\end{align}
Combining \cref{eq:C87,eq:219}, we derive that
\begin{align}\label{eq:29}
	\sum\limits_{k = 1}^d \frac{1}{n} \E\left[\log\left( 1 + \mu_n^{(k, +)}[c_2^{(k)}] q_{\Theta} \right) \right] - \sum\limits_{k = 1}^d \frac{1}{n} \E\left[ \log\left( \mu_n^{(k, +)}\left[ \exp\left(h_n^{(k)}(\vlambda, \vtheta) \right) \right] \right)\right] = o_n(1).
\end{align}
Similarly, we can prove that
\begin{align}\label{eq:30}
	\sum\limits_{k = 1}^d \frac{1}{n} \E\left[\log\left( 1 + \mu_n^{(k, -)}[c_2^{(k)}] q_{\Theta} \right) \right] - \sum\limits_{k = 1}^d \frac{1}{n} \E\left[ \log\left( \mu_n^{(k, -)}\left[ \exp\left(h_n^{(k)}(\vlambda, \vtheta) \right) \right] \right)\right] = o_n(1).
\end{align}
Since $c_2^{(k)}$ is independent of $\vtheta_k$, by \cref{lemma:some-facts} we have $\mu_n^{(k, +)}[c_2^{(k)}] = \mu_n^{(k,-)}[c_2^{(k)}]$. Finally, we combine \cref{eq:diff-of-Phi,eq:29,eq:30}, which gives $\Phi_n^{(d)}- \Phi_n^{(0)} = o_n(1)$. Thus, we have completed the proof of \cref{lemma:free-energy-1}.

\subsection{Proof of \cref{lemma:free-energy-2}}\label{sec:proof-free-energy-density2}

In this section we prove \cref{lemma:free-energy-2}. Applying Gaussian integration by parts, we obtain that
\begin{align*}
	\Phi_n^{(0)}(h,s) = \frac{1}{n} \E\Big[ \log \Big( \int \exp \Big(\tilde{H}_n'(\vlambda; \bY'(h), \bx'(s)) + H_n(\vlambda; \bY'(h)) + H_n(\vlambda; \bx'(s))\Big)\tensorl \Big) \Big],
\end{align*}
where
\begin{align*}
	\tilde{H}_n'(\vlambda; \bY'(h), \bx'(s)) =&   \frac{\langle \bLambda, \vlambda \rangle^2 \|\bTheta\|^2 / (nd) + 2\langle \bLambda, \vlambda \rangle\langle \vlambda, \bZ \bTheta \rangle / (nd)^{3/4} + \|\bZ^{\top} \vlambda\|^2 / \sqrt{nd}}{2q_{\Theta}^{-1} + 2\|\vlambda\|^2 / \sqrt{nd}} \\
	& - \frac{d}{2}\log\Big( 1 + \frac{q_{\Theta}}{\sqrt{nd}}\|\vlambda\|^2 \Big).
\end{align*}
By triangle inequality,
\begin{align}
	 & \sup\limits_{\|\vlambda\|_{\infty} \leq K} \left| \frac{\langle \bLambda, \vlambda \rangle^2 \|\bTheta\|^2 / (nd)}{2q_{\Theta}^{-1} + 2\|\vlambda\|^2 / \sqrt{nd}} - \frac{q_{\Theta}^2}{2n}\langle \bLambda, \vlambda \rangle^2   \right| \nonumber \\
	  \leq &\sup\limits_{\|\vlambda\|_{\infty} \leq K}\left| \frac{ \langle \bLambda, \vlambda \rangle^2\left(\|\bTheta\|^2 - dq_{\Theta} \right) / (nd)}{2q_{\Theta}^{-1} + 2\|\vlambda\|^2 / \sqrt{nd}} \right| +\sup\limits_{\|\vlambda\|_{\infty} \leq K} \left| \frac{q_{\Theta}^2 \langle \bLambda, \vlambda \rangle^2 \|\vlambda\|^2 / (n^{3/2}d^{1/2})}{2q_{\Theta}^{-1} + 2\|\vlambda\|^2 / \sqrt{nd}} \right|\nonumber \\
	\leq & \frac{nK^4 q_{\Theta}}{2d} \left| \|\bTheta\|^2 - d q_{\Theta} \right| + \frac{q_{\Theta}^3n^{3/2}K^6}{2d^{1/2}}. \label{eq:225}
\end{align}
Furthermore, notice that the following inequalities hold:
\begin{align}
	& \sup\limits_{\|\vlambda\|_{\infty} \leq K} \left| \frac{\langle \bLambda, \vlambda \rangle \langle \vlambda ,\bZ\bTheta \rangle / (nd)^{3/4}}{q_{\Theta}^{-1} + \|\vlambda\|^2 / \sqrt{nd}} \right| \leq \frac{n^{3/4}q_{\Theta}K^3}{d^{3/4}}\|\bZ \bTheta\|, \label{eq:226} \\
	& \sup\limits_{\|\vlambda\|_{\infty} \leq K}\left| \frac{\left\| \bZ^\sT \vlambda \right\|^2 / \sqrt{nd}}{2q_{\Theta}^{-1} + 2\|\vlambda\|^2 / \sqrt{nd}} - \frac{\left\| \bZ^\sT \vlambda \right\|^2 / \sqrt{nd}}{2q_{\Theta}^{-1}} + \frac{\left\| \bZ^\sT \vlambda \right\|^2 \|\vlambda\|^2 / (nd)}{2q_{\Theta}^{-2}} \right| \leq  \frac{n^2K^6 q_{\Theta}^3}{2d\sqrt{nd}} \|\bZ\bZ^\sT\|_{\rm{op}}, \label{eq:227} \\
	&  \sup\limits_{\|\vlambda\|_{\infty} \leq K} \left| \frac{d}{2} \log \left(1 + \frac{q_{\Theta}}{\sqrt{nd}} \|\vlambda\|^2 \right) - \frac{dq_{\Theta}}{2\sqrt{nd}}\|\vlambda\|^2 + \frac{q_{\Theta}^2}{4n} \|\vlambda\|^4 \right| \leq  \frac{q_{\Theta}^3 n^{3/2}K^6}{6d^{1/2}}, \label{eq:228} \\
	& \sup\limits_{\|\vlambda\|_{\infty} \leq K} \left| \frac{q_{\Theta}^2 \|\bZ^{\top} \vlambda\|^2 \|\vlambda\|^2}{2nd} - \frac{q_{\Theta}^2}{2n} \|\vlambda\|^4  \right| \leq \frac{nq_{\Theta}^2K^4}{2d}\|\bZ\bZ^{\top} - d\id_n\|_{\rm{op}}, \label{eq:C94}
\end{align}
where in \cref{eq:228}, we use the fact that for all $x \in [0, \infty)$, there exists $y \in [0, x]$ such that
\begin{align*}
	\log(1 + x) = x - \frac{x^2}{2} + \frac{x^3}{3(1 + y)^3}.
\end{align*}
Next, we combine \cref{eq:225,eq:226,eq:227,eq:228,eq:C94}, and conclude that
\begin{align}
	& \left|\Phi_n^{(0)}(h,s) - \tilde{\Phi}_n(h,s) \right| \nonumber \\
	 \leq & \frac{1}{n} \E\left[ \frac{nK^4 q_{\Theta} \left| \|\bTheta\|^2 - d q_{\Theta} \right|}{2d} + \frac{q_{\Theta}^3n^{3/2}K^6}{2d^{1/2}} + \frac{n^{3/4}q_{\Theta}K^3\|\bZ \bTheta\|}{d^{3/4}}  +  \frac{n^{3/2}K^6 q_{\Theta}^3\|\bZ\bZ^\sT\|_{\rm{op}}}{2d^{3/2}} \right. \label{eq:229}\\
	 & \left.  + \frac{q_{\Theta}^3 n^{3/2}K^6}{6d^{1/2}} + \frac{nq_{\Theta}^2 K^4\|\bZ\bZ^{\top} - d\id_n\|_{\rm{op}}}{2d} \right]. \nonumber 
\end{align}
Using \cref{lemma:concentration-of-sample-covariance}, we see that 
\begin{align}
	 \E\left[ \left\| \frac{1}{d} \bZ\bZ^\sT - \id_n  \right\|_{\rm{op}} \right] \leq & 100\sqrt{\frac{n}{d}} + \int_{100\sqrt{n/d}}^{\infty} \exp\left( -\frac{dx^2}{3200} \right) \dd x \nonumber \\
	 \leq &  100\sqrt{\frac{n}{d}} + \frac{40}{\sqrt{d}}\int_0^{\infty} \exp\left( -\frac{y^2}{2} \right)\dd y. \label{eq:230}
\end{align}
Finally, we combine \cref{eq:229}, \eqref{eq:230}, the assumption that $d \gg n$, and conclude that as $n,d \to \infty$, $|\Phi_n^{(0)}(h,s)- \tilde{\Phi}_n(h,s)| = o_n(1) $, thus completing the proof of  \cref{lemma:free-energy-2}.

\subsection{Proof of \cref{lemma:free-energy-3}}\label{sec:proof-free-energy-density3}
Recall that $\bW \overset{d}{=} \GOE(n)$. Then for all fixed orthogonal matrix $\bO \in \RR^{n \times n}$,  $\bO^\top\bW \bO \overset{d}{=} \bW$ and $\bO^\top (\bZ\bZ^\top - d\id_n) \bO \overset{d}{=} (\bZ\bZ^\sT - d\id_n)$. By orthogonal invariance, we can couple $\left( \bZ \bZ^\sT - d\id_n \right) / {\sqrt{nd}}$ with $\bW$ such that they admit the following eigen-decomposition:
\begin{align}
	\frac{1}{\sqrt{nd}}\left( \bZ \bZ^\sT - d\id_n \right) = \bOmega^\sT \bS_1 \bOmega, \qquad  \bW = \bOmega^\sT \bS_2 \bOmega. \label{eq:231}
\end{align}
In the above display, $\bOmega$ is Haar-distributed on the orthogonal matrix group $\mathcal{O}(n)$, $\bS_1$ and $\bS_2$ are diagonal matrices containing ascendingly ordered eigenvalues of matrices $\left( \bZ \bZ^\sT - d\id_n \right) / \sqrt{nd}$ and $\bW$, respectively. Furthermore, $\bS_1, \bS_2 $ are both independent of $\bOmega$. Direct computation implies the following inequality:
\begin{align*}
	\left|\tilde{\Phi}_n(h,s) - \Phi_n^Y(h,s)\right| \leq &  \frac{1}{n} \E\left[ \sup\limits_{\|\vlambda\|_{\infty} \leq K} \left| \frac{q_{\Theta}}{2}\vlambda^\sT\left( \bW - \frac{1}{\sqrt{nd}}\left(\bZ\bZ^\sT - d\id_n \right) \right)\vlambda  \right|  \right] \\\
	\leq & \frac{q_{\Theta}K^2}{2} \E\left[ \|\bS_1 - \bS_2\|_{\rm{op}} \right].
\end{align*}
Let $\sigma_i(\bS_j)$ be the $i$-th largest eigenvalue of $\bS_j$ for $j \in [2]$, then $\|\bS_1 - \bS_2\|_{\rm{op}} = \max_{i \in [n]}|\sigma_i(\bS_1) - \sigma_i(\bS_2)|$. We denote by ESD$(\bM)$ the empirical spectral distribution of matrix $\bM$. Then using random matrix theory, ESD$(\bS_1)$ and ESD$(\bS_2)$ both converge almost surely to the semicircle law (see \cite{bai1988convergence}). Furthermore, according to the results in \cite{bai1988necessary, paul2012asymptotic, karoui2003largest}, asymptotically speaking, we have $\sigma_1(\bS_1), \sigma_1(\bS_2) \overset{a.s.}{\rightarrow} 2$ and $\sigma_n(\bS_1), \sigma_n(\bS_2) \overset{a.s.}{\rightarrow} -2$. Therefore, we see that $\|\bS_1 - \bS_2\|_{\rm{op}} \overset{a.s.}{\rightarrow} 0$ as $n,d \to \infty$.

By Theorem 1.1 in \cite{bandeira2016sharp}, for all $0 < \eps \leq 1 / 2$
\begin{align*}
	\E[\|\bW\|_{\rm{op}}] \leq (1 + \eps)\left\{2\sqrt{1 + \frac{1}{n}} + \frac{6}{\sqrt{\log(1 + \eps)}} \sqrt{\frac{2\log n}{n}} \right\}.
\end{align*}
In the above equation, we first let $n \rightarrow \infty$ then let $\eps \rightarrow 0^+$, which gives $\limsup_{n \rightarrow \infty}\E[\|\bS_2\|_{\rm{op}}] \leq 2$. By Fatou's lemma, we further have $\liminf_{n \rightarrow \infty} \E[\|\bS_2\|_{\rm{op}}] \geq 2$, thus $\lim_{n \rightarrow \infty} \E[\|\bS_2\|_{\rm{op}}] = 2$. By  \cref{lemma:concentration-of-sample-covariance}, for any $\eps > 0$, there exists $M > 0$, such that for $n,d$ large enough
	\begin{align*}
		\E\left[ \left\| \bS_1 \right\|_{\rm{op}} \mathbbm{1}\left\{\| \bS_1\|_{\rm{op}} \geq M \right\} \right] < \eps.
	\end{align*}
	Dominated convergence theorem gives $\limsup_{n \rightarrow \infty} \E[\|\bS_1\|_{\rm{op}} \mathbbm{1}\{\|\bS_1\|_{\rm{op}} < M\}] \leq 2$, thus we have $\limsup_{n \rightarrow \infty} \E[\|\bS_1\|_{\rm{op}} ] \leq 2 + \eps$. On the other hand, Fatou's lemma implies $\liminf_{n \rightarrow \infty} \E[\|\bS_1\|_{\rm{op}} ] \geq 2$, thus $\lim_{n \rightarrow \infty}\E[\|\bS_1\|_{\rm{op}}] = 2$. Finally, notice that $ \|\bS_1\|_{\rm{op}} + \|\bS_2\|_{\rm{op}} - \|\bS_1 - \bS_2\|_{\rm{op}} \geq 0$. We then apply Scheffé's lemma to both $\|\bS_1 - \bS_2\|_{\rm{op}}$ and $ \|\bS_1\|_{\rm{op}} + \|\bS_2\|_{\rm{op}} - \|\bS_1 - \bS_2\|_{\rm{op}}$, which gives $\E[\|\bS_1 - \bS_2\|_{\rm{op}}] \rightarrow 0$. This concludes the proof of \cref{lemma:free-energy-3}.

\subsection{Proof of \cref{lemma:free-energy-4}}\label{sec:proof-free-energy-density4}

The first claim is a direct consequence of \cref{lemma:free-energy-1,lemma:free-energy-2,lemma:free-energy-3}. As for the second claim, it is straightforward that the free energy densities $\Phi_n(h,s)$ and $\Phi^Y_n(h,s)$ are well-defined on $[0, \infty) \times [0, \infty)$ and differentiable for all $h,s \in (0, \infty)$.   
	
By Nishimori identity (\cref{lemma:nishimori}) and Gaussian integration by parts, we see that for $h,s > 0$,
\begin{align*}
	& \frac{\partial}{\partial h}\Phi_n(h, s) = \frac{1}{4n^2} \E\left[ \langle \bLambda \bLambda^{\top}, \E[\bLambda \bLambda^{\top} \mid \bA, \bY'(h), \bx'(s)] \rangle \right], \\
	& \frac{\partial}{\partial s}\Phi_n(h, s) = \frac{1}{2n} \E\left[ \langle \bLambda, \E[\bLambda \mid \bA, \bY'(h), \bx'(s)] \rangle \right]. 
\end{align*}
Recall that $h,s$ stand for the signal-to-noise ratios in the perturbed model. Therefore, we obtain that for fixed $s \geq 0$, $\frac{\partial}{\partial h}\Phi_n(h, s)$ is increasing in $h$ and for fixed $h \geq 0$, $\frac{\partial}{\partial s}\Phi_n(h, s)$ is increasing in $s$. 

As a result, for all fixed $h,s \geq 0$, the mappings $x \mapsto \Phi_n(h, x)$, $x \mapsto \Phi_n(x, s)$ are convex on $(0, \infty)$. Since these mappings are obviously continuous, we obtain that they are convex functions on $[0, \infty)$. Similarly, we can show that for all fixed $h,s \geq 0$, $x \mapsto \Phi_n^Y(h, x)$, $x \mapsto \Phi_n^Y(x, s)$ are convex functions on $[0, \infty)$. This concludes the proof of the second claim. 
	
Finally, we prove the third claim. This proof is based on Guerra's interpolation technique. For $t \in [0,1], x,q \in \RR_+$, we define
\begin{align*}
	 H_n^G(\vlambda; t, x, q) :=& \frac{xt}{2n}\langle \bLambda, \vlambda \rangle^2 + \frac{\sqrt{xt}}{2} \vlambda^{\top} \bW' \vlambda - \frac{xt}{4n} \|\vlambda\|^4 + \sqrt{(1 - t)xq} \langle \bg', \vlambda \rangle \\
	 & + (1 - t)xq \langle  \bLambda, \vlambda\rangle - \frac{(1 - t)xq}{2}\|\vlambda\|^2, \\
	 \Psi^G_n(t,x,q) := & \frac{1}{n} \E\left[ \log \int \exp\left(H_n^G(\vlambda; t, x, q) \right)\tensorl  \right].
\end{align*}
%
By \cite[Theorem 13]{lelarge2019fundamental}, we see that $\lim_{n \rightarrow \infty} \Psi^G_n(1,x,q) = \sup_{y \geq 0} \cF(x, y)$. Using \cref{lemma:nishimori} and Gaussian integration by parts, we see that
\begin{align*}
	\frac{\partial}{\partial t}\Psi^G_n(t,x,q) =& \frac{1}{n}\E\left[ \frac{x}{4n}\langle \bLambda, \vlambda \rangle^2 - \frac{xq}{2}\langle \bLambda, \vlambda\rangle \, \Big | \, \frac{\sqrt{xt}}{n} \bLambda \bLambda^{\sT} + \bW', \sqrt{(1 - t)xq} \bLambda + \bg' \right] \\
	=& \frac{x}{4}\E\left[ \Big( \frac{1}{n}\langle \bLambda, \vlambda \rangle - q \Big)^2 \, \Big | \, \frac{\sqrt{xt}}{n} \bLambda \bLambda^{\sT} + \bW', \sqrt{(1 - t)xq} \bLambda + \bg' \right] - \frac{xq^2}{4} \\
	 \geq & -  \frac{xq^2}{4}.
\end{align*}
Direct computation reveals that
\begin{align*}
	\Psi^G_n(0, x, q) = \cF(x, q) + \frac{xq^2}{4}.
\end{align*}
Therefore, for all $x,q \geq 0$,
\begin{align}\label{eq:D67}
	\Psi^G_n(1,x,q) = \Psi^G_n(0,x,q) + \int_0^1 \frac{\partial}{\partial t}\Psi_n^G(t, x, q) \dd t \geq \cF(x, q). 
\end{align}
According to \cite[Proposition 17]{lelarge2019fundamental}, for all but countably many $x > 0$, $\cF(x, \cdot)$ has a unique maximizer $q^{\ast}(x)$. For these $x$, we plug $q = q^{\ast}(x)$ into \cref{eq:D67}, which implies for all but countably many $x > 0$ and all $t \in [0,1]$, $\lim_{n \rightarrow \infty} \Psi_n^G(t, x, q^{\ast}(x)) = \cF(x, q^{\ast}(x)) + {xq^{\ast}(x)^2}(1 - t) / 4$. Notice that $\Psi_n^G(t, x, q^{\ast}(x)) = \Phi_n^Y(0, s)$ if $xt = q_{\Theta}^2$ and $(1 - t)xq^{\ast}(x) = s$. This concludes the proof of the third claim of the lemma.

\subsection{Proof of \cref{lemma:bounded-approx-free-energy}}\label{proof:lemma:bounded-approx-free-energy}
For $t \in [0,1]$, we define the interpolated Hamiltonian as
	\begin{align*}
		H_{n,t}^{\eps}(\vlambda, \vtheta; h) := & \frac{t}{\sqrt{nd}}\langle \bLambda, \vlambda \rangle \langle \bTheta, \vtheta\rangle + \frac{\sqrt{t}}{\sqrt[4]{nd}} \vlambda^{\top} \bZ \vtheta - \frac{t}{2\sqrt{nd}}\|\vlambda\|^2\|\vtheta\|^2 + \\
		& \frac{1 - t}{\sqrt{nd}} \langle \barbLambda, \bar\vlambda \rangle\langle \bTheta, \vtheta \rangle + \frac{\sqrt{1 - t}}{\sqrt[4]{nd}} \bar\vlambda^{\top} \bZ' \vtheta - \frac{1 - t}{2\sqrt{nd}}\|\bar\vlambda\|^2\|\vtheta\|^2 + \\
		& \frac{ht}{2n}\langle \bLambda, \vlambda \rangle^2 + \frac{\sqrt{ht}}{2}\vlambda^\sT\bW \vlambda - \frac{ht}{4n}\|\vlambda\|^4 + \\
		& \frac{h(1 - t)}{2n}\langle \bar\bLambda, \bar\vlambda \rangle^2 + \frac{\sqrt{h(1 - t)}}{2}\bar\vlambda^\sT\bW' \bar\vlambda - \frac{h(1 - t)}{4n}\|\bar\vlambda\|^4.
	\end{align*}
where $\bZ' = (Z_{ij}')_{i \in [n], j \in [d]}$ is an independent copy of $\bZ$ and is independent of everything else. We emphasize that $\bZ, \bZ', \bW, \bW', \bLambda, \bTheta$ are mutually independent.  Notice that $H_{n,t}^{\eps}(\vlambda, \vtheta; h)$ is the Hamiltonian that corresponds to observations $(\bA_1, \bA_2, \bY_1, \bY_2)$ 
\begin{align*}
	& \bA_1 = \frac{\sqrt{t}}{\sqrt[4]{nd}} \bLambda \bTheta^\sT + \bZ, \qquad \bA_2 = \frac{\sqrt{1 - t}}{\sqrt[4]{nd}} \bar{\bLambda} \bTheta^\sT + \bZ', \\
	& \bY_1 = \frac{\sqrt{ht}}{n} \bLambda \bLambda^{\top} + \bW, \qquad \bY_2 = \frac{\sqrt{h(1 - t)}}{n} \barbLambda \barbLambda^{\top} + \bW'.
\end{align*}
We define the corresponding free energy density
\begin{align*}
	\Phi_{n,t}^{\eps}(h) := \frac{1}{n} \E\left[ \log \left( \int \exp\left( H_{n,t}^{\eps}(\vlambda, \vtheta; h)\right)\tensorl\tensort \right) \right].
\end{align*}
At the endpoints, we have $\Phi_{n,0}^{\eps}(h) = \bar{\Phi}_n^{\eps}(h)$ and $\Phi_{n,1}^{\eps}(h) = {\Phi}_n(h, 0)$. For simplicity, we denote by $\langle \cdot \rangle_{h,\eps, t}$ the expectation with respect to the posterior distribution $\P(\bLambda = \cdot, \bTheta = \cdot \mid \bA_1, \bA_2, \bY_1, \bY_2)$. Using Gaussian integration by parts and Nishimori identity (\cref{lemma:nishimori}), we have
\begin{align*}
	\frac{\partial}{\partial t}\Phi_{n,t}^{\eps}(h) =&  \frac{1}{2n\sqrt{nd}} \sum\limits_{i \in [n], j \in [d]} \E\left[ \langle (\bLambda_i \vlambda_i - \bar\bLambda_i \bar{\vlambda}_i) \bTheta_j \vtheta_j \rangle_{h,\eps, t} \right] + \frac{h}{4n^2}\E[\langle (\bLambda^{\top}\vlambda)^2 \rangle_{h,\eps, t}] - \frac{h}{4n^2}\E[\langle (\bar\bLambda^{\top}\bar\vlambda)^2 \rangle_{h,\eps, t}]  \\
	=& \frac{1}{2\sqrt{nd}} \E\left[ \left\langle(\bLambda_1 \vlambda_1 - \bar\bLambda_1 \bar{\vlambda}_1) \langle \bTheta, \vtheta \rangle \right\rangle_{h,\eps, t}\right]+ \frac{h}{4n^2}\sum_{i \in [n], j \in [n]}\E\left[ \langle \bLambda_i\bLambda_j \vlambda_i \vlambda_j -  \bar\bLambda_i\bar\bLambda_j \bar\vlambda_i \bar\vlambda_j \rangle_{h,\eps, t} \right]. 
\end{align*}
Next, we provide upper bound for the above partial derivative.  Invoking Holder's inequality, we see that  
\begin{align*}
	& \Big|\frac{\partial}{\partial t}\Phi_{n,t}^{\eps}(h) \Big|\\
	\leq & \frac{1}{2\sqrt{nd}} \E\left[\left\langle(\bLambda_1 \vlambda_1 - \bar\bLambda_1 \bar{\vlambda}_1)^2 \right\rangle_{h,\eps,t}^{1/2}\left\langle  \langle \bTheta, \vtheta \rangle^2 \right\rangle_{h,\eps, t}^{1/2}\right] + \frac{h}{2n^2} \sum_{i \in [n], j \in [n]} \E[(\bLambda_i\bLambda_j - \bar\bLambda_i \bar\bLambda_j)^2]^{1/2}\E[\bLambda_i^2\bLambda_j^2]^{1/2} \\
	\leq & \frac{1}{2\sqrt{nd}} \E\left[\left\langle(\bLambda_1 \vlambda_1 - \bar\bLambda_1 \bar{\vlambda}_1)^2 \right\rangle_{h,\eps, t} \right]^{1/2}\E\left[\left\langle\langle \bTheta, \vtheta \rangle^2 \right\rangle_{h,\eps, t} \right]^{1/2} + {h} \E[\bLambda_1^4]^{3/4}\E[(\bLambda_1 - \bar\bLambda_1)^4]^{1/4}.
\end{align*}
We denote by $\langle \cdot \rangle_{h,\eps, t, \ast}$ the expectation with respect to the posterior distribution 
$$\P( \bTheta = \cdot \mid \bA_1, \bA_2, \bY_1, \bY_2, \bLambda, \bar{\bLambda}).$$
Direct computation gives the following inequality:
\begin{align}
	\E\left[\left\langle\langle \bTheta, \vtheta \rangle^2 \right\rangle_{ h,\eps, t, \ast} \right] =& d\E\left[\left\langle \bTheta_1^2\vtheta_1^2  \right\rangle_{h,\eps, t, \ast}\right] + d(d - 1)\E\left[\left\langle \bTheta_1\bTheta_2\vtheta_1\vtheta_2  \right\rangle_{h,\eps, t, \ast}\right] \nonumber \\
	\leq & d\E[\bTheta_1^4] + d(d - 1) \E\left[\left\langle \bTheta_1\bTheta_2\vtheta_1\vtheta_2  \right\rangle_{h,\eps, t, \ast}\right]. \label{eq:232}
\end{align}
Recall that $r_n = d^{1/4}n^{-1/4}$. We define the mapping 
\begin{align*}
	F_{\Theta}(\delta) := r_n^2 \E[ \E[ \bThetas \mid r_n^{-1}\delta \bThetas + \bG ]^2 ],
\end{align*}
where $\bThetas \sim \mu_{\Theta}$, $\bG \sim \normal(0,1)$ and $\bThetas \perp \bG$. Notice that
\begin{align}
	\E\left[\left\langle \bTheta_1\bTheta_2\vtheta_1\vtheta_2  \right\rangle_{h,\eps, t, \ast}\right] \leq \frac{n}{d}\E\Big[ F_{\Theta}\Big( \sqrt{({t\|\bLambda\|_2^2 + (1 - t)\|\bar{\bLambda}\|_2^2})/{n}} \Big)^2 \Big].  \label{eq:233}
\end{align}
Direct computation gives
\begin{align}
	& \frac{\dd}{\dd \delta}\E\left[\bThetas \mid \delta \bThetas + \bG \right] \nonumber \\
	 = & (2\delta \bThetas + \bG)\Var[\bThetas \mid \delta \bThetas + \bG]- \delta \E[\bThetas^3 \mid \delta \bThetas + \bG] + \delta \E[\bThetas^2 \mid \delta \bThetas + \bG] \E[\bThetas \mid \delta \bThetas + \bG], \label{eq:188}
\end{align}
Next, we apply triangle inequality to upper bound the right hand side of \cref{eq:188}, which gives
\begin{align*}
	& \left|\frac{\dd}{\dd\delta} \E\left[ \bThetas \mid r_n^{-1} \delta \bThetas + \bG \right] \right| \\
	\leq & r_n^{-1} \times \left\{ (2r_n^{-1} \delta |\bThetas| + |\bG|) \Var\left[ \bThetas \mid r_n^{-1} \delta \bThetas + \bG  \right]    \right. \\
	& \left. + r_n^{-1} \delta \E[|\bThetas|^3 \mid r_n^{-1} \delta \bThetas + \bG] + r_n^{-1} \delta \E[\bThetas^2 \mid r_n^{-1} \delta \bThetas + \bG]\E[|\bThetas|\mid r_n^{-1} \delta \bThetas + \bG] \right\}. 
\end{align*}
Leveraging the above formulas and \Holder's inequality, we obtain that for $n,d$ large enough
\begin{align}\label{eq:FThetadelta}
	F_{\Theta}(\delta) =& \sqrt{\frac{d}{n}} \E\left[ \E[\bThetas \mid r_n^{-1} \delta \bThetas + \bG]^2 \right] \nonumber \\
	\leq & \sqrt{\frac{d}{n}} \E\left[ \left( \int_0^{\delta} \left|\frac{\dd}{\dd x} \E\left[ \bThetas \mid r_n^{-1} x \bThetas + \bG \right] \right| \dd x \right)^2\right] \nonumber \\
	\leq & \sqrt{\frac{d}{n}} \E\left[ \delta \int_0^{\delta} \left|\frac{\dd}{\dd x} \E\left[ \bThetas \mid r_n^{-1} x \bThetas + \bG \right] \right|^2 \dd x\right] \nonumber \\
	\leq & 4 \delta \E\left[ \int_0^{\delta} 4r_n^{-2} x^2 \bThetas^2 \Var\left[ \bThetas \mid r_n^{-1} x \bThetas + \bG  \right]^2 + \bG^2 \Var\left[ \bThetas \mid r_n^{-1} x \bThetas + \bG  \right]^2 \right. \nonumber \\
	& \left. + r_n^{-2} x^2\E[|\bThetas|^3 \mid r_n^{-1} x \bThetas + \bG]^2 + r_n^{-2} x^2 \E[\bThetas^2 \mid r_n^{-1} x \bThetas + \bG]^2\E[|\bThetas|\mid r_n^{-1} x \bThetas + \bG]^2 \dd x  \right] \nonumber \\
	 \leq &  C_{\mu_{\Theta}}(\delta^4 + 1).
\end{align}
In the above display, $C_{\mu_{\Theta}} > 0$ is a constant that depends only on $\mu_{\Theta}$. 
	
We define the set $S = \left\{ \|\bLambda\|_2^2 \leq n\E[\bLambdas^2] + n\|\bLambdas^2\|_{\Psi_1} \right\}$, where $\|\cdot\|_{\Psi_1}$ is the sub-exponential norm of $\bLambdas^2$. Then by Bernstein's inequality \cite[Theorem 2.8.1]{vershynin2018high}, we can conclude that there exists a constant $C_{\mu_{\Lambda}} > 0$ depending only on $\mu_{\Lambda}$, such that for all $s \geq 1$,
	\begin{align*}
		\P\left( \|\bLambda\|_2^2 \geq n\E[\bLambdas^2] + sn\|\bLambdas^2\|_{\Psi_1} \right) \leq 2\exp\left( -C_{\mu_{\Lambda}}ns \right). 
	\end{align*}
Therefore, for $n,d$ large enough we have
\begin{align}
	& \frac{n}{d}\E\left[ F_{\Theta}\left( \sqrt{({t\|\bLambda\|_2^2 + (1 - t)\|\bar{\bLambda}\|_2^2})/{n}} \right)^2 \right] \nonumber \\
	=& \frac{n}{d}\E\left[ F_{\Theta}\left( \sqrt{({t\|\bLambda\|_2^2 + (1 - t)\|\bar{\bLambda}\|_2^2})/{n}} \right)^2 \mathbbm{1}_S \right] + \frac{n}{d}\E\left[ F_{\Theta}\left( \sqrt{({t\|\bLambda\|_2^2 + (1 - t)\|\bar{\bLambda}\|_2^2})/{n}} \right)^2 \mathbbm{1}_{S^c} \right]\nonumber \\
	\overset{(i)}{\leq} & \frac{4C_{\mu_{\Theta}}^2n}{d} + \frac{2C_{\mu_{\Theta}}^2n}{d}\left( \E[\bLambdas^2] + \|\bLambdas^2\|_{\Psi_1} \right)^4 + \frac{2nC_{\mu_{\Theta}}^2}{d} \E\left[ \left(\|\bLambda\|_2^2 / n \right)^4\mathbbm{1}_{S^c}  \right]\nonumber \\
	\leq & \frac{4C_{\mu_{\Theta}}^2n}{d} + \frac{2C_{\mu_{\Theta}}^2n}{d}\left( \E[\bLambdas^2] + \|\bLambdas^2\|_{\Psi_1} \right)^4 + \\
	& \frac{2nC_{\mu_{\Theta}}^2}{d} \int_1^{\infty} 4\P\left(\|\bLambda\|_2^2 / n \geq \E[\bLambdas^2] + s\|\bLambdas^2\|_{\Psi_1} \right)\left(\E[\bLambdas^2] + s\|\bLambdas^2\|_{\Psi_1} \right)^3\|\bLambdas^2\|_{\Psi_1} \dd s\nonumber \\
	\leq & \frac{4C_{\mu_{\Theta}}^2n}{d} + \frac{2C_{\mu_{\Theta}}^2n}{d}\left( \E[\bLambdas^2] + \|\bLambdas^2\|_{\Psi_1} \right)^4 + \frac{2nC_{\mu_{\Theta}}^2}{d}\int_1^{\infty}8\exp\left( -C_{\mu_{\Lambda}}ns \right)\left(\E[\bLambdas^2] + s\|\bLambdas^2\|_{\Psi_1} \right)^3\|\bLambdas^2\|_{\Psi_1} \dd s \nonumber \\
		 \leq &  \frac{C_1n}{d}, \label{eq:234}
\end{align}
where $C_1 > 0$ is a constant depending only on $(\mu_{\Theta}, \mu_{\Lambda})$, and in \emph{(i)} we use \cref{eq:FThetadelta}. Furthermore, 
\begin{align}
	\E\left[\left\langle(\bLambda_1 \vlambda_1 - \bar\bLambda_1 \bar{\vlambda}_1)^2 \right\rangle_{h,\eps, t} \right] \leq & 2\E\left[ \bLambda_1^2 \langle (\vlambda_1 - \bar{\vlambda}_1)^2 \rangle_{h,\eps, t} \right] + 2\E\left[(\bLambda_1 - \bar{\bLambda}_1)^2\langle  \bar{\vlambda}_1^2 \rangle_{h,\eps, t}   \right] \nonumber \\
	\leq & 2\E\left[ \bLambda_1^4  \right]^{1/2} \E\left[ \langle (\vlambda_1 - \bar{\vlambda}_1)^2  \rangle_{h,\eps, t}^2  \right]^{1/2} + 2\E\left[(\bLambda_1 - \bar\bLambda_1)^4 \right]^{1/2} \E\left[ \langle \bar{\vlambda}_1^2 \rangle_{h,\eps,t}^2 \right]^{1/2} \nonumber \\
	\leq & 2\E\left[ \bLambda_1^4  \right]^{1/2}\E\left[ (\bLambda_1 - \bar\bLambda_1)^4  \right]^{1/2} + 2\E\left[(\bLambda_1 - \bar\bLambda_1)^4 \right]^{1/2}\E\left[ \bar{\bLambda}_1^4 \right] \nonumber \\
	\leq &  C_2 \sqrt{\eps},\label{eq:235}
\end{align}
where $C_2 > 0$ is a constant depending only on $\mu_{\Lambda}$. Finally, we combine \cref{eq:232,eq:233,eq:234,eq:235}, and conclude that
\begin{align*}
	\left|\frac{\partial}{\partial t}\Phi_{n,t}^{\eps}(h)\right| \leq C_0 \sqrt[4]{\eps}
\end{align*}
for all $t, h \in [0,1]$, where $C_0 > 0$ is a constant depending only on $(\mu_{\Theta}, \mu_{\Lambda})$. This concludes the proof of the lemma.

\section{Achieving the Bayesian MMSE}\label{sec:tec-lemmas-upper-bound}
In this section we prove the technical lemmas required to prove \cref{thm:upper-bound}.
\subsection{Proof of \cref{lemma:log-truncate2}}\label{sec:proof-of-lemma:log-truncate2}
We define the set
\begin{align*}
	\Omega := \left\{ |\bLambda_i| \leq 2K_0 \sqrt{\log n}: i \in [n] \right\}.
\end{align*}
By \cref{eq:C4-31} we see that $\P(\Omega^c) \leq 2n^{-3}$. Furthermore, on $\Omega$ we have $\bA = {\bar\bA}$. For matrix $\bX \in \RR^{n \times d}$, we define the mapping  $\bar \bM: \RR^{n \times d} \rightarrow \RR^{n \times n}$ such that $\bar \bM(\bX) = \E[\barbLambda \barbLambda^{\top} \mid \bar\bA = \bX]$. Then we have
\begin{align}
	\frac{1}{n} \E\left[\left\| \bLambda \bLambda^{\top} - \E[\bLambda \bLambda^{\top} \mid \bA] \right\|_F^2 \right]^{1/2} \overset{(i)}{\leq} & \frac{1}{n} \E\left[\left\| \bLambda \bLambda^{\top} - \bar{\bM}(\bA) \right\|_F^2 \right]^{1/2} \nonumber \\
	\overset{(ii)}{\leq} & \frac{1}{n} \E\left[\left\| \barbLambda \barbLambda^{\top} - \bar{\bM}(\bA) \right\|_F^2 \right]^{1/2} + \frac{1}{n} \E\left[\left\| \barbLambda \barbLambda^{\top} - \bLambda \bLambda^{\top} \right\|_F^2 \right]^{1/2}, \label{eq:C24}
\end{align}
where \emph{(i)} is by the fact that the posterior expectation achieves Bayesian MMSE, and \emph{(ii)} is by triangle inequality. Applying triangle inequality and \Holder's inequality, we have
\begin{align}
	\frac{1}{n} \E\left[\left\| \barbLambda \barbLambda^{\top} - \bar{\bM}(\bA) \right\|_F^2 \right]^{1/2} \leq & \frac{1}{n} \E\left[\left\| \barbLambda \barbLambda^{\top} - \bar{\bM}(\bA) \right\|_F^2 \mathbbm{1}_{\Omega}\right]^{1/2} + \frac{1}{n} \E\left[\left\| \barbLambda \barbLambda^{\top} - \bar{\bM}(\bA) \right\|_F^2 \mathbbm{1}_{\Omega^c}\right]^{1/2}\nonumber \\
	\leq & \frac{1}{n} \E\left[\left\| \barbLambda \barbLambda^{\top} - \bar{\bM}({\bar\bA}) \right\|_F^2 \right]^{1/2} + \frac{1}{n} \E\left[\left\| \barbLambda \barbLambda^{\top} - \bar{\bM}(\bA) \right\|_F^4 \right]^{1/4}\P(\Omega^c)^{1/4}.\label{eq:C25}
\end{align}	
Direct computation reveals that $ \E[\| \barbLambda \barbLambda^{\top} - \bLambda \bLambda^{\top} \|_F^2 ]^{1/2} / n= o_n(1)$ as $n,d \rightarrow \infty$, and $ \E[\| \barbLambda \barbLambda^{\top} - \bar{M}(\bA) \|_F^4]^{1/4} / n \leq 8K_0^2 \log n$. As a result, we conclude that $ \E[\| \barbLambda \barbLambda^{\top} - \bar{M}(\bA) \|_F^4]^{1/4} \P(\Omega^c)^{1/4} / n = o_n(1)$. Combining these analysis with \cref{eq:C24,eq:C25} concludes the proof of the lemma.
	
%
%
%
%

\subsection{Proof of \cref{lemma:toMMSE}}\label{sec:proof-of-lemma:toMMSE}
%
%
%
%
%
Let $\bW, \bW' \iidsim \GOE(n)$ that are independent of $\bLambda$. For $t \in [0,1]$, $s \geq 0$, we define 
\begin{align*}
	& \bY_{a,t}^{(s)} := \frac{q_{\Theta}\sqrt{(1 - t)s} \bLambda \bLambda^{\top}}{n} + \bW, \\
	& \bY_{b,t}^{(s)} := \frac{\sqrt{ts}}{n}(q_{\Theta}^{1/2} \barbLambda + r_n^{-1} \bar\bg)(q_{\Theta}^{1/2} \barbLambda + r_n^{-1} \bar\bg)^{\top} + \bW'.
\end{align*}
For $\bx, \by \in \RR^n$, we define the corresponding truncated vectors $\bar\bx,\bar\by \in \RR^n$ such that $\bar x_i = x_i \mathbbm{1}\{|x_i| \leq 2K_0\sqrt{ \log n}\}$ and $\bar y_i = y_i \mathbbm{1}\{|y_i| \leq C_3 \sqrt{\log n}\}$ for all $i \in [n]$. The Hamiltonian that corresponds to $(\bY_{a,t}^{(s)}, \bY_{b,t}^{(s)})$ can be expressed as 
\begin{align*}
	& H_{n,t}^{(s)}(\bx, \by) \\
	 := & \frac{ts}{2n}((q_{\Theta}^{1/2}\barbLambda + r_n^{-1}  \bar\bg)^{\top} (q_{\Theta}^{1/2}\bar\bx + r_n^{-1}\bar\by))^2 + \frac{\sqrt{ts}}{2}(q_{\Theta}^{1/2}\bar\bx + r_n^{-1}\bar\by)^{\top} \bW' (q_{\Theta}^{1/2}\bar\bx + r_n^{-1}\bar\by)  - \frac{ts}{4n}\| q_{\Theta}^{1/2}\bar\bx + r_n^{-1}\bar\by\|^4 \\
	& + \frac{(1 - t)sq_{\Theta}^2}{2n}(\bLambda^{\top} \bx)^2 + \frac{\sqrt{(1 - t)s}q_{\Theta}}{2} \bx^{\top} \bW \bx - \frac{(1 - t)sq_{\Theta}^2}{4n}\|\bx\|^4. 
\end{align*}
The corresponding free energy density can be written as 
\begin{align*}
		G_n(t, s) := \frac{1}{n} \E\left[ \log \left( \int \exp\left( H_{n,t}^{(s)}(\bx, \by) \right)  P_{\bLambda}^{\otimes n}(\dd \bx)  P_{\normal(0,1)}^{\otimes n}(\dd \by)\right)  \right],
\end{align*}
where $P_{\normal(0,1) }^{\otimes n}$ is the distribution of $\normal(\mathbf{0}, \id_n)$. Invoking \cref{lemma:nishimori} and Gaussian integration by parts, we obtain that
\begin{align*}
	\frac{\partial}{ \partial t} G_n(t, s) = &  \frac{s}{4n^2} \E\left[ \big \|\E\big [(q_{\Theta}^{1/2} \barbLambda + r_n^{-1} \bar\bg)(q_{\Theta}^{1/2} \barbLambda + r_n^{-1} \bar\bg)^{\top} \mid \bY_{a,t}^{(s)}, \bY_{b,t}^{(s)}\big ] \big \|_F^2 \right] \\
	& - \frac{s q_{\Theta}^2}{4n^2} \E\left[ \big\| \E\big[\bLambda \bLambda^{\top} \mid \bY_{a,t}^{(s)}, \bY_{b,t}^{(s)}\big] \big\|_F^2 \right].
\end{align*}
Leveraging \Holder's inequality, we see that
\begin{align}
	& \Big|\frac{\partial}{ \partial t} G_n(t,s)\Big| \nonumber \\
	 \leq & \left| \E\left[\frac{s}{4}\E\left[(q_{\Theta}^{1/2} \bar\bLambda_1 + r_n^{-1} \bar g_1)(q_{\Theta}^{1/2} \bar\bLambda_2 + r_n^{-1} \bar g_2) \mid \bY_{a,t}^{(s)}, \bY_{b,t}^{(s)}\right]^2- \frac{s q_{\Theta}^2}{4} \E\left[ \bLambda_1 \bLambda_2 \mid \bY_{a,t}^{(s)}, \bY_{b,t}^{(s)}\right]^2 \right] \right| + o_n(1) \nonumber \\
	\leq & \frac{s}{4} \E\left[ \left( q_{\Theta} \bar\bLambda_1 \bar\bLambda_2 + r_n^{-1} q_{\Theta}^{1/2} \bar\bLambda_1 \bar{g}_2 + r_n^{-1} q_{\Theta}^{1/2} \bar\bLambda_2 \bar{g}_1 + r_n^{-2} \bar{g}_1 \bar{g}_2 - q_{\Theta}\bLambda_1 \bLambda_2\right)^2 \right]^{1/2} \times \nonumber \\
	& \E\left[ \left( q_{\Theta} \bar\bLambda_1 \bar\bLambda_2 + r_n^{-1} q_{\Theta}^{1/2} \bar\bLambda_1 \bar{g}_2 + r_n^{-1} q_{\Theta}^{1/2} \bar\bLambda_2 \bar{g}_1 + r_n^{-2} \bar{g}_1 \bar{g}_2 + q_{\Theta}\bLambda_1 \bLambda_2\right)^2 \right]^{1/2} + o_n(1). \label{eq:36} 
\end{align}	
The upper bound given in the last line of \cref{eq:36} is independent of $t$ and converges to 0 as $n,d \rightarrow \infty$. 
Therefore, we conclude that as $n,d \to \infty$
\begin{align*}
	\sup_{t \in (0,1), s \in [0,2]}\, \Big| \frac{\partial}{\partial t} G_n(t, s) \Big| \rightarrow 0,
\end{align*}
which further implies that $|G_n(1,s) - G_n(0,s)| = o_n(1)$ for all $s \in [0,2]$. Recall that $\cF(\cdot, \cdot)$ is defined in \cref{eq:FsQ}. Using \cite[Theorem 13]{lelarge2019fundamental}, we have $G_n(0, s) = \sup_{q \geq 0} \cF(q_{\Theta}^2 s, q) + o_n(1)$. Observe that $s \mapsto G_n(1, s)$ is convex differentiable on $(0, \infty)$, and converges point-wisely to $\sup_{q \geq 0} \cF(q_{\Theta}^2 s, q) $ as $n, d \to \infty$, the later is differentiable at $s = 1$ for all but countably many values of $q_{\Theta} > 0$ according to \cite[Proposition 17]{lelarge2019fundamental}. Invoking \cref{lemma:convex-derivative}, \cref{lemma:nishimori} and Gaussian integration by parts, we conclude that for all but countably many $q_{\Theta} > 0$
	\begin{align*}
		\lim_{n \rightarrow \infty} \frac{1}{n^2} \E\left[ \big\| q_{\Theta}^{-1}\bM_n(\bY_2) - q_{\Theta}^{-1}\big( q_{\Theta}^{1/2} \barbLambda + r_n^{-1} \bar\bg \big)\big( q_{\Theta}^{1/2} \barbLambda + r_n^{-1} \bar\bg \big)^{\top} \big\|_F^2 \right] = \lim_{n \to \infty} \MMSEsy_n(\mu_{\Lambda}; q_{\Theta}).
	\end{align*}
	Notice that $\lim_{n \rightarrow \infty} \E[ \| ( q_{\Theta}^{1/2} \barbLambda + r_n^{-1} \bar\bg )( q_{\Theta}^{1/2} \barbLambda + r_n^{-1} \bar\bg )^{\top} / n - q_{\Theta} \bar\bLambda \bar\bLambda^{\top} / n \|_F^2 ] = 0$, then the proof of the lemma follows immediately from triangle inequality. 
		
\subsection{Proof of \cref{lemma:log-truncate}}\label{sec:proof-of-lemma:log-truncate}
We define the set
\begin{align*}
	\Omega := \left\{ |\bTheta_j| \leq 2K_2 \sqrt{\log d}: j \in [d]\right\}.
\end{align*}
By \cref{eq:37} we have $\P(\Omega^c) \leq 2d^{-3}$. Furthermore, on the set $\Omega$ we have $\bA = \bar\bA$. For $\bX \in \RR^{n \times d}$, we define the mapping $\bM(\bX) := \E[\bLambda \bLambda^{\top} \mid \bar\bA = \bX]$. Leveraging triangle inequality and \Holder's inequality, we obtain that
\begin{align*}
	\frac{1}{n} \E\left[\left\| \bLambda \bLambda^{\top} - {M}(\bA) \right\|_F^2 \right]^{1/2}  \leq & \frac{1}{n} \E\left[\left\| \bLambda \bLambda^{\top} - {M}(\bA) \right\|_F^2 \mathbbm{1}_{\Omega}\right]^{1/2} + \frac{1}{n} \E\left[\left\| \bLambda \bLambda^{\top} - {M}(\bA) \right\|_F^2 \mathbbm{1}_{\Omega^c}\right]^{1/2} \\
	\leq & \frac{1}{n} \E\left[\left\| \bLambda \bLambda^{\top} - {M}(\bar\bA) \right\|_F^2 \right]^{1/2} + \frac{1}{n} \E\left[\left\| \bLambda \bLambda^{\top} - {M}(\bA) \right\|_F^4 \right]^{1/4}\P(\Omega^c)^{1/4}.
\end{align*}
Since the posterior expectation minimizes the expected $l^2$ risk, we then have 
\begin{align*}
	\frac{1}{n^2}\E\left[\left\| \bLambda \bLambda^{\top} - {M}(\bA) \right\|_F^2 \right] \geq \frac{1}{n^2} \E\left[\left\| \bLambda \bLambda^{\top} - \E[\bLambda \bLambda^{\top} \mid \bA] \right\|_F^2 \right].
\end{align*}
By the bounded-support assumption, we see that $\E[\| \bLambda \bLambda^{\top} - {M}(\bA) \|_F^4]^{1/4}/ n \leq 2K_1^2$. Thus, as $n,d \rightarrow \infty$,
	\begin{align*}
		\frac{1}{n} \E\left[\left\| \bLambda \bLambda^{\top} - {M}(\bA) \right\|_F^4 \right]^{1/4}\P(\Omega^c)^{1/4} \rightarrow 0,
	\end{align*}
	which completes the proof of the lemma.

\subsection{Proof of \cref{lemma:C9}}\label{sec:proof-of-lemma:C9}
For $t \in [0,1]$, we define the interpolated Hamiltonian as 
\begin{align*}
	H_{n,t}^{[s]}(\vlambda, \vtheta; h) := &  \sum\limits_{i\in [n], j \in [d]} \left\{ \frac{ts}{\sqrt{nd}}\bLambda_i\vlambda_i \bTheta_j \vtheta_j + \frac{\sqrt{ts}}{\sqrt[4]{nd}} Z_{ij} \vlambda_i\vtheta_j -  \frac{ts}{2\sqrt{nd}} \vlambda_i^2\vtheta_j^2 \right\} +   \\
	& \sum\limits_{i\in [n], j \in [d]}\left\{ \frac{s(1 - t)}{\sqrt{nd}}\bLambda_i\vlambda_i\bar\bTheta_j\bar\vtheta_j + \frac{\sqrt{s(1 - t)}}{\sqrt[4]{nd}}Z_{ij}'\vlambda_i\bar\vtheta_j - \frac{s(1 - t)}{2\sqrt{nd}} \vlambda_i^2 \bar\vtheta_j^2 \right\}+ H_n(\vlambda; \bY'(h) ),
\end{align*}
where $\bZ' = (Z_{ij}')_{i \in [n], j \in [d]}$ is an independent copy of $\bZ$ and is independent of everything else.  Note that $H_{n,t}^{[s]}(\vlambda, \vtheta; h)$ is the Hamiltonian corresponding to the observations $(\bA_1^{(s,t)}, \bA_2^{(s,t)},  \bY'(h))$, where $\bA_1^{(s,t)} = {\sqrt{ts}}\bLambda\bTheta^{\top} / {\sqrt[4]{nd}}  + \bZ$, $\bA_2^{(s,t)} = {\sqrt{(1 - t)s}} \bLambda \barbTheta^{\top} / {\sqrt[4]{nd}}+ \bZ'$ and $\bY'(h) = {\sqrt{h}} \bLambda \bLambda^{\top} / n + \bW'$. Here, we recall that $\bW' \sim \GOE(n)$, and $\bW', \bZ, \bZ'$ are mutually independent. We define the free energy density corresponding to the Hamiltonian $H_{n,t}^{[s]}(\vlambda, \vtheta; h)$ as
\begin{align*}
	\Phi_{n,t}^{[s]}(h) := \frac{1}{n} \E\left[ \log \left( \int \exp\left( H_{n,t}^{[s]}(\vlambda, \vtheta; h)\right)\tensorl\tensort \right) \right].
\end{align*}
At the endpoints, we have $\Phi_{n,0}^{[s]}(h) = \bar{\Phi}_n(s,0,0,h)$ and $\Phi_{n,1}^{[s]}(h) = {\Phi}_n(s,0,0,h)$. We denote by $\langle \cdot \rangle_{t,h}^{[s]}$ the expectation with respect to the posterior distribution $\P(\cdot \mid \bA_1^{(s,t)}, \bA_2^{(s,t)}, \bY'(h))$. Then we have
	\begin{align}
		\Big| \frac{\partial}{\partial t} \Phi_{n,t}^{[s]}(h) \Big| \overset{(i)}{=} & \Big| \frac{s}{2n \sqrt{nd}} \sum_{i \in [n], j \in [d]} \E\left[ \langle \bLambda_i \vlambda_i (\bTheta_j \vtheta_j - \bar\bTheta_j \bar\vtheta_j) \rangle_{t,h}^{[s]}  \right] \Big| \nonumber \\
		 \overset{(ii)}{\leq} & {\frac{s}{2n\sqrt{nd}}\sum\limits_{i \in [n], j \in [d]}\E\left[\langle \bLambda_i^2\vlambda_i^2 \rangle_{t,h}^{[s]} \right]^{1/2} \E\left[ \langle (\bTheta_j \vtheta_j - \bar\bTheta_j \bar\vtheta_j)^2 \rangle_{t,h}^{[s]} \right]^{1/2}}, \label{eq:C29}
	\end{align}
	where \emph{(i)} is by Gaussian integration by parts and Nishimori identity (Lemma \ref{lemma:nishimori}), and \emph{(ii)} is by \Holder's inequality. For all $j \in [d]$, using power mean inequality and \Holder's inequality, we have 
\begin{align}
	\E\left[ \langle (\bTheta_j \vtheta_j - \bar\bTheta_j \bar\vtheta_j)^2 \rangle_{t,h}^{[s]} \right] \leq & 2\E[\bTheta_j^2\langle (\vtheta_j - \bar\vtheta_j)^2 \rangle_{t,h}^{[s]}] + 2\E[(\bTheta_j - \bar\bTheta_j)^2\langle \bar\vtheta_j^2 \rangle_{t,h}^{[s]}] \nonumber \\
	\leq & 2\E[\bTheta_j^4]^{1/2}\E[\langle(\vtheta_j - \bar\vtheta_j)^2 \rangle_{t,h}^{[s] 2}]^{1/2} + 2\E[(\bTheta_j - \bar\bTheta_j)^4]^{1/2}\E[\langle \bar\vtheta_j^2 \rangle_{t,h}^{[s] 2}]^{1/2} \nonumber \\
	\leq & 2\E[\bTheta_j^4]^{1/2}\E[(\bTheta_j - \bar\bTheta_j)^4]^{1/2} + 2\E[(\bTheta_j - \bar\bTheta_j)^4]^{1/2}\E[\bar\bTheta_j^4]^{1/2}. \label{eq:C30}
\end{align}
Notice that
\begin{align}
	\E[(\bTheta_j - \bar\bTheta_j)^4] \leq & \int_{2K_2\sqrt{\log d}}^{\infty} 4x^3\P(|\bThetas| \geq x) \dd x \leq 4 \int_{4K_2^2{\log d}}^{\infty} y \exp\Big(-\frac{y}{K_2^2}\Big) \dd y \nonumber \\
	= & -4(yK_2^2 + K_2^4) \exp\Big(-\frac{y}{K_2^2} \Big) \Big|^{\infty}_{4K_2^2{\log d}} = \frac{4K_2^4 + 16K_2^4 \log d}{d^4}. \label{eq:C31}
\end{align}
Combining \cref{eq:C29,eq:C30,eq:C31}, we obtain that $\sup_{t \in [0,1], h \geq 0, S_0 \geq s \geq 0}\Big| \frac{\partial}{\partial t}\Phi_{n,t}^{(s)}(h) \Big| \rightarrow 0$ as $n, d \rightarrow \infty$, thus completing the proof of the lemma.

\subsection{Proof of \cref{lemma:G1}}\label{sec:proof-of-lemma:G1}
Using \cref{lemma:C9}, we have $\lim_{n,d \rightarrow \infty}\left|\bar\Phi_n(s,0,0, h) - \Phi_n(s,0,0, h) \right| = 0$. Similar to the proof of \cref{lemma:free-energy-converge}, we can conclude that $\lim_{n,d \rightarrow \infty}|\Phi_n(s,0,0, h) - \sup_{q \geq 0} \cF(q_{\Theta}^2 s^2 + h, q)| = 0$ as $n,d \rightarrow \infty$. 
	Therefore, in order to prove the lemma, it suffices to show
\begin{align*}
	\lim\limits_{n,d \rightarrow \infty}\sup\limits_{a,a' \in [0,10]} \left|\bar\Phi_n(s,a,a', h) - \bar\Phi_n(s,0,0, h) \right| = 0.
\end{align*} 
%
Using Gaussian integration by parts and Nishimori identity (Lemma \ref{lemma:nishimori}), we obtain that for all $a,a' \in [0,10]$,
\begin{align}\label{eq:27}
\begin{split}
	& \frac{\partial}{\partial \ep_n} \bar\Phi_n(s,a,a',h) = \frac{a^2}{2d} \E\left[ \barbTheta^{\top} \E[\barbTheta \mid \bar\bA(s), \bx'(a'), \bar\bx(a), \bY'(h)] \right] \leq 50\E_{\bThetas \sim \mu_{\Theta}}[\bar\bTheta_0^2], \\
	& \frac{\partial}{\partial \ep_n'} \bar\Phi_n(s,a,a',h) = \frac{{a'}^2}{2n}\E\left[ \bLambda^{\top} \E[\bLambda \mid \bar\bA(s), \bx'(a'), \bar\bx(a), \bY'(h)] \right] \leq 50\E_{\bLambdas \sim \mu_{\Lambda}}[\bLambdas^2].
\end{split}
\end{align}
Notice that if $\ep_n = \ep_n' = 0$, then $\bar\Phi_n(s,a,a',h) = \bar\Phi_n(s,0,0,h)$. Therefore, by \cref{eq:27}, we conclude that as $n, d \rightarrow \infty$, 
\begin{align*}
	\sup_{a,a' \in [0,10]} \left|\bar\Phi_n(s,a,a',h) - \bar\Phi_n(s,0,0,h) \right| \leq 50\big(\E_{\bLambdas \sim \mu_{\Lambda}}[\bLambdas^2] + \E_{\bThetas \sim \mu_{\Theta}}[\bar\bTheta_0^2]\big)(\ep_n + \ep_n') \rightarrow 0,
\end{align*}
thus completing the proof of the lemma.

\subsection{Proof of \cref{lemma:C11}}\label{sec:proof-of-lemma:C11}

Since $|\bar\bTheta_0| \leq 2K_2\sqrt{\log d}$, we then have
\begin{align}
	& \left| \E[\langle U(\barbtheta^{(1)}) (\barbtheta^{(1)})^{\top}\barbtheta^{(2)} / d \rangle_{1,a,a',h}] - \E[\langle (\barbtheta^{(1)})^{\top}\barbtheta^{(2)} / d \rangle_{1,a,a',h}] \E[\langle U(\barbtheta^{(1)}) \rangle_{1,a,a',h}]  \right| \nonumber \\
	\leq & 4K_2^2 \log d \E[\langle | U(\bar\vtheta) - \E[\langle U(\bar\vtheta) \rangle_{1,a,a',h}]|\rangle_{1,a,a',h}].\label{eq:G188}
\end{align} 
Using Gaussian integration by parts and Nishimori identity (\cref{lemma:nishimori}), we have
\begin{align}
	& \E[\langle (\barbtheta^{(1)})^{\top}\barbtheta^{(2)} / d \rangle_{1,a,a',h}] \E[\langle U(\barbtheta^{(1)}) \rangle_{1,a,a',h}] = a \E[\langle (\barbtheta^{(1)})^{\top}\barbtheta^{(2)} / d \rangle_{1,a,a',h} ]^2,  \label{eq:G189}\\
	& \E[\langle U(\barbtheta^{(1)}) (\barbtheta^{(1)})^{\top}\barbtheta^{(2)}/ d \rangle_{1,a,a',h}] = a \E[\langle ((\barbtheta^{(1)})^{\top}\barbtheta^{(2)} / d)^2\rangle_{1,a,a',h}]. \label{eq:G190}  
\end{align}
Next, we combine \cref{eq:G188,eq:G189,eq:G190}, and conclude that for all $a \in [10^{-1},10]$, 
\begin{align*}
	& \E[\langle ((\barbtheta^{(1)})^{\top}\barbtheta^{(2)} / d - \E[\langle (\barbtheta^{(1)})^{\top}\barbtheta^{(2)} / d\rangle_{1,a,a',h}])^2 \rangle_{1,a,a',h}] \\
	 \leq &  40K_2^2\log d\E[\langle | U(\barbtheta) - \E[\langle U(\barbtheta) \rangle_{1,a,a',h}]|\rangle_{1,a,a',h}]. 
	\end{align*}

\subsection{Proof of \cref{lemma:G3}}\label{sec:proof-of-lemma:G3}
One can verify that $\bar\phi_n(1,a,a',h)$ is twice differentiable for $a,a' \in (0,10)$. Using Gaussian integration by parts and Nishimori identity, we can compute its partial derivatives:
\begin{align}
& \frac{\partial }{\partial a} \bar\phi_n(1,a,a',h) = \ep_n \langle U(\barbtheta) \rangle_{1,a,a',h}, \label{eq:pU} \\
& \frac{\partial^2 }{\partial a^2 } \bar\phi_n(1,a,a',h) = n\ep_n^2\langle (U(\barbtheta) - \langle U(\barbtheta)\rangle_{1,a,a',h})^2 \rangle_{1,a,a',h} + \ep_n \langle 2\bar\bTheta^{\top} \barbtheta / d - \barbtheta^{\top} \barbtheta / d \rangle_{1,a,a',h}.  \label{eq:G191}
\end{align}
Notice that $|\bar\bTheta_0| \leq 2K_2\sqrt{\log d}$, then $\left|\E[\frac{\partial }{\partial a} \bar\phi_n(1,a,a',h)] \right| =  \ep_n a \E[\langle \barbTheta^{\top} \barbtheta / d \rangle_{1,a,a',h}] \leq 40 \ep_n K_2^2 \log d$ for all $a,a' \in (0,10)$. Using these results, we further obtain that
\begin{align}
	& \langle (U(\barbtheta) - \langle U(\barbtheta)\rangle_{1,a,a',h})^2 \rangle_{1,a,a',h} \leq \frac{1}{n\ep_n^2}\left( \frac{\partial^2 }{\partial a^2 } \bar\phi_n(1,a,a',h) + 12K_2^2 \ep_n \log d  \right), \nonumber\\
	& \int_1^2 \int_1^2 \E [\langle (U(\barbtheta) - \langle U(\barbtheta)\rangle_{1,a,a',h})^2 \rangle_{1,a,a',h}] \dd a \dd a' \nonumber \\
	 \leq & \int_1^2 \E\left[  \frac{1}{n\ep_n^2}\left( \frac{\partial }{\partial a} \bar\phi_n(1, a,a',h)\Big\vert_{a = 2} - \frac{\partial }{\partial a} \bar\phi_n(1, a,a',h)\Big\vert_{a = 1} + 12K_2^2 \ep_n \log d \right) \right] \dd a' \nonumber \\
	  \leq & CK_2^2n^{-1}\ep_n^{-1} \log d, \label{eq:G194}
\end{align}
where $C > 0$ is a numerical constant. Leveraging \cref{eq:G191}, we conclude that the following two functions are convex for all fixed $a' \in (0,10)$ and $h \geq 0$:
\begin{align*}
	& a \mapsto \bar\phi_n(1, a,a',h) + 6\ep_nK_2^2 a^2 \log d , \\
	& a \mapsto \E[\bar\phi_n(1, a,a',h)] + 6\ep_nK_2^2 a^2 \log d . 
\end{align*}
By \cref{lemma:A10}, for all $a \in [1,2]$, $b \in (0, 1/2), a' \in [1/2,3]$ and $h \geq 0$, we have
\begin{align}
	& \E \left[ \left| \frac{\partial }{\partial a} 
	\bar\phi_n(1, a,a',h) - E[ \frac{\partial }{\partial a} \bar\phi_n(1, a,a',h)] \right| \right] \nonumber \\
	\leq & \E\left[ \frac{\partial }{\partial a} \bar\phi_n(1, a + b,a',h) - \frac{\partial }{\partial a} \bar\phi_n(1, a - b,a',h) \right] + 24\ep_nK_2^2b \log d + \frac{3v_n(h)}{b}. \label{eq:G192}
\end{align}
Again we use the fact that $\left|\E[\frac{\partial }{\partial a} \bar\phi_n(1, a,a',h)] \right|  \leq 40 \ep_n K_2^2 \log d$ for all  $a,a' \in (0,10)$, and conclude that
\begin{align}
	& \int_1^2 \E\left[ \frac{\partial }{\partial a} \bar\phi_n(1, a + b,a',h) - \frac{\partial }{\partial a}  \bar\phi_n(1, a - b,a',h) \right] \dd a \nonumber \\
	 = & \E\left[  \bar\phi_n(1, b + 2, a',h) -  \bar\phi_n(1, b + 1, a',h) -  \bar\phi_n(1, 2 - b, a',h) +  \bar\phi_n(1, 1 - b, a',h) \right] \nonumber \\
		 \leq & C'K_2^2b\ep_n \log d, \label{eq:G193}
\end{align}
where $C' > 0$ is a numerical constant. Then we combine \cref{eq:G192,eq:G193} and obtain that
\begin{align}\label{eq:Cpp}
	\int_1^2\int_1^2 \E \left[ \left| \frac{\partial }{\partial a} \bar\phi_n(1, a,a',h) - E[ \frac{\partial }{\partial a} \bar\phi_n(1, a,a',h)] \right| \right] \dd a \dd a' \leq C''\left(b\ep_n K_2^2 \log d + \frac{v_n(h)}{b}\right),
\end{align}
where $C'' > 0$ is another numerical constant. Later in \cref{lemma:concentration-of-free-energy} we will see that under the current conditions, for $n,d$ large enough we have $v_n(h) < \frac{1}{4}K_2^2\ep_n \log d$. Since $b$ is arbitrary in $(0, 1/2)$, we can then take $b = \sqrt{v_n(h) / (\eps_nK_2^2 \log d)}$ in \cref{eq:Cpp} and apply this to \cref{eq:pU}, which gives 
\begin{align}\label{eq:G195}
	\int_1^2 \int_1^2 \E\left[ \left| \langle U(\barbtheta)  \rangle_{1,a,a',h} - \E [\langle U(\barbtheta)  \rangle_{1,a,a',h}] \right| \right] \dd a \dd a' \leq 2C''K_2 \sqrt{v_n(h) \ep_n^{-1} \log d}. 
\end{align} 
Finally, we combine \Holder's inequality, \cref{eq:G194,eq:G195} and concludes the proof of the lemma. 
		
\subsection{Proof of \cref{lemma:concentration-of-free-energy}}\label{sec:proof-of-lemma:concentration-of-free-energy}
Conditioning on $(\bLambda, \barbTheta)$, we consider the mapping
\begin{align*}
	f: (\bZ, \bg, \bg', \sqrt{n}\bW') \mapsto \bar\phi_n( 1, a, a', h). 
\end{align*}
For $n,d$ large enough, the following inequality holds for all $a,a' \in [0,10]$.
\begin{align*}
	\|\nabla f\|^2 \leq CK_1^2K_2^2d^{1/2}n^{-3/2}\log d,
\end{align*}
where $C > 0$ is a numerical constant. By Gaussian Poincar\'e inequality \cite{van2014probability},  we conclude that for $n,d$ large enough
\begin{align}
	\E_{\bZ, \bg, \bg', \bW'} \left[ \left( \bar\phi_n(1,  a, a', h) - \E_{\bZ, \bg, \bg', \bW'}[\bar\phi_n( 1, a, a',h)] \right)^2 \right] \leq CK_1^2K_2^2d^{1/2}n^{-3/2}\log d . \label{eq:conc187}
\end{align}
In the above display, the expectations are taken over $(\bZ, \bg, \bg', \bW')$. 

Next, we show that $\E_{\bZ, \bg, \bg', \bW'}[\bar\phi_n(1, a, a',h)]$ (as a function of $(\bLambda, \barbTheta)$), concentrates around its expectation. Notice that for $n,d$ large enough, for all $i \in [n], j \in [d]$ we have 
\begin{align*}
	& \left|\frac{\partial}{\partial \bLambda_i} \E_{\bZ, \bg, \bg', \bW'}[\bar\phi_n(1, a, a',h)]\right| \leq C'K_1K_2^2d^{1/2}n^{-3/2} \log d, \\
	& \left|\frac{\partial}{\partial \bar\bTheta_j} \E_{\bZ, \bg, \bg', \bW'}[\bar\phi_n(1, a, a',h)]\right| \leq C''K_1^2 K_2 d^{-1/2}n^{-1/2}(\log d)^{1/2},
\end{align*} 
where $C', C'' > 0$ are numerical constants. By Efron-Stein inequality \cite{van2014probability}, we see that there exists a numerical constant $C''' > 0$, such that for $n,d$ large enough
\begin{align}
	\E\left[ \left(\E_{\bZ, \bg, \bg', \bW'}[\bar\phi_n(1, a, a', h)] - \E[\bar\phi_n(1, a, a', h)] \right)^2 \right] \leq C'''K_1^4 K_2^4 dn^{-2}(\log d)^2. \label{eq:conc188}
\end{align}
Finally, we combine \cref{eq:conc187,eq:conc188} and conclude that for $n,d$ large enough, there exists a numerical constant $C_1 > 0$, such that for all $a,a' \in [0,10]$
\begin{align*}
	\E\left[ \left(\bar\phi_n(1, a, a', h) -  \E[\bar\phi_n(1, a, a', h)] \right)^2 \right] \leq C_1^2K_1^4K_2^4 dn^{-2} (\log d)^2 , 
\end{align*}
which concludes the proof of the lemma using Cauchy–Schwarz inequality. 

\subsection{Proof of the first claim of \cref{lemma:C14}}\label{sec:proof-of-lemma:C14}
Invoking \cref{lemma:G1}, as $n, d \rightarrow \infty$ we have
\begin{align}\label{eq:48}
	\sup_{a \in [0, 2a_{\ast}], a' \in [0, 2a_{\ast}']}\left| \bar\Phi_n(1,a,a',h) - \sup_{q \geq 0} \cF(q_{\Theta}^2 + h, q) \right| = o_n(1). 
\end{align}
%
By Jensen's inequality, for $a \sim \Unif[a_{\ast} / 2, a_{\ast}]$ and $a' \sim \Unif[a'_{\ast} / 2, a'_{\ast}]$ we have 
	\begin{align}\label{eq:49}
		& \frac{1}{n^2}\E\left[\big\| \E[\bLambda \bLambda^{\top} \mid \bar\bA(1), \bY'(h)] - \E[\bLambda \bLambda^{\top} \mid \bar\bA(1), \bx'(a'), \bar\bx(a), \bY'(h)] \big\|_F^2 \right] \nonumber  \\ 
		\leq & \frac{1}{n^2}\E\left[\big\| \E[\bLambda \bLambda^{\top} \mid \bar\bA(1), \bY'(h)] - \E[\bLambda \bLambda^{\top} \mid \bar\bA(1), \bx'(a'_{\ast}), \bar\bx(a_{\ast}), \bY'(h)] \big\|_F^2 \right].
	\end{align}
	Notice that the mapping $h \mapsto \bar\Phi_n(1,a_{\ast}, a_{\ast}', h)$ is convex and differentiable, and $h \mapsto \sup_{q \geq 0} \cF(q_{\Theta}^2 + h, q)$ is differentiable for all $q_{\Theta}^2 + h \in D$. Therefore, using Gaussian integration by parts, \cref{lemma:nishimori,lemma:convex-derivative}, we conclude that for $h + q_{\Theta}^2 \in D$, as $n,d \rightarrow \infty$ the right hand side of \cref{eq:49} converges to 0 as $n,d \rightarrow \infty$. Furthermore, 
\begin{align*}
	 \frac{\partial}{\partial a'} \E\left[ \|\bLambda_{1,a,a',h}\|^2 \right] 	=  \ep_n'a'\E\left[ \|\E[\bLambda \bLambda^{\top} \mid \bar\bA(1), \bar\bx(a), \bx'(a'), \bY'(h)] - \bLambda_{1,a,a',h} \bLambda_{1,a,a',h}^{\top}\|_F^2  \right] \leq  4K_1^4n^2\ep_n'a'.
\end{align*}
Therefore, for $a \sim \Unif[a_{\ast} / 2, a_{\ast}]$ and $a' \sim \Unif[a_{\ast}'/2, a_{\ast}']$, using the above equation we have
\begin{align}\label{eq:50}
	& \frac{1}{n^2} \E\left[\|\E[\bLambda\bLambda^{\top} \mid \bar\bA(1), \bx'(a'), \bar\bx(a), \bY'(h)] - \bLambda_{1,a,a',h} \bLambda_{1,a,a',h}^{\top}\|_F^2\right] \nonumber \\
	\leq & \frac{8}{n^2a_{\ast}(a_{\ast}')^2\ep_n'} \int_{a_{\ast}/2}^{a_{\ast}}\int_{a_{\ast}'/2}^{a_{\ast}'}\frac{\partial}{\partial a'} \E\left[ \|\bLambda_{1,a,a',h}\|^2 \right] \dd a' \dd a \nonumber \\
	\leq & \frac{4K_1^2}{n \ep_n'(a_{\ast}')^2}, 
\end{align}
which converges to zero as $n,d \rightarrow \infty$ under the current assumptions.  \cref{eq:49,eq:50} imply 
\begin{align*}
	\frac{1}{n^2}\| \E[\bLambda\bLambda^{\top} \mid \bar\bA(1), \bY'(h)] \|_F^2 = \frac{1}{n^2} \|\bLambda_{1,a,a',h}\|^4+ o_P(1). 
\end{align*}
By \cref{cor:overlap-concentration} we have
\begin{align*}
	\frac{1}{n} \|\bLambda_{1,a,a',h}\|^2 = \frac{1}{n}\E\left[\|\bLambda_{1,a,a',h}\|^2 \mid a, a' \right] + o_P(1).
\end{align*}
By Jensen's inequality, for all $a \in [a_{\ast} / 2, a_{\ast}]$ and $a' \in [a_{\ast}' / 2, a_{\ast}']$ we have 
 $$\E\left[\|\bLambda_{1,a,a',h}\|^2 \mid a, a' \right] \leq \E\left[\|\bLambda_{1,a_{\ast},a_{\ast}',h}\|^2  \right].$$ 
Next, we combine the above equations and obtain that as $n,d \to \infty$
\begin{align*}
	\frac{1}{n^2}\| \E[\bLambda\bLambda^{\top} \mid \bar\bA(1), \bY'(h)] \|_F^2 \leq \frac{1}{n^2}\E\left[\|\bLambda_{1,a_{\ast},a_{\ast}',h}\|^2  \right]^2 + o_P(1). 
\end{align*}
Similarly, if we consider $a \sim \Unif [a_{\ast}, 2a_{\ast}]$ and $a' \sim  \Unif [a_{\ast}', 2a_{\ast}']$, then we have
\begin{align*}
	\frac{1}{n^2}\| \E[\bLambda\bLambda^{\top} \mid \bar\bA(1), \bY'(h)] \|_F^2 \geq \frac{1}{n^2}\E\left[\|\bLambda_{1,a_{\ast},a_{\ast}',h}\|^2  \right]^2 + o_P(1). 
\end{align*}
Since support$(\bLambdas) \subseteq [-K_1, K_1]$, by Lebesgue dominated convergence theorem we have
\begin{align}\label{eq:51}
	\lim_{n,d \rightarrow \infty}\frac{1}{n^2}\E\left[\|\bLambda_{1,a_{\ast},a_{\ast}',h}\|^2 \right]^2 = \lim_{n,d \rightarrow \infty}\frac{1}{n^2}\E[\| \E[\bLambda\bLambda^{\top} \mid \bar\bA(1), \bY'(h)] \|_F^2].
\end{align}
By \cref{lemma:G1}, as $n,d \rightarrow \infty$ we have $\bar\Phi_n(1, 0, 0, h) \rightarrow \sup_{q \geq 0}\cF(q_{\Theta}^2 + h, q)$. Furthermore, the mapping $h \mapsto \bar\Phi_n(1, 0, 0, h)$ is convex and differentiable. Therefore, if $q_{\Theta}^2 + h \in D$, then by \cref{lemma:convex-derivative} and Gaussian integration by parts we have
\begin{align}\label{eq:52}
	\lim_{n,d \rightarrow \infty}\frac{1}{4n^2}\E[\| \E[\bLambda\bLambda^{\top} \mid \bar\bA(1), \bY'(h)] \|_F^2] = \frac{\partial}{\partial h} \sup_{q \geq 0}\cF(q_{\Theta}^2 + h, q).
\end{align}
The proof of the first claim follows immediately from \cref{eq:51,eq:52}. 


\subsection{Proof of the second claim of \cref{lemma:C14}}\label{sec:proof-of-lemma:C15}
We let $a \sim \Unif[a_{\ast} / 2, a_{\ast}]$ and $a' \sim \Unif[a_{\ast}' / 2, a_{\ast}']$, then by \cref{cor:overlap-concentration}, we have
\begin{align*}
	& \frac{1}{n}\|\bLambda_{1,a,a',h}\|^2 = \frac{1}{n}\E\left[\|\bLambda_{1,a,a',h}\|^2 \mid a,a' \right] + \delta_{n,1}, \\
	& \frac{1}{\sqrt{nd}}\|\barbTheta_{1,a,a',h}\|^2 = \frac{1}{\sqrt{nd}}\E\left[\|\barbTheta_{1,a,a',h}\|^2 \mid a,a' \right] + \delta_{n,2},
\end{align*}
where $\E[\delta_{n,1}^2]$ and $\E[\delta_{n,2}^2]$ are random variables that converge to 0 as $n,d \rightarrow \infty$. By \cref{eq:C47} we have
\begin{align*}
	 \frac{1}{2n\sqrt{nd}}\|\bM_{1,0,0,h}\|_F^2 =  \frac{1}{2n\sqrt{nd}}\|\bLambda_{1,a,a',h}\|^2 \|\barbTheta_{1,a,a',h}\|^2 + \delta_{n,0}.
\end{align*}
In the above equation, $\delta_{n,0}$ is a random variable satisfying $\E[|\delta_{n,0}|] \rightarrow 0$ as $n,d \rightarrow \infty$. Therefore, for all $a \in [a_{\ast} / 2, a_{\ast}]$ and $a' \in [a_{\ast}' / 2, a_{\ast}']$

%

%
\begin{align}
	& \frac{1}{2n\sqrt{nd}}\|\bM_{1,0,0,h}\|_F^2 \nonumber \\
	\leq & \frac{1}{2n\sqrt{nd}}\E\left[\|\bLambda_{1,a,a',h}\|^2 \right]\E\left[\|\barbTheta_{1,a,a',h}\|^2 \right] + \frac{1}{2}|\delta_{n,1}||\delta_{n,2}|  \\
	&   + \frac{|\delta_{n,1}|}{2\sqrt{nd}}\E\left[\|\barbTheta_{1,a,a',h}\|^2\right] + \frac{|\delta_{n,2}|}{2n}\E\left[\|\bLambda_{1,a,a',h}\|^2 \right]+ \delta_{n,0}, \label{eq:C52}
\end{align}
Notice that 
\begin{align}
	 \limsup_{n,d \rightarrow \infty}\frac{1}{\sqrt{nd}}\E\left[\|\barbTheta_{1,a_{\ast}, a'_{\ast}, h}\|^2 \right]< \infty, \qquad  \limsup_{n,d \rightarrow \infty}\frac{1}{n}\E\left[\|\bLambda_{1,a_{\ast}, a'_{\ast}, h}\|^2 \right] < \infty. \label{eq:C54}
\end{align}
We plug \cref{eq:C54} into \cref{eq:C52} then take the expectation, which implies that as $n,d \rightarrow \infty$ we have
\begin{align}
	& \liminf_{n,d \rightarrow \infty}\frac{1}{2n\sqrt{nd}}\E\left[ \|\bLambda_{1,a_{\ast}, a'_{\ast}, h}\|^2\right]\E\left[ \|\barbTheta_{1,a_{\ast}, a'_{\ast}, h}\|^2 \right] \nonumber \\
	\geq & \lim_{n,d \rightarrow \infty} \frac{1}{2n\sqrt{nd}}\E\left[\|\bM_{1,0,0,h}\|^2\right] = D_{\Theta}(h). \label{eq:C55}
\end{align}
Similarly, if we let $a \sim \Unif[a_{\ast}, 2a_{\ast}]$ and $a' \sim \Unif[a_{\ast}', 2a_{\ast}']$, then we can conclude that
\begin{align}
	& \limsup_{n,d \rightarrow \infty}\frac{1}{2n\sqrt{nd}}\E\left[ \|\bLambda_{1,a_{\ast}, a'_{\ast}, h}\|^2\right] \E\left[ \|\barbTheta_{1,a_{\ast}, a'_{\ast}, h}\|^2 \right] \nonumber \\
	\leq & \lim_{n,d \rightarrow \infty} \frac{1}{2n\sqrt{nd}}\E\left[\|\bM_{1,0,0,h}\|^2\right] = D_{\Theta}(h). \label{eq:C56}
\end{align}
Notice that $D_{\Theta}(h) =  2 q_{\Theta}^2 \frac{\partial}{\partial h} \sup_{q \geq 0} \cF(q_{\Theta}^2 + h, q)$, then the proof of the second claim follows from \cref{eq:C55}, \cref{eq:C56} and the first claim.

\section{Proofs for the Gaussian mixture clustering example}

\subsection{Proof of \cref{prop:GMM} claim (a)}\label{sec:proof-of-prop:GMM}

We define the pairwise overlap achieved by estimator $\hat\bLambda$ as 
\begin{align*}
 \Pairoverlap_n := \frac{2}{n^2}\sum\limits_{i<j} \mathbbm{1}\left\{ \mathbbm{1}\{\bLambda_i = \bLambda_j\} = \mathbbm{1}\{\hat{\bLambda}_i = {\hat{\bLambda}_j}\} \right\}.
\end{align*} 
We notice that
\begin{align}\label{eq:45}
	\Pairoverlap_n =& \frac{2}{n^2} \sum_{i < j} \left( \mathbbm{1}\{\bLambda_i = \hat\bLambda_i\}\mathbbm{1}\{\bLambda_j = \hat\bLambda_j\} + (1 - \mathbbm{1}\{\bLambda_i = \hat\bLambda_i\})(1 - \mathbbm{1}\{\bLambda_j = \hat\bLambda_j\} ) \right) \nonumber\\
	=& 2\Overlap_n^2 + 1 - 2\Overlap_n + o_n(1).
\end{align}
According to \cite[Section 2.3]{lelarge2019fundamental}, under the symmetric model \eqref{model:symmetric},  if $q_{\Theta} \leq 1$, then we have $\lim_{n \rightarrow \infty} \MMSEsy_n(\mu_{\Lambda}; q_{\Theta}) = 1$. This is also the mean square error achieved by the constant estimator $\mathbf{0}_{n \times n}$. For an estimate of the labels $\hat\bLambda \in \{-1, +1\}^n$ and $a \in (0,1)$, we define the rescaled vector $\hat{\bLambda}_a := \sqrt{a} \hat{\bLambda} \in \{-\sqrt{a},+\sqrt{a}\}^n$. Then by \cref{thm:lower-bound}, we have
\begin{align*}
	\frac{1}{n^2}\E\left[\left\| \bLambda \bLambda^\intercal - \hat{\bLambda}_a \hat{\bLambda}_a^\intercal \right\|_F^2\right] =& \E\left[(1 - a)^2(\Pairoverlap_n + n^{-1}) + (1 + a)^2(1 - n^{-1} -  \Pairoverlap_n)\right] \\
	 \geq & \MMSEas_n(\mu_{\Lambda}, \mu_{\Theta}). 
\end{align*}
\cref{thm:lower-bound} implies that $\lim_{n,d \rightarrow \infty} \MMSEas_n(\mu_{\Lambda}, \mu_{\Theta}) = 1$, thus
\begin{align*}
	\limsup\limits_{n,d \rightarrow \infty} \E\left[ \Pairoverlap_n\right] \leq \frac{a^2 + 2a}{4a},
\end{align*}
which holds for every $a \in(0,1)$. Let $a \rightarrow 0^+$, we  then conclude that $\limsup_{n,d \rightarrow \infty}\E[\Pairoverlap_n] \leq 1 / 2$. Next, we plug this result into \cref{eq:45} then apply dominated convergence theorem, which gives $\Overlap_n \toP 1 / 2$. In summary, partial recovery of component identity is impossible in the current setting.

\subsection{Proof of \cref{thm:GMM1} part (a)}\label{sec:proof-of-thm:GMM1}

Let $\hat{\bLambda} \in \RR^{n \times k} $ be any estimator of the cluster assignments constructed based on data matrix $\bA$. For $a > 0$, we define $\bhLambda_a := \sqrt{a} \bhLambda$.  Under the current conditions, as $n,d \to \infty$ we have
 \begin{align*}
 & \MMSEas_n(\mu_{\Lambda}, \mu_{\Theta}) \\
  = & k^{-2}(k - 1) + o_n(1)  \\
 \leq & \frac{1}{n^2}\E\left[ \big\|\bLambda\bLambda^{\top} - k^{-1}\mathbf{1}_{n \times n} - \bhLambda_a \bhLambda_a^{\top} \big\|_F^2  \right] \\
 =& k^{-2}(k - 1) + \frac{a^2}{n^2}\sum_{i,j \in [n]} \E\left[ \mathbbm{1}\{\hat\bLambda_i = \hat\bLambda_j\} \right]  - \frac{2a}{n^2}\sum_{i,j \in [n]} \E\left[ \mathbbm{1}\{\hat\bLambda_i = \hat\bLambda_j\} (\mathbbm{1}\{\bLambda_i = \bLambda_j\} - k^{-1}) \right] + o_n(1) \\
 \leq & k^{-2}(k - 1) + a^2  - \frac{2a}{n^2}\sum_{i,j \in [n]} \E\left[ \mathbbm{1}\{\hat\bLambda_i = \hat\bLambda_j\} (\mathbbm{1}\{\bLambda_i = \bLambda_j\} - k^{-1}) \right] + o_n(1). 
\end{align*}
Using the above equation we can conclude that 
\begin{align*}
	\limsup_{n,d \rightarrow \infty} \frac{1}{n^2}\sum_{i,j \in [n]} \E\left[ \mathbbm{1}\{\hat\bLambda_i = \hat\bLambda_j\} (\mathbbm{1}\{\bLambda_i = \bLambda_j\} - k^{-1}) \right]  \leq\frac{a}{2}.
\end{align*}
Since $a > 0$ is arbitrary, we then have
\begin{align}\label{eq:supr}
	\limsup_{n,d \rightarrow \infty} \frac{1}{n^2}\sum_{i,j \in [n]} \E\left[ \mathbbm{1}\{\hat\bLambda_i = \hat\bLambda_j\} (\mathbbm{1}\{\bLambda_i = \bLambda_j\} - k^{-1}) \right] \leq 0.
\end{align}
For $s,r \in [k]$, we define
\begin{align}\label{eq:Csk}
 C_{sr} := \frac{1}{n}\sum\limits_{i = 1}^n \mathbbm{1}\{\hat\bLambda_i = \be_s, \bLambda_i = \be_r\}.
\end{align}
We immediately see that $C_{sr} \geq 0$ and $\sum_{s \in [k]}\sum_{r \in [k]} C_{sr} = 1$. Furthermore, notice that 
\begin{align}\label{eq:14}
 	\frac{1}{n^2}\sum_{i,j \in [n]} \mathbbm{1}\{\hat\bLambda_i = \hat\bLambda_j, \bLambda_i = \bLambda_j\} = & \frac{1}{n^2}\sum_{i,j \in [n]} \sum_{s,r \in [k]} \mathbbm{1}\{\hat\bLambda_i = \hat\bLambda_j = \be_s, \bLambda_i = \bLambda_j = \be_r\} \nonumber \\
 = & \sum_{s,r \in [k]}\Big(\frac{1}{n}\sum_{i \in [n]}  \mathbbm{1}\{\hat\bLambda_i  = \be_s, \bLambda_i =\be_r\}\Big)^2 \\
 	= &  \sum_{s, r \in [k]} C_{sr}^2, \nonumber 
\end{align}
\begin{align}\label{eq:15}
 \frac{1}{kn^2}\sum\limits_{i,j \in [n]} \mathbbm{1}\{\hat\bLambda_i = \hat\bLambda_j\} = &\frac{1}{kn^2}\sum\limits_{i,j \in [n]} \sum_{s \in [k]} \mathbbm{1}\{\hat\bLambda_i = \be_s\} \mathbbm{1}\{\hat\bLambda_j = \be_s\} \nonumber \\
 =& \frac{1}{k}\sum_{s \in [k]} \Big( \frac{1}{n}\sum_{i \in [n]} \sum_{r \in [k]} \mathbbm{1}\{\hat\bLambda_i = \be_s, \bLambda_i = \be_r\} \Big)^2 \\
 = & \frac{1}{r}\sum_{s \in [k]} (C_{s1} + \cdots + C_{sk})^2. \nonumber
\end{align}
Next, we subtract \cref{eq:14} by \cref{eq:15} and apply \cref{eq:supr}, which gives 
\begin{align*}
	\lim_{n \rightarrow \infty}\E\Big[ \sum_{s \in [k]}\sum_{1 \leq r_1 < r_2 \leq k}(C_{sr_1} - C_{s r_2})^2 \Big] = 0.
\end{align*}
Note that for all $s \in [k]$, there exists $r_s \in [k]$, such that $[k] = \{r_s: s \in [k]\}$ and $\Overlap_n = \sum_{s \in [k]}C_{sr_s}$. Combining the above results, we conclude that $\Overlap_n = k^{-1} + o_P(1)$, thus completing the proof of part (a).

\subsection{Proof of \cref{thm:GMM1} part (b)}\label{sec:proof-of-thm:GMM2}

Suppose the statement is not true, then for any $\hat\bLambda$ and any subsequence of $\NN_+$, there further exists a subsequence $\{n_i\}_{i \in \NN_+} \subseteq \NN_+$ of the previous subsequence, such that $n_i < n_{i + 1}$ and $\lim_{i \rightarrow \infty} \E[\Overlap_{n_i}] = k^{-1}$. Therefore, $\Overlap_{n_i} \toP k^{-1}$. In the following parts of the proof, we will restrict to this subsequence $\{n_i\}_{i \in \NN_+}$.
 
We assume $\hat\bLambda \in \RR^{n \times k}$ such that $\hat\bLambda \overset{d}{=} \mu(\bLambda = \cdot \mid \bA)$. 
Recall that for $s,r \in [k]$, $C_{sr}$ is defined in \cref{eq:Csk}.  Furthermore, notice that $\hat\bLambda \overset{d}{=} \bLambda$, then by the law of large numbers, for all $s,r \in [k]$ we have
 	\begin{align}\label{eq:16}
 		C_{s1} + C_{s2} + \cdots + C_{sk} \toP \frac{1}{k}, \qquad C_{1r} + C_{2r} + \cdots + C_{kr} \toP \frac{1}{k}.
 	\end{align}
 	For $\delta > 0$, if $C_{11} > k^{-2} + \delta$, then by \cref{eq:16} $\sum_{2 \leq s,r \leq k} C_{sr} > (1 - k^{-1})^2 + \delta + o_P(1)$. As a result, we conclude that there exists a permutation $\pi$ of $\{2,3,\cdots, k\}$, such that $C_{2\pi(2)} + C_{3\pi(3)} + \cdots + C_{k \pi(k)} \geq k^{-2}(k - 1) + (r - 1)^{-1}\delta + o_P(1)$. Therefore, $\Overlap_n \geq C_{11} + C_{2\pi(2)} + C_{3\pi(3)} + \cdots + C_{k \pi(k)} > k^{-1} + k(k - 1)^{-1}\delta + o_P(1)$. For $s,r \in [k]$, we define the set $S_{sr}^{\delta} := \{C_{sr} > k^{-2} + \delta\}$. Since $\Overlap_{n_i} \toP k^{-1}$, using the above analysis we derive that $\lim_{i \rightarrow \infty}\P(S_{11}^{\delta}) = 0$. Indeed, we can repeat such analysis for all $s,r \in [k]$ and conclude that $\lim_{i \rightarrow \infty}\P(S_{sr}^{\delta}) = 0$, thus $C_{sr} \leq k^{-2} + \delta + o_P(1)$ along the subsequence $\{n_i\}_{i \in \NN_+}$. Since $\sum_{s,r \in [k]}C_{sr} = 1$ and $\delta$ is arbitrary, we deduce that $C_{sr} \toP k^{-2}$ along the subsequence $\{n_i\}_{i \in \NN_+}$. This further implies that $C_{sr} \toP k^{-2}$. 

However, according to \cite[Theorem 2]{banks2018information} and \cref{thm:upper-bound}, we see that under the conditions of this part 
 \begin{align*}
 	\liminf_{n,d \rightarrow \infty}\frac{1}{n^2}\sum_{i,j \in [n]}\E\left[ \mathbbm{1}\{\hat\bLambda_i = \hat\bLambda_j\} \mathbbm{1}\{\bLambda_i = \bLambda_j\} \right] = &\liminf_{n \rightarrow \infty}\frac{1}{n^2}\sum_{i,j \in [n]}\E\left[ \E[ \mathbbm{1}\{\bLambda_i = \bLambda_j\}\mid \bA] ^2 \right] > k^{-2}.
 \end{align*}
 Finally, we plug \cref{eq:14} into the formula above, which leads to $\liminf_{n,d \rightarrow \infty} \E[\sum_{s,k \in [r]} C_{sk}^2] > k^{-2}$. This is in contradiction with the previously established claim that $C_{sr} \toP k^{-2}$ for all $s, r \in [k]$, thus completing the proof of part (b).

\end{appendices}

\end{document}